\documentclass[11pt,a4paper,reqno,final]{article}
\usepackage{amsmath,amsfonts,amssymb}
\usepackage{amsthm}
\usepackage{cite}
\usepackage{comment}
\usepackage{graphicx}
\usepackage{stix}
\usepackage[margin=0.85in]{geometry}
\usepackage[english]{babel}
\usepackage{mathtools}
\usepackage{esint}
\usepackage{multirow}

\usepackage{authblk}
\usepackage[most]{tcolorbox}
\usepackage{amsmath,amssymb}
\usepackage[most]{tcolorbox}

\usepackage{float}
\usepackage{xcolor}
\usepackage{mathtools}

\newcommand{\mycomment}[1]{}

\newcommand{\dt}{\,\mathrm{d}t}
\usepackage{authblk}
\usepackage{tikz}
\newtheorem{thm}{Theorem}[section]
\newtheorem{lemma}[thm]{Lemma}
\newtheorem{rem}[thm]{Remark}

\newtheorem{example}{Example}[section]
\numberwithin{figure}{section}
\numberwithin{table}{section}
\numberwithin{equation}{section}
\usepackage{subcaption}
\newcommand\norm[1]{\lVert#1\rVert}

\usepackage[normalem]{ulem}
\newcounter{corr}
\definecolor{violet}{rgb}{0.580,0.,0.827}
\newcommand{\corr}[3]{\typeout{Warning : a correction remains in page \thepage}
  \stepcounter{corr}
	      {\color{blue}\ifmmode\text{\,\sout{\ensuremath{#1}}\,}\else\sout{#1}\fi}
              {\color{red}#2}
              {\color{violet} #3}
}

\title{\bf Hybrid-high order method in space and implicit schemes in time for  the biharmonic wave equation} 
\author{Raman Kumar\footnote{Corresponding author, Department of Mathematics, Indian Institute of Technology Bombay, Mumbai, Maharashtra
400076, India ({ramanthalor18@gmail.com}), currently working in IMAG, CNRS, University of Montpellier, Montpellier, France.}
\quad
Neela Nataraj\footnote{Department of Mathematics, Indian Institute of Technology Bombay, Mumbai, Maharashtra
400076, India ({neela@math.iitb.ac.in}).}\quad
Aamir Yousuf\footnote{IITB–Monash Research Academy, Indian Institute of Technology Bombay, Mumbai, Maharashtra 400076, India ({aamir72@iitb.ac.in}).}}
\begin{document}

\maketitle    
\begin{abstract}
This article presents the numerical analysis for the biharmonic wave equation with clamped boundary conditions employing two variants of the {hybrid high-order} method for the space discretization and two implicit time-stepping schemes for the time discretization. The Newmark scheme directly discretizes the second-order time derivative, while the Crank-Nicolson scheme discretizes a reformulated system where we introduce velocity as an independent variable to create coupled first-order equations. Optimal orders of convergence in space and time are achieved for both schemes. The numerical experiments validate the theoretical convergence rates and show the effectiveness of the proposed methods.  To the best of our knowledge, this is the {\it first} work in literature that addresses hybrid-high order method and implicit time schemes for the biharmonic wave equation. 
\end{abstract}
\vspace{.01in}

    
    \section{Introduction}
\subsection{Model Problem}
This paper discusses two variants of  Hybrid-High Order~(HHO) methods, popularly known as HHO-A and HHO-B methods, for space discretization and implicit schemes for time discretization of the  biharmonic wave equation that seeks $u(x,t)$ such that 
\begin{equation}
u_{tt} + \Delta^2 u = f(x,t), \quad (x,t) \in \Omega \times \left(0,T \right] \label{P1 strong_form}
\end{equation}
with initial and clamped boundary conditions
\begin{equation}
u(x,0)=u_0(x) \;\; \mbox{and}\;\; u_t(x,0)=u_1(x) \text{ in }\Omega;\quad    u=\frac{\partial u}{\partial n} =0 \text{ on }\partial \Omega \times \left(0,T \right].\label{P1 strong_icbc}
\end{equation}
Here, $\Omega$ is a bounded polytopal Lipschitz domain in $\mathbb{R}^{2}$ (resp. $\mathbb{R}^{3}$) with  boundary $\partial \Omega$, $\Delta^2 u:= \frac{\partial^4 u}{\partial
				x^4} + 2 \:\frac{\partial^4 u}{\partial x^2 \partial y^2}
			+\frac{\partial^4 u}{\partial y^4}$ (resp. $\frac{\partial^4 u}{\partial
				x^4} +\frac{\partial^4 u}{\partial y^4}+\frac{\partial^4 u}{\partial z^4}+ 2 \:\frac{\partial^4 u}{\partial x^2 \partial y^2}+ 2 \:\frac{\partial^4 u}{\partial x^2 \partial z^2}+ 2 \:\frac{\partial^4 u}{\partial y^2 \partial z^2}
			$) is the biharmonic operator, $T$ is the final time, and  $u_{t} $, $u_{tt} $ denote the first and second order  derivatives with respect to time, respectively.
            
%\subsection{ Weak formulation-Wellposedness and Regularity}
\medskip \noindent
The biharmonic wave equation has applications in various physical phenomena such as elasticity of thin plates with the dynamic load constraints and bending of plates. Moreover, this can be explored as a fundamental model for advanced Kirchhoff-type equations \cite{MR4403922,MR1422248,MR0953313}, which investigate wave transmission, standing waves, and vibration modes within plates.
%Elastic plates are essential components in many mechanical structures, playing a vital role in practical applications, especially in industrial construction and microchip production. As core elements in diverse mechanical systems, these thin plates are susceptible to cracks, either from material fatigue or intentional design features. Consequently, a thorough understanding of their behavior and resilience is of great importance.
%Consequently, the mechanical properties of these specific structures have been extensively researched in numerous studies\cite{,}}. 

\subsection{Literature overview}
Numerical methods and analysis for the biharmonic problem have been investigated thoroughly in literature,  see for example, { \cite{MR345432,MR2684360,MR2142191,MR2520159,MR4376276,MR4485999,MR3767813,MR4597463,MR4480625,MR2298696}.}
%Designing $C^{1}$ finite element methods (FEMs) for fourth-order equations, which demand the continuity of both basis functions and their first-order derivatives across closure of the domain and conforming to $H^{2}(\Omega),$ poses several challenges. To address these issues, various non-standard finite element methods (FEMs) have been developed, including the discontinuous Galerkin (DG) FEM \cite{MR2755946,DG-BI}, the Morley FEM \cite{MR4230429,MR3934690,MR2207619}, and the $C^{0}$ interior penalty (IP) method \cite{MR4481121,MR2298696}. An abstract framework for analyzing lowest-order finite element methods (FEMs) for the clamped plate biharmonic problem has recently been established \cite{MR4376276}. 
Using the backward Euler method in the time direction and lowest-order finite element spaces a unified analysis in 
for fourth-order nonlinear parabolic problems has been studied recently \cite{MR4908953} in the literature.  Numerical approaches are designed for the second-order hyperbolic equations, for more details see \cite{MR3438428,MR2272600,MR0945124,MR4456271,MR2942377,MR1241479,MR0349045,MR2052507}. Despite their physical relevance, not much work is available in the literature for biharmonic wave equation; possibly due to the complexities arising from the higher-order nature of the equation in space and time. In
\cite{MR4444402}, classical $C^{1}$-conforming FEMs were developed using Bogner-Fox-Schmit elements and the collocation technique in space and time directions, respectively. In \cite{MR2114953}, a mixed velocity-moment formulation is analyzed for the fourth-order Kirchhoff-Love dynamic plate equations. The study employs Lagrange finite elements for spatial discretization and both explicit and implicit central difference schemes for temporal discretization. For fourth-order vibration problems, error estimates using a discrete Galerkin in space and a second-order accurate scheme in the time direction have been discussed in \cite{MR0317559}. In \cite{MR3003381,he23}, mixed FEMs for fourth-order wave equations under various boundary conditions have been investigated. A unified analysis of the biharmonic wave problem based on lowest-order nonstandard finite elements with explicit leapfrog and implicit Newmark schemes is studied in \cite{amir} recently, {however, the mixed (in time) scheme analyzed herein, which of importance, is {\it not} discussed in \cite{amir}.} {\it The emergence of HHO methods for spatial discretization motivates the investigation of their coupling with time discretization schemes for vibration problems.} 

\medskip \noindent
% The HHO methods are independent of dimensional constraints in their construction, support general meshes, and accommodate polytopal meshes as well as nonmatching interfaces. They also provide a local formulation that utilizes equilibrated fluxes and are computationally efficient because of the static condenzation of cell unknowns. The HHO methods employ discrete unknowns within both the cells and faces. 
% The two fundamental components of the HHO methods are a {\it reconstruction operator}, which locally reconstructs a displacement field or its gradient from the local cell and face unknowns, and a {\it stabilization operator}, which weakly enforces consistency between the local face unknowns and the trace of the cell unknowns on each mesh face.
The Hybrid High-Order~(HHO) methods provide a dimension-independent 
framework with geometric flexibility, as they support general 
polytopal meshes and non-matching interfaces. Their efficiency arises from a fully local formulation that relies on fluxes and static condensation of unknown cells.
In HHO methods, discrete unknowns are placed both on cells and on faces. {The formulation is built upon two key operators: a \textit{reconstruction operator} that locally recovers displacement field or gradient from cell and face unknowns, and a \text{stabilization operator} that weakly enforces consistency between face unknowns and cell traces across local mesh interfaces.}
This design ensures {\it optimal convergence} while preserving computational 
efficiency and robustness on complex {computational domains, for more details see \cite{di2020hybrid}}.

\medskip \noindent
{For time-dependent models, spatial discretization using HHO methods has been discussed for Sobolev equations \cite{xie2022hybrid_1,xie2022hybrid}, Cahn-Hillard equations \cite{chave2016hybrid}, Fisher-Kolmogorov equation \cite{NNRK}, and wave equations \cite{ MR4412337,MR4254716}. In \cite{ern2023convergence}, the second-order wave equation is studied  using hybrid methods in space, namely, {HHO}, {HDG}, and {WG} with the leapfrog scheme for time discretization. {The paper  establishes that the fully discrete problem satisfies an algebraic energy balance (see \cite[Lemma 3.3]{ern2023convergence}).} Analogous energy balance arguments also hold for the biharmonic wave equation discussed in this article. The disadvantage of the explicit scheme in \cite{ern2023convergence} is that the management of face unknowns necessitates the use of a sparse matrix that couples all mesh faces and this reduces the computational efficiency typically expected from explicit time-stepping methods. Additionally, the scheme is conditionally stable and requires the triangulation to be quasi-uniform. 
{The second-order wave equation is also addressed in \cite{MR4412337,MR4254716}, where the HHO method is employed for spatial discretization and the Newmark scheme is used for time discretization. These papers propose and implement fully discrete schemes, including the Newmark scheme for the wave equation in its primal form, as well as implicit and explicit Runge-Kutta schemes for the mixed first-order system.}
%and \cite{MR4254716} discusses error estimates for the spatial semidiscrete scheme.}

\subsection{Contributions}
\begin{itemize}
\item We prove that both the Newmark and Crank–Nicolson time‑integration schemes {with HHO-A and HHO-B spatial discretizations for the biharmonic wave equations \eqref{P1 strong_form}-\eqref{P1 strong_icbc} are unconditionally stable, with no Courant–Friedrichs–Lewy (CFL) type restriction linking spatial and temporal discretizations.}
\item The approximation properties of the projection operator defined in \eqref{proj_op_1} in \(L^2\)-norm for the HHO-A method represent a novel contribution to the literature. This result is crucial for proving the \(L^2\) error estimates presented in this paper. 
\item The HHO scheme for the space discretization and {\it Newmark scheme} for the time discretization  is well-suited for the reduction of unknowns, leading to a global linear system at each time step that only couples the face unknowns. %Alternatively, the leapfrog scheme can be used to discretize the time derivative, yielding an explicit, second-order accurate scheme that is conditionally stable. } 
{Optimal order convergence for displacement  $u$ of $\mathcal{O}(h^{k+1})$ in space direction and $\mathcal{O}((\Delta t)^2)$ in time direction with respect to energy-norm  are derived for both HHO-A and HHO-B schemes.  Moreover,  optimal order convergence for displacement  $u$  of  $\mathcal{O}(h^{k+2})$ for $k=0$ and $\mathcal{O}(h^{k+3})$ for $k\ge1$ in space direction and $\mathcal{O}((\Delta t)^2)$ in time direction in $L^2$-norm are obtained for HHO-A scheme.}
\item The {HHO scheme} for spatial discretization and the {\it Crank-Nicolson scheme} for temporal discretization have been analyzed. The HHO-Crank-Nicolson method is well-suited for time discretization in mixed form for the biharmonic wave equation as a first-order system in time, leading to a global linear system at each time step that only couples the face unknowns. 
%An alternative for time discretization is the implicit backward Euler scheme, though this scheme converges only at a linear rate in time.}
The analysis demonstrates {optimal order convergence for displacement $u$ of $\mathcal{O}(h^{k+1})$ in space direction and $\mathcal{O}((\Delta t)^2)$ in time direction with respect to energy-norm  for both HHO-A and HHO-B schemes. Optimal order convergence for displacement  $u$ and velocity $u_t$ of  $\mathcal{O}(h^{k+2})$ for $k=0$ and $\mathcal{O}(h^{k+3})$ for $k\ge1$ in space direction and $\mathcal{O}((\Delta t)^2)$ in time direction in $L^2$-norm are established for HHO-A scheme.}

%However, the proof for the same result has yet to be explored for the HHO-B scheme, and the reason for the hindrance in extending the result to the HHO-B scheme has been discussed in Subsection~\ref{l2 app}.

\item The results of the numerical experiments are in strong agreement with the theoretical error bounds for both schemes, confirming the predicted convergence rates and demonstrating the  accuracy of the proposed formulations. 
\end{itemize}

\subsection{Notation}
Standard notation for Lebesgue and Sobolev spaces applies throughout the paper. 
% For $i:=(i_{1},\cdots,i_{d})$ with $i_{1},\cdots,i_{d} \ge 0$ and  $|i|:=i_{1}+\cdots+i_{d},$ 
% $v:=(v_{1},\cdots,v_{d})$, let $D^{i}v:=\frac{\partial^{|i|}v}{\partial v_{1}^{i_{1}}\partial v_{2}^{i_{2}}\cdots \partial v_{d}^{i_{d}}}.$ 
 Let $\nabla$ and $\nabla^{2}$ denote the gradient and Hessian operators, respectively. For any $D\subset\mathbb{R}^{d}~(d=2,3)$ and any integer $r\geq 0,$ define the Sobolev space $H^{r}(D):=\{v\in L^{2}(D):\; \nabla^{i}v\in L^{2}(D)\; \text{for}\; 0\leq |i|\leq r\}.$ For $v \in H^{r}(D)$,  the associated norm and semi-norm are denoted by 
$\displaystyle \norm{v}_{H^{r}(D)}=\big(\sum\limits_{|i|\leq r}\norm{\nabla^{i}v}^{2}_{L^{2}(D)}\big)^{\frac{1}{2}}$ and $\displaystyle |v|_{H^{r}(D)}=\big( \sum_{|i|=r}\norm{\nabla^{i}v}^{2}_{L^{2}(D)}\big)^{\frac{1}{2}},$ respectively.
{The $L^{2}$ inner product and norm defined on $\Omega$ are denoted by $(\cdot,\cdot)$ and $\norm{\cdot},$ respectively. 
%Let $X$ denote a normed space with norm $\norm{\bullet}_{X}$ and seminorm $|\bullet|_{X}$, respectively, and 
Let $(X,\|\bullet\|_X)$ be a Hilbert space and for $0\le a < b$, the Bochner space $L^p(a,b;X)$ consists of all strongly measurable functions $g:(a,b) \to X$ such that \cite{MR2597943} 
$$\norm{g}_{L^{p}(a,b; X)}:=\big(\int_{a}^{b}\norm{g(t)}_{L^{p}(X)}^{p}dt\big)^{1/p},\;\; 1\leq p < \infty\;\; \mbox{and}\;\; \norm{g}_{L^{\infty}(a,b;X)}:= \mbox{ess}\sup_{a\leq t\leq b}\norm{g(t)}_{X}.$$
For $m \in \mathbb{N}$, the {Sobolev-Bochner space} $H^m(a,b;X)$ consists of all functions $u \in L^2(0,T;X)$ whose weak derivatives up to order $m$ exist and belong to $L^2(0,T;X)$,  that is
$$H^m(a,b;X) := \left\{u \in L^2(a,b;X) : \frac{d^k u}{dt^k} \in L^2(a,b;X) \text{ for } k = 0, 1, \ldots, m\right\}.$$
We adopt ${L^{r}(0,T; X)}={L^{r}(X)}$ and  $H^m(0,T;X)=H^m(X)$ for notational simplicity.} Let $C^k([0,T];X)$ denote all $C^k$ functions $g :[0,T] \rightarrow X$ with $\norm {g}_{C^k([0,T];X)}=\underset{0 \le i \le k}{\sum} \underset{0 \le t\le T}{\max} {\norm{g^i(t)}} < \infty, \text{ where }g^i(t)=\frac{\partial^i g}{\partial t^i}.$ 
 The notation $a \lesssim b$ means $a \leq C b,$ where the generic constant $C>0$ is independent of mesh-parameter and time-discretization parameter.  {This paper frequently employs a Young's inequality given by 
 \begin{equation}\label{Young's}
ab \leq \frac{a^{2}}{2\epsilon}+ \frac{\epsilon b^{2}}{2} \quad \text{ for}~ a,b\geq 0 \; \text{ and } \epsilon >0.
 \end{equation}}
 % For  real numbers $a > 0$, $b > 0$, and $\epsilon > 0$, the  weighted {Young's} inequality reads $ab \le \frac{\epsilon}{2}a^2 + \frac{1}{2\epsilon}b^2.$}
\subsection{Layout of the paper}
The rest of the article is divided into the following parts: Section \ref{section2} discusses the preliminaries of the HHO method and some approximation properties. The main results are presented in Section \ref{Sect_3}. 
Section~\ref{subsec_l2_est}
establishes $L^2$ approximation properties of the HHO-A method that will be used to prove $L^2$ error estimates for HHO-A  Newmark and Crank-Nicolson fully discrete schemes.
%Section \ref{Sect_4} demonstrates some useful approximation properties. 
%Proof of the Theorem \ref{m1_H^2_bound} is described in Section \ref{Sect_5}.
The stability and error analysis for the Newmark scheme and Crank-Nicolson scheme are discussed in Sections \ref{sect_6} and \ref{Sec_cr_nic}, respectively. The numerical test settings are described and the results of the numerical simulations are provided in Section~\ref{num_sec}.
\section{Preliminaries}\label{section2}
The first subsection discusses the weak form of \eqref{P1 strong_form}-\eqref{P1 strong_icbc}. This is followed by settings for the HHO method in  Subsection \ref{sect_2}, a description of local and global operators in Subsection~\ref{subsec-pro}, discrete functional framework in Subsection~\ref{subsec-stab}, discrete norms and approximation properties in Subsections~\ref{subs_dis_nor} and \ref{Sect_4}, respectively. This is followed by a subsection that states the estimates employed in the sequel. 
\subsection{Weak formulation}
The weak formulation that corresponds to \eqref{P1 strong_form}-\eqref{P1 strong_icbc} seeks $u \in H_0^2(\Omega)$ such that for a.e. $t \in (0,T]$
\begin{align}
   \begin{split}  (u_{tt},v)+(\nabla^2 u,\nabla^2v)&= \left(f,v\right)  \text{ for all }v \in H_0^2(\Omega), \label{P1 weak_form}\\  
    u(0)&=u_0 \text{ and } u_t(0)=u_1.
     \end{split}
\end{align}
For $f \in L^2( L^2(\Omega))$, $u_0 \in H^2_0(\Omega)$, and $u_1 \in L^2(\Omega)$, there exists a unique $u \in C^0([0,T];H^2_0(\Omega)) \cap  C^1([0,T]; L^2(\Omega)) $ satisfying \eqref{P1 weak_form} \cite[ p. 93]{MR0350178}.
\subsection{HHO setting for biharmomic operator}\label{sect_2}
This subsection discusses the discrete setting for the biharmonic operator for the HHO-A and  HHO-B schemes.
%, focusing on three key HHO components: {\it local and global reduction operators}, {\it local reconstruction operator}, and {\it stabilizer}.
%\subsection{Mesh Partition}
The notations in this subsection are adopted from \cite{MR4485999}. For every $h>0$, {let $\left\{{\cal T}_{h}\right\}_{h}$ be a mesh family such that each mesh ${\cal T}_{h}$ consists} of a finite number of non-empty disjoint open polygonal cells $K$ with { boundary $\partial K$} that { comprises of}  planar faces and covers $\Omega$ exactly.
% For every $h>0$,  the mesh ${\cal T}_{h}$ is such that  $\bar{\Omega}=\cup_{K\in {\cal T}_{h}}\bar{K}$\footnote{change this}
%    and comprises of a finite number of non-empty disjoint open polygon cells $K$ each with a planar face. 
The presence of hanging nodes is possible in such meshes, and the mesh-size is defined by $h:=\max_{K\in {\cal T}_{h}}h_K$, where $h_{K}$ represents the diameter of cell $K$. %{The boundary ${\cal F}_K$ of any mesh cell $K\in {\cal T}_{h}$ is partitioned {into interior and boundary faces} denoted as ${\cal F}_K^{\text{i}}:=\overline{{{\cal F}_K \cap {\Omega}}}$ and ${\cal F}_K^{\text{b}}:= {\cal F}_K\cap \partial \Omega.$ We partition the mesh as ${\cal T}_{h}:={\cal T}_{h}^{\text{i}}\cup {\cal T}_{h}^{\text{b}}.$ } 
A closed subset $F$ of $\Omega$ is denoted as a mesh face if it is a subset of an affine hyperplane $H_{F}$ with positive $(d-1)$-dimensional Hausdorff measure, and if either of the following two conditions holds: (i). There exist $K_{1}(F)$ and $K_{2}(F)$ in ${\cal T}_{h}$ such that $F \subset \partial K_{1}(F)\cap \partial K_{2}(F)\cap H_{F}$. Then, $F$ is called as an internal face. (ii). There exists $K(F)\in {\cal T}_{h}$ such that $F \subset \partial K(F) \cap \partial \Omega\cap H_{F}$. In this case, $F$ is referred to as a boundary face.
 The collection of all faces of ${\cal T}_{h}$ is denoted by  ${\cal F}_{h}$ and ${\cal F}^{0}_{h}={\cal F}_{h}\setminus\partial\Omega$ denotes the set of all interior faces. 
 %We assume that the mesh ${\cal T}_{h}~(h>0)$ is admissible \cite[Definition 1]{MR3283758}, that is, ${\cal T}_{h}$ allows for a matching simplicial submesh ${\cal T}^{1}_{h}$ (where each cell and face in ${\cal T}^{1}_{h}$ is a subset of a cell and a face of ${\cal T}_{h}$, respectively). 
 The mesh family ${\cal T}_{h}$  is shape-regular, that is,  %ll cells and faces of ${\cal T}_{h}$ have diameters that are uniformly comparable to those of the corresponding cells and faces of ${\cal T}^{1}_{h}$. Now, for 
 for any $K\in {\cal T}_{h}$ and $F\in {\cal F}_{K}$, the diameter $h_{F}$ is comparable to $h_{K}$, 
%\begin{equation}
  $ 2\rho^{2} h_{K} \leq h_{F} \leq h_{K},$
%\end{equation}
where $\rho$ is the mesh regularity parameter. In addition, there exists an integer $M_{\partial}$ depending on $\rho$ and $d$ such that 
 $\max_{K\in {\cal T}_{h}} \mbox{card}({\cal F}_{K})\leq M_{\partial}.$ {For any partition ${\cal T}_{h}$ of $\Omega$,  let  $H^{r}({\cal T}_{h})=  \Pi_{K\in {\cal T}_{h}} H^{r}(K)$. % and ${\cal P}_{r}(K),$ is the space of polynomials of degree at most $r\geq 0$ in each $K.
 Let $\mathcal{P}_r(K)$ denotes the polynomial space of degree $r \geq 0$ on a cell $K\in {\cal T}_{h}$ and $\mathcal{P}_r(\mathcal{F}_K) = \Pi_{F \in \mathcal{F}_K} \mathcal{P}_r(F)$ represents the face polynomial space. Moreover, we denote $\mathcal{P}_{-1}(\mathcal{F}_{K})=\{0\}.$ 
For any cell $X:=K\in {\cal T}_{h}$ or any face $X:=F$ of $K$, denote the $L^2$ inner product (resp. norm)  by  $(\cdot,\cdot)_{X}$ (resp. $\norm{\cdot}_{X}$).}

\begin{table}[H]
\centering
 \begin{tabular}{|l||p{6cm}|p{6cm}|}
    \hline
    \textbf{Properties} & \textbf{Local HHO Space} & \textbf{Global HHO Space} \\\hline
    Definition & ${\widehat{V}_{K}} := \mathcal{P}_{k+2}(K) \times \mathcal{P}_{r}(\mathcal{F}_{K}) \times \mathcal{P}_{k}(\mathcal{F}_{K})$ for $K\in\mathcal{T}_h$  with $r=k+1$ (resp. $k+2$) for HHO-A (resp. HHO-B) & $\widehat{V}_{h} := \mathcal{P}_{k+2}(\mathcal{T}_{h}) \times \mathcal{P}_{r}(\mathcal{F}_{h}) \times \mathcal{P}_{k}(\mathcal{F}_{h})$ with $r=k+1$ (resp. $k+2$) for HHO-A (resp. HHO-B)\\
    \hline 
    Generic Element & $\widehat{v}_{K} :=(v_K, v_{\mathcal{F}_K}, \zeta_{\mathcal{F}_K})$ & $\widehat{v}_{h} := (v_{h}, v_{\mathcal{F}_{h}}, \zeta_{\mathcal{F}_{h}})$ \\
    \hline
    Components & $v_K \in \mathcal{P}_{k+2}(K)$, $v_{\mathcal{F}_K} \in \mathcal{P}_{r}(\mathcal{F}_K)$, $\zeta_{\mathcal{F}_K} \in \mathcal{P}_k(\mathcal{F}_K)$ & $v_{h} = (v_{K})_{K\in \mathcal{T}_{h}}$, $v_{\mathcal{F}_h} = (v_{F})_{F\in \mathcal{F}_{h}}$, $\zeta_{\mathcal{F}_{h}} = (\zeta_{F})_{F\in \mathcal{F}_{h}}$ %\textcolor{blue}{such that $v_{\mathcal{F}_h}|_{F} = v_{F}$, $\zeta_{\mathcal{F}_h}|_{F} = \zeta_{F}$} 
    \\
    \hline
    Component-wise Meaning & Solution in cell $K$, trace on $\mathcal{F}_K$, and its normal derivative on the  cell boundary along the direction of the outward unit normal ${\bf n}_K$ & Solution in all cells, traces on all faces,  and its the normal derivative in the direction of the unit normal vector ${\bf n}_F$ orienting $F$ with single-valued interface unknowns \\
    \hline
    % Face Relations & {\color{red} \centering ---} & $v_{\mathcal{F}_K}|_{F} = v_{F}$, $\zeta_{\mathcal{F}_K}|_{F} = (\mathbf{n}_{K} \cdot \mathbf{n}_{F})\zeta_{F}$ \\
    % \hline
\end{tabular}
\caption{Definitions of local and global HHO Spaces.}
\label{HHO_Table}
\end{table}

%\medskip \noindent
% For any $K\in {\cal T}_{h},$ the {\it local HHO space} reads
% \begin{equation}
% V_{K}=\{\widehat{v}_{K}:=(v_{K},(v_{F})_{F\in {\cal F}_{K}},(\zeta_{F})_{F\in {\cal F}_{K}}): v_{K}\in {\cal P}_{k+2}(K),\; v_{F}\in {\cal P}_{r}(F),\;\zeta_{F}\in {\cal P}_{k}(F),\;\forall F\in {\cal F}_{K}\},\nonumber
% \end{equation}
% where ${\cal P}_{s}(X)~(X=K,F)$ denotes the polynomial space of degree $s\ge 0$ defined on $X$,
% $ r:= \begin{cases} 
% k+1 \text{ for HHO A}, \\
% k+2 \text{ for HHO B},
% \end{cases}$
% $v_{K},$ $v_{F},$ and $\zeta_{F}$ are the interior of element $K,$ trace and normal derivative at the face $F$, respectively.
% \medskip\noindent
% The {\it global HHO space} with the single-valued interface reads  
%  \begin{align*}
%  &V_{h}=\big\{\widehat{v}_h:=((v_{K})_{K\in {\cal T}_{h}},(v_{F})_{F\in {\cal F}_{h}},(\zeta_{F})_{F\in {\cal F}_{h}}):\nonumber\\&~~~~~~~~~~~v_{K}\in {\cal P}_{k+2}(K)\; \text{ for all } K\in{\cal T}_{h}, \; v_{F}\in {\cal P}_{r}(F),\;\zeta_{F}\in {\cal P}_{k}(F) \;\text{ for all } F\in {\cal F}_{h}\big\}.\nonumber
%  \end{align*}
%  The restriction of a generic element $\widehat{v}_h \in V_{h}$ to $K\in {\cal T}_{h}$ is denoted by $\widehat{v}_{K}=(v_{K},(v_{F})_{F\in {\cal F}_{K}},(\zeta_{F})_{F\in {\cal F}_{K}}) \in V_{K}.$ Moreover, denote $v_{h}=(v_{K})_{K\in {\cal T}_{h}}.$

\medskip
\noindent
For a $K \in {\cal T}_h$, the local (resp. global) HHO space ${\widehat{V}_{K}}$ (resp. ${\widehat{V}_h})$ and the details of discrete unknowns are presented in Table~\ref{HHO_Table}.  Impose the boundary conditions in $\widehat{V}_{h}$ to define
$ \displaystyle	\widehat{V}^{0}_{h}\coloneqq \{\widehat{v}_{h}\in \widehat{V}_{h}:v_{{\cal F}_{h}}|_{F}=\zeta_{{\cal F}_{h}}|_{F}=0\;\text{for all } F\in {\cal F}_{h}\cap\partial \Omega\}.$
\subsection{Local and global operators}\label{subsec-pro}
This subsection is dedicated to several useful local and global projection operators and the local reconstruction operator for HHO methods.

\medskip\noindent
{\it{$L^2$-projection operator.}} For any integer $s\ge 0$ and given $X=K \in {\cal T}_{h}$ or $F\in{\cal F}_{h}$ or ${\cal F}_{K}$ (the boundary of $K\in {\cal T}_{h}$),
 the $L^2$-projection operator $\Pi^{s}_{X}: L^{2}(X)\rightarrow {\cal P}_{s}(X)$ is defined as follows: For any $v\in L^{2}(X),$
\begin{eqnarray}
\left(\Pi^{s}_{X}(v),w\right)_{X}:=\left(v,w\right)_{X}\;\; \text{ for all }\; w\in {\cal P}_{s}(X).\label{pi-pro}
\end{eqnarray}
\medskip\noindent{\it{Interpolation operator \cite[eqn. 2.12]{MR4485999}.}} For HHO-A (in two dimensions), for all $F \in {\mathcal F}_h$,  let $J_{F}^{k+1}:H^{1}(F)\rightarrow {\cal P}_{k+1}(F)$ be the {\it interpolation operator} that satisfies: %generated by affine mappings. Moreover, $J_{F}^{k+1}$
%has the following properties defined as: 
for any $v\in H^{1}(F)$,
\begin{equation}\label{eq:def.JF}
(v-J_{F}^{k+1}(v),{w})_{F}=0\;\text{ for all } \;w \in {\cal P}_{k-1}(F) \quad\mbox{and}\quad  (\partial_{\bf t}(v-J_{F}^{k+1}(v)),w)_{F}=0\;\text{ for all }\;w \in {\cal P}_{k}(F),
\end{equation}
where the tangential derivative $\partial_{\bf t}(\bullet)$ is understood to act facewise. Next, we write \eqref{eq:def.JF} on the whole boundary of the mesh cell $K$. To end this, we define $J_{\mathcal{F}_{K}}^{k+1}:H^{1}(\mathcal{F}_{K})\to {\cal P}_{k+1}(\mathcal{F}_{K})$ such that $J_{\mathcal{F}_{K}}^{k+1}(v)= J_{F_{K}}^{k+1}(v|_{F})$ for all $v\in H^{1}(\mathcal{F}_{K}):=\left\{v\in L^{2}(\mathcal{F}_{K}): v|_{F}\in H^{1}(F) \text{ for all } F\in \mathcal{F}_{K}\right\}$; see \cite[eqn. 2.14]{MR4485999}.
%In addition,
% $J_{\mathcal{F}_{K}}|_{F}:= J_{F}$ for all $F\in \mathcal{F}_{K}.$

% let $J_{\mathcal{F}_{K}}^{k+1}:H^{1}(\mathcal{F}_{K})\rightarrow {\cal P}_{k+1}(\mathcal{F}_{K})$ be the {\it interpolation operator}  \cite[eqn. 2.14]{MR4485999} %generated by affine mappings. Moreover, $J_{F}^{k+1}$
% %has the following properties defined as: 
% with $v\in H^{1}(\mathcal{F}_{K})$,
% \begin{equation}
% (v-J_{\mathcal{F}_{K}}^{k+1}(v),\widetilde{\theta})_{\mathcal{F}_{K}}=0\;\text{ for all } \;\widetilde{\theta}\in {\cal P}_{k-1}(\mathcal{F}_{K}) \quad\mbox{and}\quad  (\partial_{\bf t}(v-J_{\mathcal{F}_{K}}^{k+1}(v)),\widetilde{\zeta})_{\mathcal{F}_{K}}=0\;\text{ for all }\;\widetilde{\zeta}\in {\cal P}_{k}(\mathcal{F}_{K}),
% \end{equation}
% {where the tangential derivative $\partial_{\bf t}(\bullet)$ is understood to act facewise}. In addition,
% $J_{\mathcal{F}_{K}}|_{F}:= J_{F}$ for all $F\in \mathcal{F}_{K}.$

%$\partial_{\bf t}v:={\bf t}_{F}\cdot \nabla v$ and ${\bf t}_{F}$ denotes the unit tangential vector along the face $F$.}

\medskip\noindent
{\it{Local reduction operator.}} For given $K\in {\cal T}_{h}$ and for all $v\in H^{2}(K)$, the {\it local reduction operator} $\widehat{I}^{k}_{K}:H^{2}(K)\rightarrow \widehat{V}_{K}$ reads: {For all $v\in H^{2}(K)$,}
\begin{equation}
 \widehat{I}^{k}_{K}(v):=(\Pi^{k+2}_{K}(v),{\cal Q}_{{\cal F}_{K}}(v),\Pi^{k}_{{\cal F}_{K}}(\partial_{\bf n} v)),\nonumber  
\end{equation}
where $\partial_{\bf n} v:={\bf n}_{K}\cdot\nabla v$ and ${\cal Q}_{\mathcal{F}_{K}}:=\begin{cases}
   J^{k+1}_{\mathcal{F}_{K}}  \text{ for HHO-A},\\
   \Pi^{k+2}_{\mathcal{F}_{K}} \text{ for HHO-B}.
\end{cases}$\\
%=J^{k+1}_{F}$(resp.  ${\cal Q}_{F}=\Pi^{k+1}_{F}$) for $r=k+1$(resp. $r=k+2$).
\medskip \noindent
{\it{Global reduction operator.}} The {\it global reduction operator} $\widehat{I}^{k}_{h}:H^{2}(\Omega)\rightarrow \widehat{V}_{h}$ reads: For all $v\in H^{2}(\Omega),$
\begin{equation}
 \widehat{I}^{k}_{h}(v):=(\Pi^{k+2}_{h}(v),{\cal Q}_{{\cal F}_{h}}(v),\Pi^{k}_{{\cal F}_{h}}(\partial_{\bf n} v)),\label{inter}
\end{equation}
where $(\Pi^{k+2}_{h}(v))|_{K}:=\Pi^{k+2}_{K}(v),$ $({\cal Q}_{{\cal F}_{h}}(v))|_{F}:={\cal Q}_{F}(v)=\begin{cases}
   J^{k+1}_{K}(v)  \text{ for HHO-A},\\
   \Pi^{k+2}_{K}(v) \text{ for HHO-B},
\end{cases}$  and $\Pi^{k}_{{\cal F}_{h}}(\partial_{\bf n} v)|_{F}:=\Pi^{k}_{F}({\bf n}_{F}\cdot\nabla v)$.\\
{\it{Local reconstruction operator.}} The {\it local reconstruction} operator ${\cal R}_{K}:\widehat{V}_{K}\rightarrow {\cal P}_{k+2}(K)$ is determined by the following equations: For given $\widehat{v}_{K}=(v_{K},v_{{\cal F}_{K}},\zeta_{{\cal F}_{K}})\in \widehat{V}_{K}$ and $K\in {\cal T}_{h},$
\begin{align}
(\nabla^{2}{\cal R}_{K}(\widehat{v}_{K}),\nabla^{2} w)_{K}&:=	(\nabla^{2} v_{K},\nabla^{2} w)_{K}+\sum_{F\in {\cal F}_{K}}\left(v_{K}-v_{F},\partial_{{\bf n}}\Delta w\right)_{F}-\sum_{F\in {\cal F}_{K}}(\partial_{{\bf n}}v_{K}-\zeta_{F},\partial_{{\bf n}{\bf n}} w)_{F}\nonumber\\&\qquad-\sum_{F\in {\cal F}_{K}}(\partial_{{ {\bf t}}}(v_{K}-v_{F}),\partial_{{\bf n}{{\bf t}}}w)_{F}\;\text{ for all }\; w\in{\cal P}_{k+2}(K),\label{loc_recon_op_2}\\
({\cal R}_{K}(\widehat{v}_{K}),\phi)_{K}&:=(v_{K},\phi)_{K}\;\text{ for all }\; \phi\in{\cal P}_{1}(K).\nonumber
\end{align} 
{\it{Elliptic type operator.}} The operator ${\cal E}_{K}$  is defined as ${\cal E}_{K}:={\cal R}_{K}\circ \widehat{I}^{k}_{K}:H^{2}(K)\rightarrow {\cal P}_{k+2}(K)$. Note that ${\cal E}_{K}$ is $H^{2}$-elliptic projection if $r=k+1$. But, it is no longer $H^{2}$-elliptic projection for $r=k+2$, but it maintains the properties of the elliptic operator, {see\cite[Equations~(4.2) and (5.2)]{MR4485999}}. More precisely, in the case of the HHO-A scheme, the elliptic projection ${\cal E}_{K}$ holds: For all $w\in H^{2}(K),$
\begin{equation}\label{eq:elli.proj.use_1}
(\nabla^{2}(w-{\cal E}_{K}(w)),\nabla^{2}\phi)_{K}=0 \quad \text{ for all } \phi \in {\cal P}_{k+2}(K). 
\end{equation}

\subsection{Discrete functional framework for HHO method}\label{subsec-stab}
{This subsection presents the bilinear forms in the HHO methods. In addition, we also introduce the intermediate projection operator $\widehat{E}_{h}$, which will be used in the subsequent analysis.}

\medskip\noindent The {\it local stabilization bilinear form } ${\cal S}_{K}:\widehat{V}_{K}\times \widehat{V}_{K}\rightarrow \mathbb{R}$ has two parts; that is, $S_K:={\cal S}_{K}^{1}+{\cal S}_{K}^{2}$, where \begin{align}
{\cal S}_{K}^{1}(\widehat{v}_{K},\widehat{w}_{K}):= \begin{cases}\sum_{F\in {\cal F}_{K}} h^{-3}_{K}(J^{k+1}_{F}(v_{F}-v_{K}),J^{k+1}_{F}(w_{F}-w_{K}))_{F}& \text{ for HHO A},\nonumber\\
\sum_{F\in {\cal F}_{K}} h^{-3}_{K}(v_{F}-v_{K},w_{F}-w_{K})_{F} &\text{ for HHO B},
%\nonumber\\&\qquad+\sum_{F\in {\cal F}_{K}} h^{-1}_{K}({\Pi }^{k}_{F}(\zeta_{F}-\partial_{{\bf n}}v_{K}),{\Pi}^{k}_{F}(\chi_{F}-\partial_{{\bf n}}w_{K}))_{F}\;\;\mbox{for}\;r=k+1,\label{stab_op1}
\end{cases} 
\end{align} and
%\begin{align}
${\cal S}_{K}^2(\widehat{v}_{K},\widehat{w}_{K}):=%\sum_{F\in {\cal F}_{K}} h^{-3}_{K}(v_{F}-v_{K},w_{F}-w_{K})_{F}\nonumber\\&\qquad+
\sum_{F\in {\cal F}_{K}} h^{-1}_{K}({\Pi }^{k}_{F}(\zeta_{F}-\partial_{{\bf n}}v_{K}),{\Pi}^{k}_{F}(\chi_{F}-\partial_{{\bf n}}w_{K}))_{F}.$  

\medskip\noindent
%{\textit{Combined reconstruction-stabilization bilinear form:}} 
For all $\widehat{v}_{h},\widehat{w}_{h} \in \widehat{V}_{h},$ the {\it bilinear form} $a_h:\widehat{V}_{h}\times \widehat{V}_{h}\rightarrow \mathbb{R}$ corresponding to $\nabla^2$ is defined by
\begin{align*}
a_{h}(\widehat{v}_{h},\widehat{w}_{h}):= \sum_{K\in{\cal T}_{h}}a_{K}(\widehat{v}_{K},\widehat{w}_{K}):= \sum_{K\in{\cal T}_{h}}(\nabla^{2}{\cal R}_{K}(\widehat{v}_{K}),\nabla^{2}{\cal R}_{K}(\widehat{w}_{K}))_{K}+{\cal S}_{h}(\widehat{v}_{h},\widehat{w}_{h}),\nonumber
\end{align*}
where ${\cal S}_{h}(\bullet,\bullet)$ denotes the \emph{global stabilization term} obtained by assembling the corresponding \emph{local stabilization forms}
$${\cal S}_{h}(\widehat{v}_{h},\widehat{w}_{h}):=\sum_{K\in{\cal T}_{h}}{\cal S}_{h}(\widehat{v}_{h},\widehat{w}_{h}).$$
\medskip\noindent
{\textit{Projection operator $\widehat{E}_{h}$.}}
Let ${\cal Y}=\{w\in H^{2}_{0}(\Omega): \Delta^{2}w\in L^{2}(\Omega)\}$. Motivated by \cite{ern2023convergence}, for a fixed time $t\in (0,T],$ define the projection operator $\widehat{E}_{h}: {\cal Y}\rightarrow\widehat{V}_{h}^{0}$ as
%\footnote{Y is subset of H2, do we need intersection here?}\\
\begin{equation}
a_{h}(\widehat{E}_{h}(w), \widehat{v}_{h})=( \Delta^{2}w,v_{h})\;\; \text{ for all } \widehat{v}_{h}\in \widehat{V}_{h}^{0}.\label{proj_op_1}
\end{equation} 
The well-posedness of the operator  $\widehat{E}_{h}%:=(E_{h},E_{{\cal F}_{h}},\widetilde{E}_{{\cal F}_{h}})
$ follows from boundedness and coercivity of the bilinear form $a_{h}(\bullet,\bullet)$~(see Lemma \ref{bdd_lmma_a_T}) and Lax-Milgram lemma.}
\subsection{Discrete norms}\label{subs_dis_nor}
\medskip\noindent
%{\textit{
%Local $H^2$-like seminorm:}} 
For all $\widehat{v}_{K}=(v_{K},v_{{\cal F}_{K}},\zeta_{{\cal F}_{K}})\in \widehat{V}_{K}$ %the  $H^2$-like seminorm on $\widehat{V}_{K}$ is defined by 
define 
\begin{equation}
|\widehat{v}_{K}|^{2}_{\widehat{V}_{K}}:=\norm{\nabla^{2}v_{K}}^{2}_{K}+\sum_{F\in {\cal F}_{K}}h^{-3}_{K}\norm{v_{F}-v_{K}}^{2}_{F}+\sum_{F\in {\cal F}_{K}}h_{K}^{-1}\norm{\zeta_{F}-\partial_{n}v_{K}}_{F}^{2}.\label{Seminorm_def_1}
\end{equation}
%{\textit{$H^2$-like norm:}} %Let $|v_{h}|_{\widehat{V}_{h}}^{2}=\sum_{K\in {\cal T}_{h}}|\widehat{v}_{K}|^{2}_{\widehat{V}_{K}}$
For all $\widehat{v}_{h}\in \widehat{V}_{h}^{0},$ define  $H^2$-like norm
%the $H^2$-like norm on the HHO space $\widehat{V}_{h}^{0}$ reads
%\begin{equation}
$\displaystyle \norm{\widehat{v}_{h}}_{\widehat{V}_{h}}^{2}:=\sum_{K\in {\cal T}_{h}}|\widehat{v}_{K}|^{2}_{\widehat{V}_{K}}$.
%\;\; \text{ for all }\widehat{v}_{h}\in \widehat{V}^{0}_{h}.$
%\end{equation}
 %{\textit{Energy norm:}} 
 For all $\widehat{v}_{h}\in \widehat{V}_{h}^0,$ the energy-norm is defined as
\begin{equation}
\norm{\widehat{v}_{h}}_{a,h}^{2}
%:=\sum_{K\in {\cal T}_{h}}a_{K}(\widehat{v}_{K},\widehat{v}_{K})
:=a_{h}(\widehat{v}_{h},\widehat{v}_{h}).\label{norm}
\end{equation}
%The norms $\norm{\cdot}_{\widehat{V}_{h}}$ and $\norm{\cdot}_{a,h}$ are equivalent in $\widehat{V}_{h}^{0},$ see Lemma \ref{bdd_lmma_a_T}.
%{\textit{Intermediate norm:}} 
For function $w\in H^{2+r}(K),\; r>\frac{3}{2}$ and  $K\in {\cal T}_{h}$, we also introduce the semi-norm $\norm{\bullet}_{\#,K}$ by
\begin{equation}
\norm{w}_{\#,K}^{2}:=\norm{\nabla^{2}w}^{2}_{K}+\sum_{F\in {\cal F}_{K}}(h^{3}_{K}\norm{\partial_{\bf n}\Delta w}^{2}_{F}+h_{K}\norm{\partial_{\bf n\bf n}w}^{2}_{F}+h_{K}\norm{\partial_{\bf n \bf t} w}^{2}_{F}).\label{m1_H_2_like_nm_def}
\end{equation}
\begin{lemma}\label{eq:lem.equiv}\rm{(}{norm equivalence and boundedness}\cite[Lemma $4.1$]{MR4485999} \rm{)}\label{bdd_lmma_a_T} For all $\widehat{v}_{K}\in \widehat{V}_{K}$ and $K\in {\cal T}_{h}$, it holds
$
|\widehat{v}_{K}|^{2}_{\widehat{V}_{K}}\lesssim  a_{K}(\widehat{v}_{K},\widehat{v}_{K}) \lesssim|\widehat{v}_{K}|^{2}_{\widehat{V}_{K}}$
where the constant absorbed in "$\lesssim$" depends only on the mesh shape-regularity, the polynomial degree $k$, and the space dimension $d.$ Consequently, 
$|\widehat{v}_{h}|^{2}_{\widehat{V}_{h}}\lesssim a_{h}(\widehat{v}_{h},\widehat{v}_{h})\lesssim |\widehat{v}_{h}|^{2}_{\widehat{V}_{h}}$ for all $\widehat{v}_{h} \in \widehat{V}_{h}$ and 
the norms $\norm{\widehat{v}_{h}}_{a,h}:=\big( a_{h}(\widehat{v}_{h},\widehat{v}_{h}) \big)^{1/2}$
and $\norm{\widehat{v}_{h}}_{\widehat{V}_{h}}$ are equivalent in $\widehat{V}_{h}^{0}$.  Moreover, there exists  $ \alpha>0 $ such that $a_{h}(\widehat{v}_{h},\widehat{w}_{h}) \le \alpha\norm{\widehat{v}_{h}}_{a,h}\norm{\widehat{w}_{h}}_{a,h}$ for all $\widehat{v}_{h},\widehat{w}_{h} \in \widehat{V}_{h}^{0}$.
 \end{lemma}
\subsection{Approximation properties}\label{Sect_4}
This subsection states some well-known estimates, approximation properties, and an integration by parts formula that are frequently used in the proofs of the main results.

%{Integration by parts formula}
\noindent For any $K\in {\cal T}_{h}$ and for sufficiently smooth functions  $w$ and $v,$ {\it integration by parts} \cite[Section 2.1]{MR4485999} shows 
\begin{equation}
(\Delta^{2}w,v)_{K}
=(\nabla^{2}w,\nabla^{2}v)_{K}+\sum_{F\in {\cal F}_{K}}\big((\partial_{\bf n}\Delta w,v)_{F}-(\partial_{\bf n n} w,\partial_{\bf n}v)_{F}-(\partial_{\bf n t}w,\partial_{\bf t}v)_{F}\big).\label{int_part_1}
\end{equation}
\begin{lemma}\cite[Lemmas $4.3$-$4.4$]{MR4485999}\label{bd_elli_proj} For any $w\in H^{2+r}(\Omega)$ with $r>3/2,$ it holds 
 \begin{align}
{\cal S}_{K}(\widehat{I}_{h}^{k}(w),\widehat{I}_{h}^{k}(w))^{\frac{1}{2}}&\lesssim \norm{w-\Pi^{k+2}_{K}(w)}_{\#,K} ,\label{eq:stab.bound}\\
\norm{w-{\cal E}_{K}(w)}_{\#,K}&\lesssim \norm{w-\Pi^{k+2}_{K}(w)}_{\#,K}\label{eq:w.EK.bound}.  
 \end{align}
\end{lemma}
\begin{thm}\label{lm1_in_dis_tr_11}
Let ${\cal T}_{h}$ be a shape-regular mesh sequence and let the polynomial degree $k\geq 0$. For each $K\in{\cal T}_{h}$ and $F\in {\cal F}_{K},$ the following estimates hold.

\medskip
\noindent
(a)  {\rm (discrete trace,  discrete inverse inequalities, Poincar\'e inequality)\cite[Lemmas $2.2$ and $(2.8)$]{MR4485999} }. For all  $w_{h}\in {\cal P}_{k}(K)$, 
$\norm{w_{h}}_{F}\lesssim h^{-\frac{1}{2}}_{K}\norm{w_{h}}_{K},\;
 \norm{\nabla w_{h}}_{K}\lesssim h^{-1}_{K}\norm{w_{h}}_{K},\;\;
 \text{and } \norm{\partial_{\bf t} w_{h}}_{F}\lesssim h^{-1}_{K}\norm{w_{h}}_{F}.$ Let $H^{2}(K)^{\perp}:=\big\{w\in H^{2}(K): \big(w,\zeta\big)_{K}=0 \;\;\text{ for all } \zeta\in {\cal P}_{1}(K)\big\}.$ Then for all $ v\in H^{2}(K)^{\perp},$  $h^{-2}_{K}\norm{v}_{K}+h^{-1}_{K}\norm{\nabla v}_{K}\lesssim \norm{\nabla^{2}v}_{K},$
   where the constants in "$\lesssim$" in all the inequalities depend on the mesh shape-regularity, the polynomial degree $k$, and the space dimension $d$.
%\medskip
%\noindent
%(b) {\rm (Poincar\'e inequality on cell)}. For all $ v\in H^{2}(K)^{\perp},$ where $H^{2}(K)^{\perp}:=\big\{w\in H^{2}(K): \big(w,\zeta\big)_{K}=0 \;\;\text{ for all } \zeta\in {\cal P}_{1}(K)\big\},$
%where the constants in $\lesssim$ in the above inequality depend on the mesh shape-regularity and the space dimension $d$.

\noindent
(b) {\rm (trace inequality)\cite[Lemma $2.2$]{MR4387067}}. For all  $v\in H^{s}(K),\;s\in(\frac{1}{2},1],$
we have $\norm{v}_{F}\lesssim \big(h^{-\frac{1}{2}}_{K}\norm{v}_{K}+h^{s-\frac{1}{2}}_{K}|v|_{H^{s}(K)}\big),$
where the constants in "$\lesssim$" depend on the mesh shape-regularity and the space dimension $d$.\\

\noindent
(c) {\rm (discrete energy estimate \cite[Theorem $4.6$]{MR4485999})}. For any  $w\in H^{2+r}(\Omega) \text{ with }r>3/2,$ it holds
\begin{align*}
&\norm{\widehat{I}^{k}_{h}(w)-\widehat{E}_{h}(w)}_{a,h}
\lesssim \big(\sum_{K \in {\cal T}_h}\norm{w-\Pi_K^{k+2}w}^2_{\#,K}\big)^{1/2}.
\end{align*}
Moreover, letting $\beta:=\min(r-1,1)\in (\frac{1}{2},1]$, we have
\begin{align*}
 \big(\sum_{K \in {\cal T}_h}\norm{w-\Pi_K^{k+2}w}^2_{\#,K}\big)^{1/2} \lesssim \begin{cases}h\big(|{w}|_{H^{3}({\cal T}_{h})}+h^{\beta}|{w}|_{H^{3+\beta}({\cal T}_{h})}\big)  \; \;\text{ for }k =0\;\;\text{if}\;\;w \in H^{3}({\cal T}_{h})\cap{H^{3+\beta}({\cal T}_{h})} ,
\\ h^{k+1}|{w}|_{H^{k+3}({\cal T}_{h})} \; \;\text{ for }k \ge 1\;\;\text{if}\;\;w \in H^{k+3}({\cal T}_{h}).
\end{cases}
\end{align*}

\medskip
\noindent
(d) {\rm (an approximation property)\cite[Lemma $2.5$]{MR4387067}}. For all  $v\in H^{r}(K), \; r\in [0,s+1],\; m\in \{0,\dots, \lfloor r\rfloor\},\; \mbox{and}\; s\geq0$,
$|v-\Pi^{s}_{K}(v)|_{H^{m}(K)}\lesssim h^{r-m}_{K}|v|_{H^{r}(K)}  \text{ for all } v\in H^{r}(K),$
where constants in "$\lesssim$"  depend on   $s.$

%\medskip\noindent
%(d) {\rm (discrete Sobolev embedding)}. Let $1\leq q\leq p^{*}$ if $1\leq p <d~(d=2,3)$, where $p*=\begin{cases}
   % \frac{d p}{d-p} \; \text{if}\;p<d,\\
  %  +\infty \;\; \text{if}\; p\geq d
%\end{cases}$ and $1\leq q <+\infty $ if $p\geq d.$ Then, it holds $\norm{v_h}_{L^{q}(\Omega)} \lesssim \norm{\widehat{v}_{h}}_{1,p,h}$ for all $\widehat{v}_{h}\in V_{h}^{0},$ where $v_h=(v_K)_{K \in {\cal T}_h}$,  $\norm{\widehat{v}_{h}}_{1,p,h}=\big(\displaystyle\sum_{K\in {\cal K}_{h}}\norm{\widehat{v}_{K}}_{1,p,K}^{2}\big)^{\frac{1}{2}}$ with $\norm{\widehat{v}_{K}}_{1,p,K}=\big(\norm{D v_{K}}_{L^{p}(K)^{d}}^{2}+\displaystyle\sum_{F\in {\cal F}_{T}}h^{1-p}\norm{v_{F}-v_{K}}^{2}_{F}\big)^{\frac{1}{2}}.$ The constants in $\lesssim$ in the above inequality depend on the mesh shape-regularity, $p$, $q$, the polynomial degree $k$, and the space dimension $d.$

\noindent
(e)  {\rm (discrete Poincar\'e inequality)\cite[Proposition $5.4$]{MR3647954}}. For all  $\widehat{v}_{h}=(v_{h},v_{{\cal F}_{h}},\zeta_{{\cal F}_{h}})\in \widehat{V}^{0}_{h}$, it holds that 
%and $ 1\leq q < \infty$,
\begin{equation}
\norm{v_h}\le  C_P \norm{\widehat{v}_{h}}_{a,h}. \label{poinacre}
\end{equation}
The constant $C_P$  depends on the mesh shape-regularity, $k$, and the space dimension $d$.
\end{thm}

\medskip \noindent 
\subsection{Temporal discretization setting}
Let $0=t_{0}<t_{1}<t_{2}<\cdots<t_{N-1}<t_{N}=T$ be the partition of $[0,T]$ with a uniform step size $\Delta t=t_{i+1}-t_{i} \text{ for any }  i=0,1,\cdots,N-1.$ For any  function $\phi: \left[0,T\right]\rightarrow \Omega,$ define
 \begin{subequations}\label{temp_dis}
\begin{align}
& \phi^n  := \phi(x,t_n)= \phi(t_n), \quad \phi^{n+1/2}:= \frac{1}{2}\left(\phi^{n+1}+\phi^n \right),
\quad \phi^{n,1/4} :=\frac{1}{4} \left(\phi^{n+1}+2\phi^n+  \phi^{n-1}\right),\label{P3 phi_c1 }\\
& \bar{\partial}_t \phi^{n+1/2} :=\frac{\phi^{n+1}-\phi^{n}}{\Delta t} ,\quad 
{\bar{\partial}}^2_t \phi^n := \frac{\phi^{n+1}-2\phi^n+\phi^{n-1}}{\Delta t^2} ,\quad 
\delta_t \phi^n := \frac{\phi^{n+1}-\phi^{n-1}}{2\Delta t}. \label{P3 phi_c2}
\end{align}
 \end{subequations}
 {It should be noted that definitions \eqref{temp_dis} extend to an element in HHO space.  For instance, if 
$
\widehat{v}_{h} ^n = (v_{h}^n,v_{\mathcal{F}_{h}}^n,\zeta_{\mathcal{F}_{h}}^n)$ $\in \widehat{V}_h,$
then \(
\widehat{v}_{h}^{n+1/2} = (v_{h}^{n+1/2},v_{\mathcal{F}_{h}}^{n+1/2},\zeta_{\mathcal{F}_{h}}^{n+1/2}),
\) and so on.

\medskip\noindent 
For  sequences $\{g_n\}$ and $\{h_n\}$, the \textit{summation by parts formulae} given below holds:
\begin{subequations}
 \begin{align}
 &\sum_{n=1}^m g_n(h_n-h_{n-1})= g_mh_m-g_1h_0-\sum_{n=1}^{m-1} (g_{n+1}-g_{n})h_{n},\label{sum_by_parts}\\
 &\sum_{n=0}^m g_{n+1}(h_{n+1}-h_{n})= g_{m+1}h_{m+1}-g_0h_0-\sum_{n=1}^{m} (g_{n+1}-g_{n})h_{n}.\label{sum_by_parts2}
 \end{align}
\end{subequations}
 
\medskip \noindent
 % \corr{}{}{[RK: The question mark on labels are not used so either we will use them in the correction of the manuscript or we need to comment these labels ... for example this label of (2.14), this is not used by (2.14), we have used like Lemm2.4(c)]. To resolve quickly I included refcheck package. At the moment I haven't commented on anyone, in corrections we may need these..}
%\section{Proof of Theorem \ref{L2_bound}}
{
\begin{lemma}{\rm (truncation error bounds  \cite[Theorem~4.4]{MR3003381},\cite[Lemma~4.1]{deka}\label{Sec_lma_1})}
Let $\varphi \in C^2([0,T]; L^2(\Omega))$ and  in addition, if
%\begin{itemize}
    ($a$)  $\varphi_{ttt} \in L^\infty(0,t_1;L^2(\Omega))$, then  $\norm{2 (\Delta t)^{-1}(\bar{\partial}_t \varphi^{1/2}-\varphi_t^0)-\varphi_{tt}^{1/2}} \lesssim  \Delta t \norm{\varphi_{ttt}}_{L^\infty(0,t_1;L^2(\Omega))},$\newline 
    ($b$)  $\varphi \in H^3(L^2(\Omega))$, then for any $1\le m\le n$, with $1\le n \le N-1$, there holds 
$$ \displaystyle \Delta t \sum_{n=1}^{m}\norm{\bar{\partial}_t \varphi^{n+1/2}- \varphi_t^{n+1/2}}^2\lesssim ( \Delta t)^4 \norm{\varphi_{ttt}}^2_{L^2(L^2(\Omega))} ,$$ and 
    ($c$)  $\varphi \in H^4(L^2(\Omega))$, then for any $1\le m\le n$, there holds 
    $$\displaystyle \Delta t \sum_{n=1}^{m} 
\norm{\bar{\partial}^2_t \varphi^n -\varphi_{tt}^{n,1/4} }^2 \lesssim  ( \Delta t )^{4}\norm{\varphi_{tttt}}^2_{L^2(L^2(\Omega))}.$$
%\end{itemize}
 \end{lemma}
 \begin{lemma}[discrete Gronwall lemma {\cite[Lemma~4.1]{MR3003381}\label{Ch2-d-gronwall}}]
 Let $\{a_n\}$, $\{b_n\}$, and $\{c_n\}$ be three non-negative sequences, with $\{c_n\}$ monotone, that satisfy 
 %\begin{equation*}
  $\   a_m+b_m \le c_m + \mu \sum_{n=0}^{m-1} a_n, \quad \mu >0, \ a_0+b_0 \le c_0. $ Then for $m \ge 0,$ it holds that 
  $    a_m+b_m \le c_m e^{m \mu}.$
\end{lemma}
\begin{rem}\label{rem2.6}
    Throughout this article, Lemma~\ref{Ch2-d-gronwall} is applied when $c_m$ is uniformly bounded and 
    $\mu =2{\Delta t}/{T}$. Hence, the growth factor in the estimate
       $ a_m + b_m \leq c_m \, e^{m\mu}$
    is uniformly bounded follows  by 
      $e^{m\mu} = e^{{m\Delta t}/{T}} \leq e \approx 2.718$
(with $m\,\Delta t \leq T$ in last inequality) . Thus, the 
    resulting constant depends solely on Euler's number $e$, and in particular 
    is independent of the discretization parameters $h$ and $\Delta t$.
\end{rem}
 }
\section{Main results}\label{Sect_3}
\noindent
This section presents the main results of the article, with the detailed proofs presented in the subsequent two sections. 
We  establish a novel $L^2$ approximation property of the projection $\widehat{E}_h$ in Subsection~\ref{subsec;l2} which  estimates $\widehat{\Pi}_h - \widehat{E}_h$ in the $L^2$ norm and  is crucial to prove $L^2$-error bounds in this article.
The Newmark scheme is introduced and its stability and error estimates are presented in Subsection~\ref{nm_schm_sub_1}.  In Subsection~\ref{crn_sub}, we introduce the Crank-Nicolson scheme and provide its stability and error estimates. This section ends with a detailed comparison between the Newmark and Crank-Nicolson schemes in Subsection~\ref{subsec-comp}.
\subsection{A novel $L^2$ projection estimate for HHO-A scheme}\label{subsec;l2}
This section establishes a novel $L^2$ approximation property for 
the HHO-A scheme, which plays a crucial role in the subsequent error 
analysis. The key idea is to apply an Aubin--Nitsche-type duality 
argument to the 
fourth-order biharmonic problem with HHO-A  setting. In contrast to the standard energy-norm 
estimates presented in this article, the $L^2$ estimate 
requires additional regularity to gain 
 extra powers of the mesh size $h$ in convergence.

\medskip \noindent
Our goal is to derive a sharp bound on 
$\|\Pi_h^{k+2}(w) - E_h(w)\|$ for any $w\in H^{2+r}(\Omega)$ with $r>3/2$. Since $\Pi_h^{k+2}(w) - E_h(w) \in L^2(\Omega)$, the Lax--Milgram lemma 
guarantees the existence of a unique $\Psi \in H^2_0(\Omega)$ satisfying
\begin{equation}
    \Delta^2 \Psi = \Pi_h^{k+2}(w) - E_h(w) \;\text{ in }\; \Omega, \qquad 
    \Psi = \partial_{\mathbf{n}}\Psi = 0 \;\text{ on }\; \partial\Omega.
    \label{eq:dual}
\end{equation}
The derivation of an optimal $L^2$ bound hinges on the additional 
$H^4$-regularity of the solution $\Psi$ to this dual problem, which 
we now make precise. Assume that the solution $\Psi \in H^2_0(\Omega)$ 
of problem \eqref{eq:dual} belongs to $H^4(\Omega)$ and satisfies 
the stability bound
\begin{equation}
    \|\Psi\|_{H^4(\Omega)} \lesssim\, \|\Pi_h^{k+2}(w) - E_h(w)\|.
    \label{eq:reg_bound}
\end{equation}
%This condition is certainly satisfied for convex domains. 
Under assumption \eqref{eq:reg_bound}, we are in a position to state and prove  the main result of this subsection.
\begin{thm}[A novel space $L^2$ approximation property]\label{m1_H^2_bound}
Any $w\in H^{2+r}(\Omega)$ with $r>3/2$ satisfying \eqref{eq:reg_bound}  yields the following bounds for HHO-A scheme:
\begin{align*}
&\norm{\Pi_{h}^{k+2}(w)-E_{h}(w)} \nonumber\\&\;\; \lesssim \begin{cases}h^{2}\big(|w|_{H^{3}({\cal T}_{h})}+h^{\beta}|w|_{H^{3+\beta}({\cal T}_{h})}+\norm{ \Delta^{2}w}\big) \;\;\text{ with }\beta:=\min(r-1,1) \text{ for }  \; k=0, \text{ if } w\in H^{3+\beta}({\cal T}_{h}),
\\ h^{k+3}\big(|w|_{H^{k+3}({\cal T}_{h})}+\norm{ \Delta^{2}w}_{H^{k-1}({\cal T}_{h})}\big) \; \;\text{ for }k \ge 1,\text{ if } w\in H^{k+3}({\cal T}_{h}).
% \label{app_prop}
\end{cases}
\end{align*}
\end{thm}
\subsection{Newmark scheme}\label{nm_schm_sub_1}
In this subsection, we develop a Newmark time-implicit scheme for the  biharmonic wave problem. We state the result on stability in Theorem~\ref{P2 full_stabilty}.  The initial error bounds, energy-norm estimates, and $L^2-$norm  estimates are discussed in Lemma~\ref{P1 lemma_on_norm_initial}, Theorem~\ref{P1 implicit_Th2}, and Theorem~\ref{err_L2_new}, respectively. The proofs of these results are presented in  Section~\ref{sect_6}.

\medskip \noindent
Let $\widehat{u}_{h}^{n}$ denote the discrete approximation of continuous solution $u$ at time $t=t_n$ for $n=0,1,\cdots,N$. Let $\widehat{u}_{h}^{0}$ be the interpolation $\widehat{I}^{k}_{h}$ of the given initial displacement $u_0$. Then using $\widehat{u}_{h}^{0}$ and the given initial velocity $u_1$, we apply Crank-Nicolson time stepping scheme at $t=t_1$ to obtain $\widehat{u}_{h}^{1}$, that is,
\begin{eqnarray}\label{hwave_int_cond}
\begin{cases}
 &\widehat{u}_{h}^{0}:=\widehat{I}^{k}_{h}(u_{0}),\\
 &{2}{(\Delta t)^{-1}}(\bar{\partial}_{t}u_{h}^{{1}/{2}},v_{h})+a_{h}(\widehat{u}_{h}^{1/{2}},\widehat{v}_{h})=(f^{1/{2}}+{2 u_{1}}{(\Delta t)^{-1}},v_{h}).
 \end{cases}
\end{eqnarray}

\medskip \noindent
For given $\widehat{u}_{h}^{0}$ and $\widehat{u}_{h}^{1}$ from \eqref{hwave_int_cond}, the {\it fully discrete Newmark scheme} seeks $\widehat{u}_{h}^{n+1}:= ({u}^{n+1}_{h},{u}^{n+1}_{{\cal F}_{h}},{\zeta}^{n+1}_{{\cal F}_{h}})\in \widehat{V}_{h}^{0}$ such that
\begin{equation}
(\bar{\partial}_{t}^{2}{u}_{h}^{n},v_{h})+a_{h}(\widehat{u}_{h}^{n,{1}/{4}},\widehat{v}_{h})=(f^{n,1/4},v_{h})\;\; \text{ for all } \widehat{v}_{h} \in \widehat{V}_h^0 \text{ and } n=1,2,\cdots,N-1.\label{hwave_algorithm}
\end{equation}

\begin{thm}[stability]\label{P2 full_stabilty}
If $u_0\in H^2_0(\Omega)$, $u_1 \in L^2(\Omega)$, and $f \in L^\infty(L^2(\Omega))$, then  \eqref{hwave_int_cond}-\eqref{hwave_algorithm} is unconditionally stable. Moreover, for $1 \le m \le N-1$,
\begin{equation*}
\norm{\bar{\partial}_{t}{u}_{h}^{m+{1}/{2}}}+\norm{\widehat{u}_{h}^{m+{1}/{2}}}_{a,h} \lesssim  \norm{\bar{\partial}_{t}{u}_{h}^{{1}/{2}}}+\norm{\widehat{u}_{h}^{{1}/{2}}}_{a,h}+ \norm{f}_{L^{\infty}(L^{2}(\Omega))},
\end{equation*}
where the  constant absorbed in "$\lesssim$" depends on $T$.
\end{thm}

 %\begin{equation}
%\widehat{e}_{h}^{n}:=\widehat{I}^{k}_{h}(u^{n})-\widehat{u}_{h}^{n}=\Big(\widehat{I}^{k}_{h}(u^{n})-\widehat{E}_{h}(u^{n})\Big)+\Big(\widehat{E}_{h}(u^{n})-\widehat{u}_{h}^{n}\Big):=\widehat{\theta}_{h}^{n}+\widehat{\rho}_{h}^{n}.\label{split}
% \end{equation}
\begin{lemma}[initial error bounds] \label{P1 lemma_on_norm_initial}
Let $u \in C([0,T];H^2_0(\Omega)\cap H^{2+r}(\Omega))$, with $r >3/2$ be a solution to \eqref{P1 weak_form} and $(\widehat{u}_h^0,u_h^{1})\in \widehat{V}_{h}^{0}  \times\widehat{V}_{h}^{0} $ solves \eqref{hwave_int_cond}. Assume that $u\in C^2([0,t_1];L^2(\Omega))$ and $u_{ttt}$ satisfy the global regularity $u_{ttt} \in L^\infty(0,t_1;L^2(\Omega)).$ In addition, if 
\begin{itemize}
\item $u \in L^\infty(0,t_1; H^{3+\beta}(\mathcal{T}_h))$, $u_t \in L^\infty(0,t_1;  H^{3}(\mathcal{T}_h))$ with  $\beta =\min(r-1,1)$, for $k=0$ ,
\item $u \in L^\infty(0,t_1; H^{k+3}(\mathcal{T}_h))$, \text{ and }$u_t \in L^\infty(0,t_1;  H^{k+3}(\mathcal{T}_h))$, for  $k\ge 1$, then
\end{itemize}
\vspace{-0.35cm}
 \begin{align*} \norm{\bar{\partial}_t(u^{1/2}-u_h^{1/2})}&+\big(\sum_{K\in {\cal T}_{h}}\norm{\nabla^{2}(u^{1/2}-{\cal R}_{K}(\widehat{u}_h ^{1/2}))}^2_{K} \big)^{1/2}=
 \begin{cases}
 \mathcal{O}(h+h^{1+\beta}+(\Delta t)^{2})\;&\text{ for }\; k = 0,\\
\mathcal{O}(h^{k+1}+(\Delta t)^{2})\;&\text{ for }\; k \ge 1.
 \end{cases}
\end{align*}
%where the constant absorbed in ``$\lesssim$'' constant in \textcolor{blue}{constant in approximation properties.}.
\end{lemma}
\begin{thm}[error estimate]\label{P1 implicit_Th2} 
Let $u \in  C([0,T];H^2_0(\Omega)\cap H^{2+r}(\Omega))$ with $r >3/2$~(resp. $\widehat{u}_{h}^{n}\in\widehat{V}_{h}^{0} ~(1 \le n\le N)$) solve \eqref{P1 weak_form}~(resp. \eqref{hwave_int_cond}-\eqref{hwave_algorithm}) and assume that $u \in H^4(L^2(\Omega))$. In addition, if 
\begin{itemize}
\item $u \in H^1( H^{2+r}(\Omega)\cap H^{3+\beta}(\mathcal{T}_h)), u_t\in L^\infty(H^{3+\beta}(\mathcal{T}_h)) \text{ with }\beta =\min(r-1,1),  \text{ for } k=0,$
\item $u \in H^1(H^{2+r}(\Omega)\cap H^{k+3}(\mathcal{T}_h)), u_t\in L^\infty(H^{k+3}(\mathcal{T}_h)),  \text{ for }  k\ge1,$ then for any $1 \le m \le N-1$
\end{itemize}
\vspace{-0.35cm}
%\begin{align}
 %\text{for } k=0,& \quad   L(u,t_1)=\norm{p}_{L^\infty(0,t_1;H^{3+\beta}(\mathcal{T}_h)}+\norm{u_t}_{L^2(H^{3+\beta}(\mathcal{T}_h))} \text{ with }\beta =\min(r-1,1),\\
 %&\text{for } k>0,& \quad norm{p}_{L^\infty(0,t_1;H^{k+3}(\mathcal{T}_h)}+
%\end{align}
% $$u \in L^\infty(H^{2+r}(\Omega) \cap H^{3+\beta}(\mathcal{T}_h)), u_t\in L^\infty(H^{2+r}(\Omega) \cap H^{3+\beta}(\mathcal{T}_h)) \text{ with }\beta =\min(r-1,1)  \text{ for } k=0,$$
 % $$u \in L^\infty(H^{2+r}(\Omega) \cap H^{k+3}(\mathcal{T}_h)), u_t\in L^\infty(H^{2+r}(\Omega) \cap H^{k+3}(\mathcal{T}_h))  \text{ for }  k\ge1.$$  
 \begin{align*} \norm{\bar{\partial}_t(u^{m+1/2}-u_h^{m+1/2})}&+\big(\sum_{K\in {\cal T}_{h}}\norm{\nabla^{2}(u^{m+1/2}-{\cal R}_{K}(\widehat{u}_h ^{m+1/2}))}^2_{K}\big)^{1/2} 
  =  \begin{cases}
 \mathcal{O}(h+h^{1+\beta}+(\Delta t)^{2})\;&\text{ for }\; k = 0,\\
\mathcal{O}(h^{k+1}+(\Delta t)^{2})\;&\text{ for }\; k \ge 1.
 \end{cases}  
\end{align*}
%The generic constant in the last bound depends on final time $T$ and $\alpha$.
%where the  constant absorbed in ``$\lesssim$'' depends on $T$  and $C$ from contnuity of $a_h(\bullet,\bullet)$.
\end{thm}

\begin{thm}[$L^2$-error estimates]\label{err_L2_new} 
Let $u \in  C([0,T];H^2_0(\Omega)\cap H^{2+r}(\Omega))$ with $r >3/2$~(resp. $\widehat{u}_{h}^{n}\in\widehat{V}_{h}^{0} ~(1 \le n\le N)$) solve \eqref{P1 weak_form}~(resp. \eqref{hwave_int_cond}-\eqref{hwave_algorithm}) and assume that $\widehat{I}^{k}_{h}(u^{0}) = \widehat{E}_{h}(u^{0})$ and $u \in H^4(L^2(\Omega))$. In addition, if
% Let $(u,p)$ and $(\widehat{u}_h^{n+1},p_h^{n+1})$ solve \eqref{P1 strong_form_1} and \eqref{HHO_Crank_Nico_sc_1}, respectively. Assume that  $u_{ttt},u_{tttt} \in L^2(L^2(\Omega))$, and for $r >1/2$,
%\begin{align}
 %\text{for } k=0,& \quad   L(u,t_1)=\norm{p}_{L^\infty(0,t_1;H^{3+\beta}(\mathcal{T}_h)}+\norm{u_t}_{L^2(H^{3+\beta}(\mathcal{T}_h))} \text{ with }\beta =\min(r-1,1),\\
 %&\text{for } k>0,& \quad norm{p}_{L^\infty(0,t_1;H^{k+3}(\mathcal{T}_h)}+
%\end{align}
\begin{itemize}
\item $u \in H^2(H^{2+r}(\Omega) \cap H^{3+\beta}(\mathcal{T}_h))\text{ with }\beta =\min(r-1,1)  \text{ for } k=0,$
\item $u \in H^2(H^{2+r}(\Omega) \cap H^{k+3}(\mathcal{T}_h))  \text{ for }  k\ge1,$ \text{ then } 
\end{itemize} for HHO-A scheme and {for any $1 \le m \le N-1$}, 
$ \norm{u^{m+1}- u_h^{m+1} } = 
\begin{cases}
    \mathcal{O}(h^ 2+ h^{2+\beta}+(\Delta t)^2) \quad\;\;\; \text{ for } k=0,\\
    \mathcal{O}(h^{k+3} +(\Delta t)^2 ) \qquad\qquad \text{ for }  k\ge1.
\end{cases} 
$
\end{thm}
\subsection{Crank-Nicolson scheme}\label{crn_sub}
In this subsection, we develop a time-implicit scheme for the first-order mixed (in time)
biharmonic wave problem. The biharmonic wave equation, where the displacement $u$ itself carries structural information but the 
vibrational velocity $u_t$ is 
of equal or greater physical significance in many applications. In the mechanics 
of biological tissues modelled as plate-like structures, shear wave velocity 
fields derived from the deflection are the primary quantity used to assess 
mechanical properties such as stiffness and viscoelasticity, with direct clinical 
relevance to the diagnosis of pathologies such as fibrosis and cancer 
\cite{Kijanka2024}. In cortical bone, modelled as a fluid-coupled elastic plate, 
the velocity of propagating wave modes captures dispersion characteristics that 
are directly linked to bone quality and porosity \cite{Aggelis2014, Nguyen2013}. 
Furthermore, in structural acoustics, the transverse velocity field of a biharmonic plate is 
the central quantity in computing structural intensity which identifies energy sources, 
sinks, and transmission paths \cite{LeBot2008}. %Furthermore, in cochlear 
%mechanics, the basilar membrane is modelled as a biharmonic plate, and its 
%vibrational velocity captures the frequency-selective travelling wave behaviour 
%underlying the mammalian hearing mechanism \cite{Kim2018}. 

\medskip\noindent
The aforementioned  
applications motivate the formulation introduced here: Rather than treating 
$u_t$ merely as an intermediate quantity, we introduce the vibrational velocity as 
an independent unknown $p = u_t$, and rewrite the equation as a first-order (in time)  
system of PDEs. This mixed formulation simultaneously approximates both the 
displacement $u$ and the vibrational velocity $p$ as primary variables, each 
governed by its own equation. The stability of the discrete solution 
is derived and the corresponding error estimates are established; the proofs of 
these main results in Theorem~\ref{Stab-1}--Theorem~\ref{err_lma} are presented 
in Section~\ref{Sec_cr_nic}.

\medskip \noindent
Introduce the velocity $p=u_{t}$ as independent unknown and rewrite \eqref{P1 strong_form}-\eqref{P1 strong_icbc} as
\begin{align}
 u_{t}-p=0 \quad \mbox{and} \quad p_{t} + \Delta^2 u = f(x,t) \quad  (x,t) \in \Omega \times \left(0,T \right] \label{P1 strong_form_1}
\end{align}
with initial and clamped boundary conditions
\begin{equation}
u(x,0)=u_0(x) \;\; \mbox{and}\;\; p(x,0)=u_1(x) \text{ in }\Omega,\quad    u=\frac{\partial u}{\partial n} =0 \text{ on }\partial \Omega \times (0,T], \text{ respectively}.\label{P1 strong_icbc_1}
\end{equation}
\begin{comment}
\textcolor{cyan}{In this subsection, we develop a time-implicit scheme for the first-order mixed biharmonic wave problem. This formulation involves splitting the equation into a system of PDEs by introducing a new variable 
$p=u_t$. The analysis not only approximates the displacement, but also the vibrational velocity. The stability of the solution 
$(\widehat{u}^n_{h}, p_h^n)$ is derived and the corresponding error estimates are established. The proofs of these main results in Theorem \ref{Stab-1} -Theorem~\ref{err_lma} are presented in  Section~\ref{Sec_cr_nic}.}
 
\medskip \noindent
%The first order (in time)  model formulation for \eqref{P1 strong_form}-\eqref{P1 strong_icbc} is given by i
\end{comment}
The Crank-Nicolson HHO scheme for \eqref{P1 strong_form_1}-\eqref{P1 strong_icbc_1} 
%the fully discrete HHO scheme for \eqref{P1 strong_form}-\eqref{P1 strong_icbc}
is defined as: Given  $\widehat{u}_{h}^{0}=\widehat{E}_{h}(u_{0})$ and  ${p}^{0}_{h}=\Pi_{h}^{k+2}(u_1),$ for $n=1,2,\cdots,N-1,$ seek $(\widehat{u}_{h}^{n+1},p^{n+1}_h)\in \widehat{V}_{h}^{0} \times {\cal P}_{k+2}({\cal T}_{h})$ such that
\begin{subequations}\label{HHO_Crank_Nico_sc_1}
\begin{align}
(\bar{\partial}_{t}u_{h}^{n+1/2},q_h)=(p_h^{n+1/2},q_h)\;\text{ for all }\; {q}_{h}\; \in {\cal P}_{k+2}({\cal T}_{h}),\label{HHO_Crank_Nico_sc_1-a}\\   (\bar{\partial}_{t}p^{n+1/2}_{h},v_{h})+a_{h}(\widehat{u}_{h}^{n+1/2},\widehat{v}_{h})=(f^{n+1/2},v_{h})\;\text{ for all }\;\widehat{v}_{h}\in \; \widehat{V}_{h}^{0} \label{HHO_Crank_Nico_sc_1-b}.
\end{align}
\end{subequations}
\begin{thm}[stability]\label{Stab-1}  
If $u_0\in \mathcal{Y},u_1 \in L^2(\Omega),$ and $f\in H^1(L^2(\Omega)),$ then  \eqref{HHO_Crank_Nico_sc_1} is unconditionally stable. Moreover, for  $0 \le m\le N-1$,  the solution $(\widehat{u}_{h}^{m+1},p_h^{m+1})$ satisfies
$$\norm{p_{h}^{m+1}}+\norm{\widehat{u}_{h}^{m+1}}_{a,{h}}\lesssim \norm{p_h^0}+\norm{\widehat{u}_{h}^{0}}_{a,{h}}+\norm{f}_{L^\infty(L^2(\Omega))}+\norm{f_t}_{L^2(L^2(\Omega))},$$
where the constant absorbed in "$\lesssim$" depends on $T$ linearly and $C_P$ from \eqref{poinacre}.
\end{thm}
\begin{thm}[error estimates]\label{err_thm}  
Let $(u,p)$~(resp. $(\widehat{u}_h^{n+1},p_h^{n+1})~(0 \le m\le N-1)$) solve \eqref{P1 strong_form_1}~(resp. \eqref{HHO_Crank_Nico_sc_1}). Assume that $u\in H^4(L^2(\Omega))$ and $p\in H^4(L^2(\Omega))$. In addition, if
%\begin{align}
 %\text{for } k=0,& \quad   L(u,t_1)=\norm{p}_{L^\infty(0,t_1;H^{3+\beta}(\mathcal{T}_h)}+\norm{u_t}_{L^2(H^{3+\beta}(\mathcal{T}_h))} \text{ with }\beta =\min(r-1,1),\\
 %&\text{for } k>0,& \quad norm{p}_{L^\infty(0,t_1;H^{k+3}(\mathcal{T}_h)}+
%\end{align}
\begin{itemize}
\item  $u \in H^2(H^{2+r}(\Omega) \cap H^{3+\beta}(\mathcal{T}_h)) \text{ and }p \in L^\infty( H^{3}(\mathcal{T}_h))\text{ with }\beta =\min(r-1,1) \text{ and }(r> 3/2)  \text{ for } k=0,$
\item $u \in H^2(H^{2+r}(\Omega) \cap H^{k+3}(\mathcal{T}_h))\text{ and }p \in L^\infty( H^{k+3}(\mathcal{T}_h))  \text{ for }  k\ge1,$ \text{ then } 
\end{itemize}
 % $$u_t \in L^\infty(H^{2+r}(\Omega) \cap H^{k+3}(\mathcal{T}_h)), u_{tt}\in L^2(H^{2+r}(\Omega) \cap H^{k+3}(\mathcal{T}_h))  \text{ for }  k\ge1. \text{ Then,   for any  } 0 \le m\le N-1,$$
 \begin{align*} 
\big(\sum_{K\in {\cal T}_{h}}\norm{\nabla^{2}(u^{m+1/2}-{\cal R}_{K}(\widehat{u}_h ^{m+1}))}^{2}_{K}\big)^{1/2} + \norm{p^{m+1}-p_h^{m+1/2}}
=
 \begin{cases}
 \mathcal{O}(h+h^{1+\beta}+(\Delta t)^{2})\;&\text{ for }\; k = 0,\\
\mathcal{O}(h^{k+1}+(\Delta t)^{2})\;&\text{ for }\; k \ge 1.
 \end{cases}
\end{align*}
Moreover, for HHO-A, it holds 
 %\begin{align*} 
 $\displaystyle \norm{ p^{m+1}- p_h^{m+1} } = \begin{cases}
    \mathcal{O}(h^2+ h^{2+\beta}+(\Delta t)^2) \quad\;\;\; &\text{ for } k=0,\\
    \mathcal{O}(h^{k+3} +(\Delta t)^2 ) \quad &\text{ for }  k\ge1.
\end{cases}$
%\end{align*}
%where the  constant absorbed in ``$\lesssim$''  depends on $T$ linearly and $C_P$ from \eqref{poinacre}.
\end{thm}
\begin{thm}[$L^2$-error estimates for $u$]\label{err_lma} 
Let $(u,p)$~(resp. $(\widehat{u}_h^{n+1},p_h^{n+1})~(0 \le n\le N-1)$) solve \eqref{P1 strong_form_1}~(resp. \eqref{HHO_Crank_Nico_sc_1}).  Assume that $u\in H^3(L^2(\Omega))$ and $p\in H^3(L^2(\Omega))$. In addition, if
% Let $(u,p)$ and $(\widehat{u}_h^{n+1},p_h^{n+1})$ solve \eqref{P1 strong_form_1} and \eqref{HHO_Crank_Nico_sc_1}, respectively. Assume that  $u_{ttt},u_{tttt} \in L^2(L^2(\Omega))$, and for $r >1/2$,
%\begin{align}
 %\text{for } k=0,& \quad   L(u,t_1)=\norm{p}_{L^\infty(0,t_1;H^{3+\beta}(\mathcal{T}_h)}+\norm{u_t}_{L^2(H^{3+\beta}(\mathcal{T}_h))} \text{ with }\beta =\min(r-1,1),\\
 %&\text{for } k>0,& \quad norm{p}_{L^\infty(0,t_1;H^{k+3}(\mathcal{T}_h)}+
%\end{align}
\begin{itemize}
\item $u \in H^1(H^{2+r}(\Omega) \cap H^{3+\beta}(\mathcal{T}_h))\text{ with }\beta =\min(r-1,1)  \text{ for } k=0,$
\item $u \in H^1(H^{2+r}(\Omega) \cap H^{k+3}(\mathcal{T}_h))  \text{ for }  k\ge1,$ 
\end{itemize}
then for HHO-A, it holds
$ \norm{u^{m+1}- u_h^{m+1} } = 
\begin{cases}
    \mathcal{O}(h^ 2+ h^{2+\beta}+(\Delta t)^2) \quad\;\;\;& \text{ for } k=0,\\
    \mathcal{O}(h^{k+3} +(\Delta t)^2 ) \quad &\text{ for }  k\ge1.
\end{cases}
$
%The generic constant in the last bound depends on final time $T$ and $C_P$ from \eqref{poinacre}.
% \begin{align*} 
 % \begin{cases}
   % h^2\big( \norm{p}_{L^\infty(0,t_1;H^{3+\beta}(\mathcal{T}_h)} +\norm{u_t}_{L^2(H^{3+\beta}(\mathcal{T}_h))}\big)+(\Delta t)^2\big(\norm{u_{ttt}}_{L^2(L^2(\Omega))}+\norm{u_{tttt}}_{L^2(L^2(\Omega))}\big) \quad \text{ for } k=0,\\
   % h^{k+3}\big(\norm{p}_{L^\infty(0,t_1;H^{k+3}(\mathcal{T}_h)}+ \norm{u_t}_{L^\infty(0,t_1;H^{k+3}(\mathcal{T}_h)}\big)+(\Delta t)^2\big(\norm{u_{ttt}}_{L^2(L^2(\Omega))}+\norm{u_{tttt}}_{L^2(L^2(\Omega))}\big) \quad \text{ for } k>0.
%\end{cases}
%\end{align*}
%\mathcal{L}\big(h,k,\norm{p}_{L^\infty(0,t_1;H^{3+\beta}(\mathcal{T}_h)},\norm{p}_{L^\infty(0,t_1;H^{k+3}(\mathcal{T}_h)}\big)_{L^2}\\
%&+\mathcal{L}(h,k,\norm{u_t}_{L^2(H^{3+\beta}(\mathcal{T}_h)},\norm{u_t}_{L^2(H^{k+3}(\mathcal{T}_h))})_{L^2} \\
%&+(\Delta t)^2\big(\norm{u_{ttt}}_{L^2(L^2(\Omega))}+\norm{u_{tttt}}_{L^2(L^2(\Omega))}\big), 
%\end{align*}
%with the  constant absorbed in ``$\lesssim$'' dependent on $T$ linearly  .
\end{thm}

\subsection{Newmark vs Crank-Nicolson 
Scheme}\label{subsec-comp}
We conclude this section with a comparative discussion of the HHO-Newmark and HHO-Crank--Nicolson schemes. While both schemes are 
implicit and unconditionally stable, they differ in several important 
aspects regarding their structure, the variables they approximate, 
the regularity they require, and the norms in which optimal convergence 
is achieved.

\medskip \noindent
The Newmark scheme is an uncoupled three-step method that solves a single 
equation for the displacement $u$ at each time level, but requires an 
additional initialization step at $t=t_1$ via the second equation in 
\eqref{hwave_int_cond}. In contrast, the Crank--Nicolson scheme is a 
two-step method that avoids this extra initialization but needs to solve 
a coupled system for both the displacement $u$ and the velocity $p$ at 
each time step. Consequently, the Newmark scheme yields an approximation 
only for $u$, whereas the Crank--Nicolson scheme  approximates 
both $u$ and $p$. Regarding the data assumptions for stability,  the Newmark 
scheme requires only $u_0 \in H^2_0(\Omega)$, $u_1 \in L^2(\Omega)$, and 
$f \in L^\infty(L^2(\Omega))$, while the Crank--Nicolson scheme demands 
the slightly higher regularity $u_0 \in \mathcal{Y}$, $u_1 \in L^2(\Omega)$, 
and $f \in H^1(L^2(\Omega))$, reflecting the need to handle the coupled 
system consistently at the discrete level.

\medskip \noindent 
Both schemes achieve identical optimal convergence rates in the energy norm 
and in the $L^2$-norm (for the HHO-A variant); see Table~\ref{tab:comparison}. 
However, the regularity requirements on the continuous solution exhibit a 
subtle reversal between the two norms. For energy-norm estimates, the Newmark 
scheme requires less regularity than the Crank-Nicolson scheme. This hierarchy is reversed for $L^2$-norm 
estimates, where the condition $u_0 \in \mathcal{Y}$ becomes essential  
even for the Newmark scheme to give meaning to the compatibility condition 
$\widehat{I}^k_h(u^0) = \widehat{E}_h(u^0)$ in Theorem~\ref{err_L2_new}. {However, the condition $\widehat{I}^k_h(u^0) = \widehat{E}_h(u^0)$ is not required in numerical implementation.} 
We emphasize that the regularity listed in Table~\ref{tab:comparison} represents 
the maximum required to ensure all results in this paper (with the precise requirements indicated in Subsections ~\ref{nm_schm_sub_1}-\ref{crn_sub}), and that 
Table~\ref{tab:comparison} provides a comprehensive  summary of both the schemes.
\begin{table}[H]
\centering
\small
\renewcommand{\arraystretch}{1.3}
\begin{tabular}{|p{2.7cm}||p{5.8cm}|p{6.5cm}|}
 \hline
  \textbf{Features} & \textbf{HHO-Newmark} & \textbf{HHO-Crank-Nicolson} \\
    \hline
    Structure & Uncoupled PDE & Coupled PDE system\\
    \hline 
    Variables & Displacement $u$ & Displacement $u$ and velocity $p$\\
    \hline 
    Time-stepping &{Implicit}&{Implicit}\\
    \hline
    Stability &{Unconditional}&{Unconditional}\\
    \hline
     Sequential steps & 3-step scheme & 2-step scheme \\
    
    \hline
    Assumption on data & $u(x,0)=u_0 \in \mathcal{Y}$, with $\mathcal{Y}$ defined in \eqref{proj_op_1} \par\smallskip
                          $ u_t(x,0)=u_1 \in L^2(\Omega)$ \par\smallskip
                          $ f\in L^\infty(L^2(\Omega))$
    &{$u_0 \in \mathcal{Y}$}  \par\smallskip
                          $ u_1 \in L^2(\Omega)$ \par\smallskip
                          $ f\in H^1(L^2(\Omega))$\\
    \hline
    Regularity requirements of solution
    &   \textbullet\; $u \in C([0,T];H^2_0(\Omega))\cap H^4(L^2(\Omega))$\par\smallskip
      \textbullet\; {$k=0$:}\; $u \in H^2(H^{2+r}(\Omega)
          \cap H^{3+\beta}(\mathcal{T}_h))$\;
         \par\smallskip
      \textbullet\; {$k\ge1$:}\; $u \in H^2(H^{2+r}(\Omega)
          \cap H^{k+3}(\mathcal{T}_h))$
    &    \textbullet\; $u \in C([0,T];H^2_0(\Omega))\cap H^4(L^2(\Omega))$\par\smallskip 
      \textbullet\; {$k\ge0$:}\; $p \in L^\infty( H^{k+3}(\mathcal{T}_h)) \cap  H^4(L^2(\Omega))$\par\smallskip
     \textbullet\; {$k=0$:}\; $u \in H^2(H^{2+r}(\Omega) \cap H^{3+\beta}(\mathcal{T}_h))$\par\smallskip
      \textbullet\; {$k\ge1$:}\; $u \in H^2(H^{2+r}(\Omega) \cap H^{k+3}(\mathcal{T}_h))$ \\
       \hline
    \multicolumn{3}{|c|}{\textit{Convergence}} \\
    \hline
{Energy norm estimates for both HHO-A and HHO-B}
  & \begin{minipage}[t]{5.6cm}
      \vspace{4pt}
      $\begin{gathered}
        \Bigl(\sum_{K\in \mathcal{T}_{h}}
          \|\nabla^{2}(u^{m+1/2}-\mathcal{R}_{K}(\widehat{u}_h^{m+1/2}))\|^2_{K}
        \Bigr)^{1/2} \\
        =\begin{cases}
            \mathcal{O}(h+h^{1+\beta}+(\Delta t)^{2}) & k = 0,\\[4pt]
            \mathcal{O}(h^{k+1}+(\Delta t)^{2}) & k \ge 1.
        \end{cases}
      \end{gathered}$
      \vspace{4pt}
    \end{minipage}
  & \begin{minipage}[t]{5.9cm}
      \vspace{4pt}
      $\begin{gathered}
 \Bigl(\sum_{K\in \mathcal{T}_{h}}
          \|\nabla^{2}(u^{m+1}-\mathcal{R}_{K}(\widehat{u}_h^{m+1}))\|^2_{K}
        \Bigr)^{1/2} \qquad\qquad\\
     =\begin{cases}
            \mathcal{O}(h+h^{1+\beta}+(\Delta t)^{2}) & k = 0,\\[4pt]
            \mathcal{O}(h^{k+1}+(\Delta t)^{2}) & k \ge 1.
        \end{cases}
      \end{gathered}$
      \vspace{4pt}
    \end{minipage} \\
\hline
{$L^2$ norm estimates for HHO-A}
  & \begin{minipage}[t]{5.6cm}
      \vspace{4pt}
      $\begin{gathered}
     \|u^{m+1}- u_h^{m+1}\|\qquad \qquad\qquad\qquad\qquad\quad\quad\\
       =\begin{cases}
            \mathcal{O}(h^{2}+h^{2+\beta}+(\Delta t)^{2}) & k = 0,\\[4pt]
            \mathcal{O}(h^{k+3}+(\Delta t)^{2}) & k \ge 1.
        \end{cases}
      \end{gathered}$
      \vspace{4pt}
    \end{minipage}
  & \begin{minipage}[t]{6.5cm}
      \vspace{4pt}
      $\begin{gathered}
         \|u^{m+1}-u_h^{m+1}\|+\|p^{m+1}-p_h^{m+1}\|\qquad\qquad\qquad\\
        =\begin{cases}
            \mathcal{O}(h^{2}+h^{2+\beta}+(\Delta t)^{2}) & k = 0,\\[4pt]
            \mathcal{O}(h^{k+3}+(\Delta t)^{2}) & k \ge 1.
        \end{cases}
      \end{gathered}$
      \vspace{4pt}
    \end{minipage} \\
\hline
\end{tabular}
\caption{Comparison of Newmark and Crank-Nicolson schemes with  HHO-A and HHO-B discretizations.}
\label{tab:comparison}
\end{table}
\section{A novel $L^2$ projection estimate for HHO-A scheme}\label{subsec_l2_est}
In this section, we prove the novel $L^2$ approximation properties of the HHO-A method, which will be used to prove $L^2$ error estimates for HHO-A  Newmark and Crank-Nicolson fully discrete schemes.

\medskip\noindent
\uline{\textbf{Proof of Theorem \ref{m1_H^2_bound}.}}
The proof is divided into ten steps. We describe the settings in \textit{Step~1} and the error equation in \textit{Step~2} that leads to the identity \eqref{ct_tt}. The terms appearing on the right-hand side of \eqref{ct_tt} are estimated in \textit{Step~3}-\textit{Step 9}. The proof is concluded in \textit{Step~10}.

\medskip\noindent
{{\text{{\it Step 1. (setting)}}}}
Set $\widehat{\varphi}_{h}:=\widehat{I}^{k}_{h}(w)-\widehat{E}_{h}(w):=(\varphi_{h},\varphi_{{\cal F}_{h}},\widetilde{\varphi}_{{\cal F}_{h}})$ with   ${\varphi}_{h}=\Pi^{k+2}_{h}(w)-{E}_{h}(w).$ 

\medskip\noindent
{{\text{{\it Step 2. (error equation)}}}}
To derive the sharp bound in $L^{2}$ norm, we recall \eqref{eq:dual} and the assumption \eqref{eq:reg_bound}: Given ${\varphi}_{h}\in L^{2}(\Omega),$ there exist a unique  $\Psi\in H^{2}_{0}(\Omega) \cap H^4(\Omega)$ such that
\begin{equation}
\Delta^{2}\Psi=\varphi_{h}\; \text{in}\; \Omega,\quad \Psi=\partial_{\bf n}\Psi=0 \; \text{on}\; \partial \Omega\;\; \text{ with }\;\; \norm{\psi}_{H^{4}(\Omega)}\lesssim \norm{{\varphi}_{h}}.\label{dual_1}
\end{equation}
%where $a(\Psi,v):= (\nabla^{2}\Psi,\nabla^{2}v)$. 
Test \eqref{dual_1} with $\varphi_{h}$ and employ  {\it integration by parts} from \eqref{int_part_1}  to obtain 
\begin{align*}
\norm{\varphi_{h}}^{2}=(\Delta^{2}\Psi,\varphi_{h})= \sum_{K\in {\cal T}_{h}}\big((\nabla^{2}\Psi,\nabla^{2}\varphi_{K})_{K}+\sum_{F\in {\cal F}_{K}}\big((\varphi_{K},\partial_{\bf n}\Delta \Psi)_{F}-(\partial_{\bf n}\varphi_{K},\partial_{\bf n\bf n} \Psi)_{F}-(\partial_{\bf t}\varphi_{K},\partial_{\bf n\bf t} \Psi)_{F}\big)\big).
\end{align*}
Since $\Psi\in H^{4}(\Omega),$ $\partial_{\bf n}\Delta \Psi$, $\partial_{\bf n\bf n} \Psi$, and $\partial_{\bf n\bf t} \Psi$~(see Table~\eqref{HHO_Table} for notation) are single-valued and are $L^{2}(K)$ functions across the local interfaces. In addition, $\varphi_{F}$, $\widetilde{\varphi}_{F}$, and $\partial_{\bf t}\varphi_{F}$ are single-valued at every mesh-interface and vanish at every boundary face. Hence, the last displayed equation yields
\begin{align}
\norm{\varphi_{h}}^{2}&=\sum_{K\in {\cal T}_{h}}\big((\nabla^{2}\Psi,\nabla^{2}\varphi_{K})_{K}+\sum_{F\in {\cal F}_{K}}\big((\varphi_{K}-\varphi_{F},\partial_{\bf n}\Delta \Psi)_{F}\nonumber\\&\qquad\qquad -(\partial_{\bf n}\varphi_{K}-\widetilde{\varphi}_{F},\partial_{\bf n\bf n} \Psi)_{F}-(\partial_{\bf t}(\varphi_{K}-\varphi_{F}),\partial_{\bf n\bf t} \Psi)_{F}\big)\big).\label{dual_eq_2}
\end{align}
The definition ${\cal R}_{K}={\cal E}_{K}\circ \widehat{I}_{K}^{k}$, symmetry of ${\cal R}_{K}$, \eqref{loc_recon_op_2}, and elementary algebra show
\begin{align}
&(\nabla^{2}{\cal R}_{K}(\widehat{\varphi}_{K}),\nabla^{2}{\cal R}_{K}(\widehat{I}^{k}_{K}(\Psi)))_{K}=(\nabla^{2}{\cal R}_{K}(\widehat{\varphi}_{K}),\nabla^{2}{\cal E}_{K}(\Psi))_{K}=(\nabla^{2}\varphi_{K},\nabla^{2}{\cal E}_{K}(\Psi))_{K}\nonumber\\
&\qquad+\sum_{F\in {\cal F}_{K}}\big((\varphi_{K}-\varphi_{F},\partial_{\bf n}\Delta{\cal E}_{K}(\Psi))_{F}
 -(\partial_{\bf n}\varphi_{K}-\widetilde{\varphi}_{F},\partial_{\bf n\bf n}{\cal E}_{K}(\Psi))_{F}-(\partial_{\bf t}(\varphi_{K}-{\varphi}_{F}),\partial_{\bf n\bf t}{\cal E}_{K}(\Psi))_{F}\big).\label{5_5_eq_lb}
\end{align}
It follows from \eqref{dual_eq_2} and \eqref{5_5_eq_lb}, after a rearrangement of terms, that
\begin{align}
\norm{\varphi_{h}}^{2}&=\sum_{K\in {\cal T}_{h}}(\nabla^{2}(\Psi-{\cal E}_{K}(\Psi)),\nabla^{2}\varphi_{K})_{K}+\sum_{K\in {\cal T}_{h}}\sum_{F\in {\cal F}_{K}}(\varphi_{K}-\varphi_{F},\partial_{\bf n}\Delta(\Psi-{\cal E}_{K}(\Psi)))_{F}
 \nonumber\\&\qquad-\sum_{K\in {\cal T}_{h}}\sum_{F\in {\cal F}_{K}}( \partial_{\bf n}\varphi_{K}-\widetilde{\varphi}_{F},\partial_{\bf n\bf n}(\Psi-{\cal E}_{K}(\Psi)))_{F}-\sum_{K\in {\cal T}_{h}}\sum_{F\in {\cal F}_{K}}(\partial_{\bf t}(\varphi_{K}-{\varphi}_{F}),\partial_{\bf n\bf t}(\Psi-{\cal E}_{K}(\Psi)))_{F} \nonumber\\&\qquad
 +\sum_{K\in {\cal T}_{h}}\sum_{F\in {\cal F}_{K}}(\nabla^{2}{\cal R}_{K}(\widehat{\varphi}_{K}),\nabla^{2}{\cal R}_{K}(\widehat{I}^{k}_{K}(\Psi)))_{K}.\label{ct_tt}
\end{align}
In the case of the HHO-A scheme, the first term on the right-hand side of \eqref{ct_tt} vanishes by using \eqref{eq:elli.proj.use_1} with $w=\Psi$ and $\phi=\varphi_{K}.$
% by the defining property of the elliptic projection ${\cal E}_{K}$ from Subsection~\ref{subsec-pro}, which yields
% \begin{equation}\label{eq:elli.proj.use_1}
% (\nabla^{2}(\Psi-{\cal E}_{K}(\Psi)),\nabla^{2}\varphi_{K})_{K}=0.    
% \end{equation}

\medskip \noindent
Now, we will simplify the last term on the right-hand side of \eqref{ct_tt}. The definitions of ${a_h(\bullet,\bullet)}$ and  $\widehat{\varphi}_{h}=\widehat{I}^{k}_{h}(w)-\widehat{E}_{h}(w)$ followed by \eqref{proj_op_1} lead to 
\begin{align}\label{eq:t4.simplify}
&{\sum_{K\in {\cal T}_{h}}}(\nabla^{2}{\cal R}_{K}(\widehat{\varphi}_{K}),\nabla^{2}{\cal R}_{K}(\widehat{I}^{k}_{h}(\Psi)))_{K}= a_{h}(\widehat{\varphi}_{h},\widehat{I}^{k}_{h}(\Psi))-{\cal S}_{h}(\widehat{\varphi}_{h},\widehat{I}^{k}_{h}(\Psi)) 
\nonumber\\& =a_{h}(\widehat{I}^{k}_{h}(w),\widehat{I}^{k}_{h}(\Psi))-( \Delta^{2}w,\Pi_{h}^{k+2}(\Psi))-{\cal S}_{h}(\widehat{\varphi}_{h},\widehat{I}^{k}_{h}(\Psi))
% \nonumber\\
% \end{align*}
% \begin{align*}
% &=a_{h}(\widehat{I}^{k}_{h}(w),\widehat{I}^{k}_{h}(\Psi))-( \Delta^{2}w-\Pi_{h}^{k+2}( \Delta^{2}w),\Pi_{h}^{k+2}(\Psi)-\Psi)-(\nabla^{2}w,\nabla^{2}\Psi)
% -{\cal S}_{h}(\widehat{\varphi}_{h},\widehat{I}^{k}_{h}(\Psi))\nonumber\\&
\nonumber\\&=
\sum_{K\in {\cal T}_{h}}(\nabla^{2}(w-{\cal E}_{K}(w)),\nabla^{2}(\Psi-{\cal E}_{K}(\Psi)))_{K} -( \Delta^{2}w-\Pi_{h}^{k+2}( \Delta^{2}w),\Pi_{h}^{k+2}(\Psi)-\Psi)\nonumber\\& \qquad + {\cal S}_{h}(\widehat{I}_{h}^{k}(w),\widehat{I}^{k}_{h}(\Psi))-{\cal S}_{h}(\widehat{\varphi}_{h},\widehat{I}^{k}_{h}(\Psi)),
\end{align}
where an integration by parts that leads to $( \Delta^{2}w,\Psi)=(\nabla^{2}w,\nabla^{2}\Psi)$; the definitions of $\Pi_{h}^{k+2}$ and ${\cal E}_{h}$, and elementary manipulations are used in the last step.

\medskip\noindent
Finally, substitute \eqref{eq:elli.proj.use_1} and \eqref{eq:t4.simplify} in~\eqref{ct_tt} to obtain
\begin{align}
\norm{\varphi_{h}}^{2}&=\sum_{K\in {\cal T}_{h}}\sum_{F\in {\cal F}_{K}}(\varphi_{K}-\varphi_{F},\partial_{\bf n}\Delta(\Psi-{\cal E}_{K}(\Psi)))_{F}
 \nonumber\\&\qquad-\sum_{K\in {\cal T}_{h}}\sum_{F\in {\cal F}_{K}}( \partial_{\bf n}\varphi_{K}-\widetilde{\varphi}_{F},\partial_{\bf n\bf n}(\Psi-{\cal E}_{K}(\Psi)))_{F}-\sum_{K\in {\cal T}_{h}}\sum_{F\in {\cal F}_{K}}(\partial_{\bf t}(\varphi_{K}-{\varphi}_{F}),\partial_{\bf n\bf t}(\Psi-{\cal E}_{K}(\Psi)))_{F} \nonumber\\&
 \qquad +\sum_{K\in {\cal T}_{h}}(\nabla^{2}(w-{\cal E}_{K}(w)),\nabla^{2}(\Psi-{\cal E}_{K}(\Psi)))_{K} -( \Delta^{2}w-\Pi_{h}^{k+2}( \Delta^{2}w),\Pi_{h}^{k+2}(\Psi)-\Psi)\nonumber\\& \qquad +{\cal S}_{h}(\widehat{I}_{h}^{k}(w),\widehat{I}^{k}_{h}(\Psi))-{\cal S}_{h}(\widehat{\varphi}_{h},\widehat{I}^{k}_{h}(\Psi)):=\sum_{i=1}^{7}T_{i}.
 \label{eq:sum}
\end{align}

 \medskip \noindent
{{\text{{\it Step 3. (control of $T_{1}$)}}}} 
The Cauchy--Schwarz inequality, definition of $\norm{\bullet}_{\widehat{V}_{h}},$ Lemma~\ref{eq:lem.equiv}, and \eqref{m1_H_2_like_nm_def} yield
\begin{align*}
|T_{1}|&\leq \sum_{K\in {\cal T}_{h}}\sum_{F\in {\cal F}_{K}}\norm{\varphi_{K}-\varphi_{F}}_{F}\norm{\partial_{\bf n}\Delta(\Psi-{\cal E}_{K}(\Psi))}_{F}\\
  & \leq \sum_{K\in {\cal T}_{h}}\left(\sum_{F\in {\cal F}_{K}}h_{K}^{-3}\norm{\varphi_{K}-\varphi_{F}}_{F}^2\right)^{\frac{1}{2}}\left(\sum_{F\in {\cal F}_{K}}h_{K}^{3}\norm{\partial_{\bf n}\Delta(\Psi-{\cal E}_{K}(\Psi))}_{F}^2\right)^{1/2}\\
  &\lesssim \sum_{K\in {\cal T}_{h}}|\widehat{\varphi}_{K}|_{\widehat{V}_{K}}\norm{\Psi-{\cal E}_{K}(\Psi)}_{\#,K} 
\lesssim  \norm{\widehat{\varphi}_{h}}_{a,h}\left(\sum_{K\in {\cal T}_{h}}\norm{\Psi-{\cal E}_{K}(\Psi)}_{\#,K}\right).
\end{align*}
An application of \eqref{eq:w.EK.bound} and Theorem~\ref{lm1_in_dis_tr_11}$(c)$ yields 
\begin{equation}\label{eq:psi.EK.bound}
\begin{aligned}
\sum_{K\in {\cal T}_{h}}\norm{\Psi-{\cal E}_{K}(\Psi)}_{\#,K} \lesssim \sum_{K\in {\cal T}_{h}}\norm{\Psi-{\Pi}^{k+2}_{K}(\Psi)}_{\#,K}
&\lesssim 
\begin{cases}
 \sum\limits_{K\in {\cal T}_{h}} h_{K}\big(|\Psi|_{H^{3}(K)}+h_{K}|\Psi|_{H^{4}(K)}\big) & \text{for } k=0 \\
 \sum\limits_{K\in {\cal T}_{h}} h_{K}^{2} |\Psi|_{H^{4}(K)} & \text{for } k\geq 1
\end{cases} \\
&\lesssim
\begin{cases}
 h\norm{\varphi_{h}} & \text{for } k=0,\\
 h^{2}\norm{\varphi_{h}} & \text{for } k\geq 1
\end{cases}
\end{aligned}
\end{equation}
with \eqref{dual_1} in the last step. 
A combination of the above two displayed results yields
$$\displaystyle |T_{1}|\lesssim  \begin{cases}
h\norm{\widehat{\varphi}_{h}}_{a,h}\norm{\varphi_{h}}&\text{for}\; k=0,\\
h^{2}\norm{\widehat{\varphi}_{h}}_{a,h}\norm{\varphi_{h}}& \text{for}\; k\geq 1.
 \end{cases} $$

\medskip\noindent
{{\text{{\it Step 4. (control of $T_{2}$)}}}} 
% The  Cauchy--Schwarz inequality and the definition of stabilizer from Subsection~\ref{subsec-stab} for HHO-A scheme to obtain
% \begin{align*}
% |T_{2}| \leq \sum_{K\in {\cal T}_{h}}\sum_{F\in {\cal F}_{K}} \norm{\partial_{\bf n}\varphi_{K}-\widetilde{\varphi}_{F}}_{F}\norm{\partial_{\bf n\bf n}(\Psi-{\cal E}_{K}(\Psi))}_{F} \leq \sum_{K\in {\cal T}_{h}}\sum_{F\in {\cal F}_{K}}h_{K}^{\frac{1}{2}} {\cal S}_{K}(\widehat{\varphi}_{K},\widehat{\varphi}_{K})^{\frac{1}{2}} \norm{\partial_{\bf n\bf n}(\Psi-{\cal E}_{K}(\Psi))}_{F}.
% \end{align*}
% An application of Theorem~\ref{lm1_in_dis_tr_11}(b),(d) and \eqref{dual_1} yield
% \begin{align*}
% \sum_{K\in {\cal T}_{h}}\sum_{F\in {\cal F}_{K}}h_{K}^{\frac{1}{2}} \norm{\partial_{\bf n\bf n}(\Psi-{\cal E}_{K}(\Psi))}_{F} 
% &\lesssim \sum_{K\in {\cal T}_{h}} h_{K}^{\frac{1}{2}}\big(h_{K}^{-\frac{1}{2}}\norm{\Psi-{\cal E}_{K}(\Psi)}_{H^{2}(K)}+h_{K}^{\frac{1}{2}}\norm{\Psi-{\cal E}_{K}(\Psi)}_{H^{3}(K)}\big)\nonumber\\&
% \lesssim \sum_{K\in {\cal T}_{h}}h^{2}_{K}\norm{\Psi}_{H^{4}(K)}\lesssim h^{2}\norm{{\varphi}_{h}}.
% \end{align*}
% Now, the above two displayed inequalities and the definition of $\norm{\bullet}_{a,h}$ establish
% $|T_{2}|\lesssim h^{2}\norm{\widehat{\varphi}_{h}}_{a,h} \norm{{\varphi}_{h}}.$
The  Cauchy--Schwarz inequality, definition of $\norm{\bullet}_{\widehat{V}_{h}},$ Lemma~\ref{eq:lem.equiv}, and \eqref{m1_H_2_like_nm_def} yield
\begin{align*}
|T_{2}|&\leq \sum_{K\in {\cal T}_{h}}\sum_{F\in {\cal F}_{K}} \norm{\partial_{\bf n}\varphi_{K}-\widetilde{\varphi}_{F}}_{F}\norm{\partial_{\bf n\bf n}(\Psi-{\cal E}_{K}(\Psi))}_{F}\\
&\lesssim \sum_{K\in {\cal T}_{h}}\left(\sum_{F\in {\cal F}_{K}} h_{K}^{-1}\norm{\partial_{\bf n}\varphi_{K}-\widetilde{\varphi}_{F}}_{F}^2\right)^{\frac{1}{2}}\left(\sum_{F\in {\cal F}_{K}}h_{K}\norm{\partial_{\bf n\bf n}(\Psi-{\cal E}_{K}(\Psi))}_{F}^2\right)^{\frac{1}{2}}\\
&\lesssim \sum_{K\in {\cal T}_{h}}|\widehat{\varphi}_{K}|_{\widehat{V}_{K}}\norm{\Psi-{\cal E}_{K}(\Psi)}_{\#,K} 
\lesssim  \norm{\widehat{\varphi}_{h}}_{a,h}\left(\sum_{K\in {\cal T}_{h}}\norm{\Psi-{\cal E}_{K}(\Psi)}_{\#,K}\right).
\end{align*}
% where definition of $\norm{\bullet}_{a,h}$ and \eqref{m1_H_2_like_nm_def} are used in last step.
% From~\eqref{m1_H_2_like_nm_def}, Lemma~\ref{bd_elli_proj}, and Theorem~\ref{lm1_in_dis_tr_11}$(c),$ we have
% \begin{align*}
% &\sum_{K\in {\cal T}_{h}}\sum_{F\in {\cal F}_{K}}h_{K}^{\frac{1}{2}} \norm{\partial_{\bf n\bf n}(\Psi-{\cal E}_{K}(\Psi))}_{F} 
% \lesssim \sum_{K\in {\cal T}_{h}}\norm{\Psi-{\cal E}_{K}(\Psi)}_{\#,K}\\&
% \lesssim \begin{cases}
%  \sum\limits_{K\in {\cal T}_{h}} h_{K}\big(|\Psi|_{H^{3}(K)}+h_{K}|\Psi|_{H^{4}(K)}\big) & \text{for}\; k=0\\
%     \sum\limits_{K\in {\cal T}_{h}} h^{2}_{K} {|\Psi|_{H^{4}(K)} }& \text{for}\; k\geq 1
%  \end{cases}\overset{\eqref{dual_1}}\lesssim
%  \begin{cases}
%  h\norm{\varphi_{h}}&\text{for}\; k=0,\\
% h^{2}\norm{\varphi_{h}}& \text{for}\; k\geq 1,\nonumber 
%  \end{cases}
% \end{align*}
The above inequality and \eqref{eq:psi.EK.bound} lead to
$$|T_{2}|\lesssim \begin{cases}
h\norm{\widehat{\varphi}_{h}}_{a,h}\norm{\varphi_{h}}&\text{for}\; k=0,\\
h^{2}\norm{\widehat{\varphi}_{h}}_{a,h}\norm{\varphi_{h}}& \text{for}\; k\geq 1.
 \end{cases} $$

\medskip\noindent
{{\text{{\it Step~5. (control of $T_{3}$)}}}} The bound for $T_{3}$ follows from analogous arguments in Step~4.

\medskip\noindent
{{\text{{\it Step 6. (control of $T_{4}$)}}}} 
% A Cauchy--Schwarz inequality reveals
% Theorem \ref{lm1_in_dis_tr_11}$(d),$ and  \eqref{dual_1} yield
% \begin{align}
% |T_{4}|&\lesssim \sum_{K\in {\cal T}_{h}}\norm{\nabla^{2}(w-{\cal E}_{K}(w))}_{K}\norm{\nabla^{2}(\Psi-{\cal E}_{K}(\Psi))}_{K} \overset{\eqref{m1_H_2_like_nm_def}}\lesssim h^{2}\sum_{K\in {\cal T}_{h}}\norm{w-\Pi^{k+2}_{K}(w)}_{\#,K}|\Psi|_{H^{4}(\Omega)}\nonumber\\& \lesssim h^{2}\sum_{K\in {\cal T}_{h}}\norm{w-\Pi^{k+2}_{K}(w)}_{\#,K} \norm{\varphi_{h}}.
% \end{align}
A Cauchy--Schwarz inequality and \eqref{m1_H_2_like_nm_def} reveal
\begin{align}\label{eq:T4.int.bound}
 |T_{4}|&\lesssim \sum_{K\in {\cal T}_{h}}\norm{\nabla^{2}(w-{\cal E}_{K}(w))}_{K}\norm{\nabla^{2}(\Psi-{\cal E}_{K}(\Psi))}_{K}\lesssim
 \sum_{K\in {\cal T}_{h}}\norm{w-{\cal E}_{K}(w)}_{\#,K}\norm{\Psi-{\cal E}_{K}(\Psi)}_{\#,K}.
\end{align}
An application of \eqref{eq:w.EK.bound} and Theorem~\ref{lm1_in_dis_tr_11}$(c)$ reveals
\begin{align*}
\norm{w-{\cal E}_{K}(w)}_{\#,K} \lesssim \begin{cases}h\big(|{w}|_{H^{3}({\cal T}_{h})}+h^{\beta}|{w}|_{H^{3+\beta}({\cal T}_{h})}\big),\; \; \beta =\min(r-1,1) & \text{ for } k=0,
\\ h^{k+1}|{w}|_{H^{k+3}({\cal T}_{h})}& \text{ for } k\geq 1.
\end{cases}      
\end{align*}
Substitute this bound and \eqref{eq:psi.EK.bound} in~\eqref{eq:T4.int.bound} to obtain
\begin{align*}
|T_{4}| \lesssim \begin{cases}h^2\big(|{w}|_{H^{3}({\cal T}_{h})}+h^{\beta}|{w}|_{H^{3+\beta}({\cal T}_{h})}\big)\norm{\varphi_{h}},\; \; \beta =\min(r-1,1) & \text{ for } k=0,
\\ h^{k+3}|{w}|_{H^{k+3}({\cal T}_{h})}\norm{\varphi_{h}}& \text{ for } k\geq 1.
\end{cases}    
\end{align*}

\medskip\noindent
{{\text{{\it Step 7. (control of $T_{5}$)}}}}
A Cauchy--Schwarz inequality shows
\begin{align}\label{eq:T5.int.bound}
|T_{5}| & \leq \norm{ \Delta^{2}w-\Pi_{h}^{k+2}( \Delta^{2}w)}\norm{\Pi_{h}^{k+2}(\Psi)-\Psi}.
\end{align}
First, we estimate for $k=0.$ The triangle inequality, $L^2$-boundedness property of $\Pi^{2}_h$ and Theorem~\ref{lm1_in_dis_tr_11}$(d)$ yield
$$\norm{ \Delta^{2}w-\Pi_{h}^{2}( \Delta^{2}w)}\norm{\Pi_{h}^{k+2}(\Psi)-\Psi}\lesssim \left(\norm{ \Delta^{2}w}+\norm{\Pi_{h}^{2}( \Delta^{2}w)}\right)\norm{\Pi_{h}^{k+2}(\Psi)-\Psi}\lesssim  h^{2}\norm{\Delta^{2}w}|\Psi|_{H^{2}(\Omega)}.$$
Next, for $k\geq1,$ again by invoking Theorem~\ref{lm1_in_dis_tr_11}$(d)$~(for both the terms), we obtain
$$\norm{ \Delta^{2}w-\Pi_{h}^{k+2}( \Delta^{2}w)}\norm{\Pi_{h}^{k+2}(\Psi)-\Psi} \lesssim h^{k+3}|\Delta^2{w}|_{H^{k-1}({\cal T}_{h})}|\Psi|_{H^{4}(\Omega)} \;\; \text{ with } w\in H^{k+3}({\cal T}_{h}).$$
These two above bounds together with \eqref{dual_1} in \eqref{eq:T5.int.bound}  establish
\begin{align*}
|T_{5}| \lesssim \begin{cases}h^2\norm{\Delta^{2}w}\norm{\varphi_{h}} & \text{ for } k=0,
\\ h^{k+3}|\Delta^2{w}|_{H^{k-1}({\cal T}_{h})}\norm{\varphi_{h}}& \text{ for } k\geq 1.
\end{cases}    
\end{align*}

\medskip\noindent
{{\text{{\it Step 8. (control of $T_{6}$)}}}} 
From a Cauchy--Schwarz inequality and \eqref{eq:stab.bound}, we can deduce that
\begin{align*}
|T_{6}| \lesssim \sum_{K\in {\cal T}_{h}} \norm{w-\Pi^{k+2}_{K}(w)}_{\#,K} \norm{\Psi-\Pi^{k+2}_{K}(\Psi)}_{\#,K}.
% \lesssim h^{2}\sum_{K\in {\cal T}_{h}}\norm{w-\Pi^{k+2}_{K}(w)}_{\#,K} \norm{\varphi_{h}}.
\end{align*}
Now, using arguments analogous to those employed in the estimate of $T_{4}$ in Step~6, we obtain
\begin{align*}
|T_{6}| \lesssim \begin{cases}h^2\big(|{w}|_{H^{3}({\cal T}_{h})}+h^{\beta}|{w}|_{H^{3+\beta}({\cal T}_{h})}\big)\norm{\varphi_{h}},\; \; \beta =\min(r-1,1) & \text{ for } k=0,
\\ h^{k+3}|{w}|_{H^{k+3}({\cal T}_{h})}\norm{\varphi_{h}}& \text{ for } k\geq 1.
\end{cases}    
\end{align*}

\medskip\noindent
{{\text{{\it Step 9. (control of $T_{7}$)}}}} 
A Cauchy--Schwarz inequality, definition of $\norm{\bullet}_{a,h}$, and \eqref{eq:stab.bound} reveal
\begin{align*}
 |T_{7}| \leq {\cal S}_{h}(\widehat{\varphi}_{h},\widehat{\varphi}_{h})^{\frac{1}{2}}{\cal S}_{h}(\widehat{I}^{k}_{h}(\Psi),\widehat{I}^{k}_{h}(\Psi))^{\frac{1}{2}}\lesssim \norm{\widehat{\varphi}_{h}}_{a,h}\sum_{K\in {\cal T}_{h}} \norm{\Psi-\Pi^{k+2}_{K}(\Psi)}_{\#,K}.
 \end{align*}
 Now, using arguments in \eqref{eq:psi.EK.bound} yield
$\displaystyle |T_{7}|\lesssim  \begin{cases}
h\norm{\widehat{\varphi}_{h}}_{a,h}\norm{\varphi_{h}}&\text{for}\; k=0,\\
h^{2}\norm{\widehat{\varphi}_{h}}_{a,h}\norm{\varphi_{h}}& \text{for}\; k\geq 1.
 \end{cases} $
 
% where Lemma \ref{bd_elli_proj} is used in the second  step and  Theorem \ref{lm1_in_dis_tr_11}$(d)$ along with \eqref{dual_1} is used in last step.

\noindent\medskip
{{\text{{\it Step~10. (final bound)}}}} Utilize the bounds from Steps 3-9 in \eqref{eq:sum}  to obtain
\begin{align*}
\norm{\varphi_{h}}\lesssim \begin{cases}
h\norm{\widehat{\varphi}_{h}}_{a,h}+h^{2}\left(\norm{ \Delta^{2}w}+|{w}|_{H^{3}({\cal T}_{h})}+h^{\beta}|{w}|_{H^{3+\beta}({\cal T}_{h})}\right),\; \; \beta =\min(r-1,1) &\text{for}\;\; k=0,\\
h^{2}\norm{\widehat{\varphi}_{h}}_{a,h}+h^{k+3}\big(|{w}|_{H^{k+3}({\cal T}_{h})}+|\Delta^2{w}|_{H^{k-1}({\cal T}_{h})}\big)&\text{for}\;\; k\geq 1.
\end{cases}
\end{align*}
The above inequality and Theorem~\ref{lm1_in_dis_tr_11}$(c)$ concludes the  of the theorem.  \qed
\section{Newmark scheme}\label{sect_6}
The stability  of  Newmark scheme is established in Subsection~\ref{sub_st}. The proofs of Lemma \ref{P1 lemma_on_norm_initial}, Theorem \ref{P1 implicit_Th2}, and Theorem~\ref{err_L2_new} that discuss initial error bounds, energy estimates, and  $L^2$ error bounds, respectively, for the Newmark scheme are discussed in Subsection \ref{subsec-error-new}.

\subsection{Stability}\label{sub_st}
This subsection proves a bound that is uniform in time (in terms of solution at initial 2 time steps) for the solution to the Newmark scheme. In finite dimensions, stability ensures the well-posedness of the discrete problem  \eqref{hwave_algorithm}.

{ The next theorem establishes stability result and its proof relies on the energy method combined with the discrete Gronwall lemma. The key idea is to choose a suitable test function in the discrete scheme that leads to an {\it energy balance equation}. A summation over time steps yields a global energy identity. The forcing term is then controlled via Cauchy--Schwarz and Young's inequalities to absorb the  energy term at the final time level to the left-hand side, leaving an inequality which fits the discrete Gronwall lemma setting.}

\medskip \noindent 
\uline{\textbf{Proof of Theorem \ref{P2 full_stabilty}.}}
The choice of test function $\widehat{v}_{h}=\widehat{u}_{h}^{n+1}-\widehat{u}_{h}^{n-1} \in \widehat{V}_h^0$  in \eqref{hwave_algorithm} leads to
\begin{equation}
    (\bar{\partial}_{t}^{2}{u}_{h}^{n},{u}_{h}^{n+1}-{u}_{h}^{n-1})+a_{h}(\widehat{u}_{h}^{n,{1}/{4}},\widehat{u}_{h}^{n+1}-\widehat{u}_{h}^{n-1})=(f^{n,1/4},{u}_{h}^{n+1}-{u}_{h}^{n-1}).\label{stab}
\end{equation}
The definitions in \eqref{temp_dis} and  elementary manipulations show that
\begin{align}
    (\bar{\partial}_{t}^{2}{u}_{h}^{n},{u}_{h}^{n+1}-{u}_{h}^{n-1})&=(\bar{\partial}_{t}{u}_{h}^{n+1/2}-{\bar{\partial}_{t}u}_{h}^{n-1/2},\bar{\partial}_{t}{u}_{h}^{n+1/2}+{\bar{\partial}_{t}u}_{h}^{n-1/2}), \label{stab_ist}\\
    a_{h}(\widehat{u}_{h}^{n,{1}/{4}},\widehat{u}_{h}^{n+1}-\widehat{u}_{h}^{n-1})&=a_h(\widehat{u}_{h}^{n+1/2}+\widehat{u}_{h}^{n-1/2},\widehat{u}_{h}^{n+1/2}-\widehat{u}_{h}^{n-1/2}),\label{stab_2}\\
    (f^{n,1/4},{u}_{h}^{n+1}-{u}_{h}^{n-1})&=\Delta t (f^{n,1/4},\bar{\partial}_{t}{u}_{h}^{n+1/2}+{\bar{\partial}_{t}u}_{h}^{n-1/2}) \label{stab_3}.
\end{align}
A combination of \eqref{stab_ist}-\eqref{stab_3} in \eqref{stab} with linearity of $a_h(\bullet,\bullet)$ yields
\begin{eqnarray}
  \norm{\bar{\partial}_{t}{u}_{h}^{n+{1}/{2}}}^{2}-\norm{\bar{\partial}_{t}{u}_{h}^{n-{1}/{2}}}^{2}+ a_{h}(\widehat{u}_{h}^{n+{1}/{2}},\widehat{u}_{h}^{n+{1}/{2}})-a_{h}(\widehat{u}_{h}^{n-{1}/{2}},\widehat{u}_{h}^{n-{1}/{2}})=\Delta t (f^{n,1/4},\bar{\partial}_{t}{u}_{h}^{n+1/2}+{\bar{\partial}_{t}u}_{h}^{n-1/2}).\nonumber
\end{eqnarray}
A summation from $n=1,\cdots,m,$ with $1\le m\leq N-1,$ and an application of \eqref{norm} leads to  the {\it energy balance equation} as
%\end{equation*}
\begin{eqnarray}
  &&\norm{\bar{\partial}_{t}{u}_{h}^{m+{1}/{2}}}^{2}+\norm{\widehat{u}_{h}^{m+{1}/{2}}}_{a,h}^{2}= \norm{\bar{\partial}_{t}{u}_{h}^{{1}/{2}}}^{2}+\norm{\widehat{u}_{h}^{{1}/{2}}}_{a,h}^{2}+\Delta t\sum^{m}_{n=1}(f^{n,1/4},\bar{\partial}_{t}{u}_{h}^{n+1/2}+{\bar{\partial}_{t}u}_{h}^{n-1/2}). \label{stab4}
\end{eqnarray}
A Cauchy--Schwarz inequality $(f^{n,1/4},\bar{\partial}_{t}{u}_{h}^{n+1/2}+{\bar{\partial}_{t}u}_{h}^{n-1/2}) \le \norm{f^{n,1/4}}\norm{\bar{\partial}_{t}{u}_{h}^{n+1/2}+{\bar{\partial}_{t}u}_{h}^{n-1/2}}$ followed by Young's inequality $ab \le \frac{\epsilon}{2} a^2 +\frac{1}{2\epsilon} b^2$ with $a=\norm{f^{n,1/4}}$, $b=\norm{\bar{\partial}_{t}{u}_{h}^{n+1/2}+{\bar{\partial}_{t}u}_{h}^{n-1/2}}$ and $\epsilon =2T$ (applied to right-hand side of above equation) and the inequalities $\Delta t\sum_{n=1}^m\norm{f^{n,1/4}}^2 \le m \Delta t \norm{f}^2_{L^\infty(L^2(\Omega))}\le T\norm{f}^2_{L^\infty(L^2(\Omega))}$ show
\begin{align}
   \Delta t\sum^{m}_{n=1}(f^{n,1/4}&,\bar{\partial}_{t}{u}_{h}^{n+1/2}+{\bar{\partial}_{t}u}_{h}^{n-1/2})\le \Delta t  T\sum_{n=1}^m\norm{f^{n,1/4}}^2+\frac{\Delta t}{4T}\sum_{n=1}^m \norm{\bar{\partial}_{t}{u}_{h}^{n+1/2}+{\bar{\partial}_{t}u}_{h}^{n-1/2}}^2\nonumber\\
    &\le \Delta t T \sum_{n=1}^m\norm{f^{n,1/4}}^2+\frac{\Delta t}{2T}\norm{\bar{\partial}_{t}{u}_{h}^{1/2}}^2+\frac{\Delta t}{2T}\norm{\bar{\partial}_{t}{u}_{h}^{m+1/2}}^2+\frac{\Delta t}{T}\sum_{n=1}^{m-1} \norm{\bar{\partial}_{t}{u}_{h}^{n+1/2}}^2\nonumber\\
   &\le T^2 \norm{f}^2_{L^\infty(L^2(\Omega))}+\frac{1}{2}\norm{\bar{\partial}_{t}{u}_{h}^{{1}/{2}}}^{2}+\frac{1}{2}\norm{\bar{\partial}_{t}{u}_{h}^{m+{1}/{2}}}^{2}+\frac{\Delta t}{T} \sum_{n=1}^{m-1}\norm{\bar{\partial}_{t}{u}_{h}^{n+{1}/{2}}}^{2}
    \label{f1}
\end{align}
with $\Delta t/2 T\le 1/2$ (for second and third terms) in the last step. A combination of \eqref{stab4}-\eqref{f1} leads to
%\begin{align*}
%&\norm{\bar{\partial}_{t}{u}_{h}^{m+{1}/{2}}}^{2}+\norm{\widehat{u}_{h}^{m+{1}/{2}}}_{a,h}^{2}\le
%  \norm{\bar{\partial}_{t}{u}_{h}^{{1}/{2}}}^{2}+\norm{\widehat{u}_{h}^{{1}/{2}}}_{a,h}^{2}+T^2 \norm{f}^2_{{L^\infty}(L^2(\Omega))}+\frac{\Delta t}{2T}\norm{\bar{\partial}_{t}{u}_{h}^{1/{2}}}^{2}+\frac{\Delta t}{2T}\norm{\bar{\partial}_{t}{u}_{h}^{{m+1}/{2}}}^{2}+\frac{\Delta t}{T} \sum^{m-1}_{n=1}\norm{\bar{\partial}_{t}{u}_{h}^{n+1/2}}^2.
%\end{align*}
%with elementary manipulations in the last step.\\
%Next, the observations $m \Delta t\le T$ (applied to third term on right-hand side) and $\Delta t/2T \le 1/2$ (applied to fourth and fifth term on right-hand side) with elementary manipulations reveal
\begin{equation*}
   \frac{1}{2} \norm{\bar{\partial}_{t}{u}_{h}^{m+{1}/{2}}}^{2}+\norm{\widehat{u}_{h}^{m+{1}/{2}}}_{a,h}^{2} \le   \frac{3}{2}\norm{\bar{\partial}_{t}{u}_{h}^{{1}/{2}}}^{2}+\norm{\widehat{u}_{h}^{{1}/{2}}}_{a,h}^{2}+ T^2\norm{f}^2_{L^{\infty}(L^{2}(\Omega))}+\frac{\Delta t}{T} \sum^{m-1}_{n=1}\norm{\bar{\partial}_{t}{u}_{h}^{n+1/2}}^2.
\end{equation*}
Finally, an application of Lemma~\ref{Ch2-d-gronwall} and Remark~\ref{rem2.6} conclude the proof. \qed
\subsection{Error estimates}\label{subsec-error-new}
In this subsection, we first introduce the error function at time step $t_n$. The initial error bounds are established in Lemma~\ref{P1 lemma_on_norm_initial} followed by the energy and $L^2$ estimates in Theorem~\ref{P1 implicit_Th2} and Theorem~\ref{err_L2_new}, respectively, for Newmark scheme.

%To derive the error estimate, the error function is defined as:
%Let  $\widehat{E}_{h}(u_{0})=\widehat{I}_{h}^{k}(u_{0}).$
\medskip
\noindent
Let   $\widehat{e}^{n}_{h}:=\widehat{I}^{k}_{h}(u^{n})-\widehat{u}^{n}_{h}=({e}^{n}_{h},{e}^{n}_{{\cal F}_{h}},{\kappa}^{n}_{{\cal F}_{h}})$ with $\widehat{e}^{n}_{h}|_K= \widehat{I}^{k}_{K}(u^{n})-\widehat{u}^{n}_{K} :=\widehat{e}^{n}_{K}=({e}^{n}_{K},{e}^{n}_{{\cal F}_{K}},{\kappa}^{n}_{{\cal F}_{K}})$ for any $n=0,1,\cdots,N$. Split the error as 
\begin{align} 
%\textcolor{red}{
%u^{n}-\widehat{u}_{h}^{n}:=\text{\footnote{Do we need to ristrict the interpolation on cell only for first sum other how is the first sum compatible}}(u^{n}-\widehat{I}^{k}_{h}(u^{n}))}+\widehat{e}^{n}_{h} \; \text{ with } 
\widehat{e}^{n}_{h}:=\widehat{\theta}^{n}_{h}+\widehat{\rho}^{n}_{h},\;\text{where}\; \widehat{\theta}^{n}_{h}=\widehat{I}^{k}_{h}(u^{n})-\widehat{E}_{h}(u^{n}) \: \text{and }
\widehat{\rho}^{n}_{h}=\widehat{E}_{h}(u^{n})-\widehat{u}_{h}^{n}.\label{split_err_eq1}
\end{align}
 Recall that ${\cal {E}}_{K}:={\cal R}_{K}\circ \widehat{I}^{k}_{K}$ and introduce a local split for the reconstruction error as
\begin{equation}
u^{n}-{\cal R}_{K}(\widehat{u}^{n}_{K})= {\cal R}_{K}(\widehat{e}^{n}_{K})+u^{n}-{\cal {E}}_{K}({u}^{n}).\label{Recons_err_split}
\end{equation}
{The next lemma establishes the initial error bounds, and its proof relies on the energy method combined with approximation properties of the interpolation operator. We derive an error equation by substituting the error split in \eqref{split_err_eq1} into the discrete scheme and utilizing the fact that the exact solution satisfies the PDE at each time step. Then a suitable test function leads to an energy-type identity for the initial error. The resulting terms are then controlled via Cauchy--Schwarz and Young's inequalities to absorb the error norms at the final time level into the left-hand side, and the remaining terms are bounded using the approximation properties of the interpolation, projection operators, and a truncation error estimate yielding the desired initial error bounds.}

\medskip \noindent 
\uline{\textbf{Proof of the Lemma \ref{P1 lemma_on_norm_initial}.}}  
The definition of $\widehat{e}^{n}_{h}$ from above followed by \eqref{inter}  and  \eqref{pi-pro}  show, for  any $\widehat{v}_h\in \widehat{V}_h^0 $
%For any $\widehat{v}_h\in \widehat{V}_h^0 $ in  \eqref{hwave_int_cond} and the\eqref{split_err_eq1}  show that 
\begin{align*}
 {2}{(\Delta t)^{-1}}(\bar{\partial}_t e_h^{1/2},v_h )&+ a_h( \widehat{e}_h ^{1/2},\widehat{v}_h)={2}{(\Delta t)^{-1}}(\bar{\partial}_t \Pi^{k+2}_h(u^{1/2})-\bar{\partial}_t u_h^{1/2},v_h) +a_h(\widehat{I}^{k}_h(u^{1/2})-{\widehat{u}}_h^{1/2},\widehat{v}_h) \nonumber\\
&\qquad={2}{(\Delta t)^{-1}}\big[(\bar{\partial}_t u^{1/2},v_h) -(\bar{\partial}_t u_h^{1/2},v_h)\big]+a_h(\widehat{I}^{k}_h(u^{1/2}),\widehat{v}_h) -a_h({\widehat{u}}_h^{1/2},\widehat{v}_h)\nonumber\\
 &\qquad= {2}{(\Delta t)^{-1}}(\bar{\partial}_t u^{1/2},v_h)-(f^{1/2} +{2}{(\Delta t)^{-1}} {u_1},v_h )+a_h( \widehat{I}^{k}_h(u^{1/2}),\widehat{v}_h)\nonumber
\end{align*}
with
%${2}{(\Delta t)^{-1}}(\bar{\partial}_{t}u_{h}^{1/{2}},v_{h})+a_{h}(\widehat{u}_{h}^{1/{2}},\widehat{v}_{h})=(f^{1/{2}}+{2}{(\Delta t)^{-1}}{ u_{1}},v_{h})$ from 
\eqref{hwave_int_cond} in the last identity. The identities $ \widehat{I}^{k}_h (u^{1/2})=\widehat{\theta}_h^{1/2}+\widehat{E}_h(u^{1/2})$ from \eqref{split_err_eq1} and $a_h( \widehat{E}_h(u^{1/2}),\widehat{v}_h)=(\Delta^2 u^{1/2},{v}_h)$  from \eqref{proj_op_1}  reveal
$$a_h( \widehat{I}^{k}_h (u^{1/2}),\widehat{v}_h)=a_h(\widehat{\theta}_h^{1/2},\widehat{v}_h)+(\Delta^2 u^{1/2},{v}_h).$$
Moreover, since $u$ satisfies \eqref{P1 strong_form}, it follows that
$$(\Delta^2 u^{1/2},{v}_h)=(f^{1/2}-u_{tt}^{1/2},{v}_h).$$
A combination of the last three displayed identities shows that for all $\widehat{v}_{h} \in \widehat{V}_h^0$
\begin{align}
   {2}{(\Delta t)^{-1}}(\bar{\partial}_t e_h^{1/2},v_h )+ a_h( \widehat{e}_h ^{1/2},\widehat{v}_h)&=
  %(\frac{2}{\Delta t}(\bar{\partial}_t \Pi^{k+2}_h u^{1/2}-\textcolor{blue}{u_1})
 ( {2}{(\Delta t)^{-1}}(\bar{\partial}_t u^{1/2}-{u_1}) -u_{tt}^{1/2},v_h)+a_h(\widehat{\theta}_h^{1/2},\widehat{v}_h)\nonumber\\
&=(\widetilde{R}^0,v_h)+a_h(\widehat{\theta}_h^{1/2},\widehat{v}_h)\label{initial_eqn}
\end{align}
%\footnote{THE $R_0$ IS NOT EXACTLY THIS, then explain, now it is }
with $\widetilde{R}^0:={2}{(\Delta t)^{-1}}(\bar{\partial}_t u^{1/2}-{u_1}) -u_{tt}^{1/2}$.

\medskip \noindent
%Since $\widehat{I}^{k}_{h}(u_{0})=\widehat{u}_{h}^{0}$, so $\widehat{e}_h^0=0$ and hence $\frac{2}{\Delta t}\widehat{e}_h^{1/2}=\frac{1}{\Delta t}\widehat{e}_h^{1}=\bar{\partial}_t \widehat{e}_h^{1/2}$.
Recall $\widehat{e}_h^0=0$ from \eqref{hwave_int_cond} (which leads to $\widehat{e}_h^{1/2}=\frac{\Delta t}{2}\bar{\partial}_t \widehat{e}_h^{1/2}$). This, 
the choice of test function $\widehat{v}_h=\widehat{e}^{1/2}_h$ in  \eqref{initial_eqn}, and  \eqref{norm} yield
\begin{equation}
\norm{\bar{\partial}_t e_h^{1/2}}^{2}+ \norm{\widehat{e}_h ^{1/2}}_{a,h}^{2}=\frac{\Delta t}{2}(\widetilde{R}^0,\bar{\partial}_t e_h^{1/2})+a_h(\widehat{\theta}_h^{1/2},\widehat{e}^{1/2}_{h}).\label{ref_eh_bd}
\end{equation}
A Cauchy--Schwarz inequality and the continuity of $a_h(\bullet,\bullet)$ from Lemma \ref{bdd_lmma_a_T} reveals
\begin{equation*}
  \frac{\Delta t}{2}(\widetilde{R}^0,\bar{\partial}_t e_h^{1/2})+a_h(\widehat{\theta}_h^{1/2},\widehat{e}^{1/2}_{h})  \le \frac{\Delta t}{2}\norm{\widetilde{R}^0}\norm{\bar{\partial}_t e_h^{1/2}}+ \alpha\norm{\widehat{\theta}_{h}^{1/2}}_{a,h}\norm{\widehat{e}^{1/2}_h}_{a,h}.\label{cs}
\end{equation*}
An appeal to Young's inequality with $a=\Delta t\norm{\widetilde{R}^0}$ (resp. $a=\alpha\norm{\widehat{\theta}_{h}^{1/2}}_{a,h}$), $b=\norm{\bar{\partial}_t e_h^{1/2}}$ (resp. $b=\norm{\widehat{e}^{1/2}_h}_{a,h}$) and $\epsilon=1$ (resp. $\epsilon=2$)  for the first (resp. second) term on the right side of above leads to
\begin{equation}
    \frac{\Delta t}{2}\norm{\widetilde{R}^0}\norm{\bar{\partial}_t e_h^{1/2}}+ \alpha \norm{\widehat{\theta}_h^{1/2}}_{a,h}\norm{\widehat{e}^{1/2}_h}_{a,h} \le \frac{{\Delta t}^2}{4}\norm{\widetilde{R}^0}^2+\frac{1}{4}\norm{\bar{\partial}_t e_h^{1/2}}^2+\alpha^2\norm{\widehat{\theta}_h^{1/2}}_{a,h}^2+\frac{1}{4}\norm{\widehat{e}^{1/2}_h}_{a,h}^2.\label{yng}
\end{equation}
A combination of \eqref{ref_eh_bd}-\eqref{yng} 
%and the bound $\norm{\widetilde{R}^0}^2 \lesssim (\Delta t)^2\norm{u_{ttt}}^2_{L^\infty(0,t_1;L^2{(\Omega)})}$ from Lemma~\ref{Sec_lma_1}(a)
yields
\begin{align}
\frac{3}{4}\norm{\bar{\partial}_t e_h^{1/2}}^{2}+\frac{3}{4} \norm{\widehat{e}_h ^{1/2}}_{a,h}^{2}&\le \frac{{\Delta t}^2}{4}\norm{\widetilde{R}^0}^2+\alpha^2\norm{\widehat{\theta}_h^{1/2}}_{a,h}^2
.\label{ref_1_eh_bd_1}
\end{align}
%\label{ref_2_bd_ah}
%\end{equation}
Recall that $\widehat{e}^{1/2}_{h}:=\widehat{I}^{k}_{h}(u^{1/2})-\widehat{u}^{1/2}_{h}$. 
The identity $\bar{\partial}_tu^{1/2}-\bar{\partial}_tu_h^{1/2}=\bar{\partial}_t e_h^{1/2}+\bar{\partial}_tu^{1/2}-\Pi_h^{k+2}(\bar{\partial}_tu^{1/2})$ follows  from this and \eqref{inter}. Thus a  triangle inequality shows 
\begin{align}
   \norm{\bar{\partial}_tu^{1/2}-\bar{\partial}_tu_h^{1/2}}^{2} \le 2\big(\norm{\bar{\partial}_t e_h^{1/2}}^{2}+ \norm{\bar{\partial}_tu^{1/2}-\Pi_h^{k+2}(\bar{\partial}_tu^{1/2})}^{2}\big).\label{ref_2_bd_ah2}
\end{align}
From \eqref{Recons_err_split}, $u^{1/2}-{\cal R}_{K}(\widehat{u}_h ^{1/2})={\cal R}_{K}(\widehat{e}_{K}^{1/2})+u^{1/2}-{\cal E}_{K}({u} ^{1/2})$.  This,  an application of triangle inequality, the inequality $
\sum_{K\in {\cal T}_{h}}\norm{\nabla^{2}{\cal R}_{K}(\widehat{e}_{K}^{1/2})}_{K}^2\leq \norm{\widehat{e}_h^{1/2}}_{a,h}^2$ from  definition of bilinear from ${a_h}(\bullet,\bullet)$, and some elementary algebra reveal
\begin{align}
    \sum_{K\in {\cal T}_{h}}\norm{\nabla^{2}(u^{1/2}-{\cal R}_{K}(\widehat{u}_h ^{1/2}))}^{2}_{K} \le 2\big(\norm{\widehat{e}_h^{1/2}}_{a,h}^2+\sum_{K\in {\cal T}_{h}}\norm{\nabla^{2}(u^{1/2}-{\cal E}_{K}({u} ^{1/2}))}^{2}_{K}\big).\label{ref_2_bd_ah}
\end{align}
Sum up  \eqref{ref_2_bd_ah} and \eqref{ref_2_bd_ah2} then utilize \eqref{ref_1_eh_bd_1} to obtain
\begin{align}
&\norm{\bar{\partial}_tu^{1/2}-\bar{\partial}_tu_h^{1/2}}^{2}+\sum_{K\in {\cal T}_{h}}\norm{\nabla^{2}(u^{1/2}-{\cal R}_{K}(\widehat{u}_h ^{1/2}))}^{2}_{K}\nonumber\\
&\le \frac{2{\Delta t}^2}{3}\norm{\widetilde{R}^0}^2 +\frac{8\alpha^2}{3}\norm{\widehat{\theta}_h^{1/2}}_{a,h}^2+2\norm{\bar{\partial}_tu^{1/2}-\Pi_h^{k+2}(\bar{\partial}_tu^{1/2})}^{2}+2\sum_{K\in {\cal T}_{h}}\norm{\nabla^{2}(u^{1/2}-{\cal E}_{K}({u} ^{1/2}))}^{2}_{K}.\label{firs_last}
\end{align}
The remainder of the proof derives the bounds of the terms on the right-hand side. An application of  Lemma~\ref{Sec_lma_1}(a) reveals
\begin{equation}
\norm{\widetilde{R}^0}^2= \norm{{2}{(\Delta t)^{-1}}(\bar{\partial}_t u^{1/2}-{u_1}) -u_{tt}^{1/2}}^2 \lesssim (\Delta t)^2\norm{u_{ttt}}^2_{L^\infty(0,t_1;L^2{(\Omega)})}.\label{new_isst}
\end{equation}
     The definition \eqref{P3 phi_c1 }  and application of  triangle inequality shows $\|\widehat{\theta}_h^{1/2}\|_{a,h}^2 \le\frac{1}{2}\big(\|\widehat{\theta}_h^{0}\|_{a,h}^2+\|\widehat{\theta}_h^{1}\|_{a,h}^2\big)$. This in combination with the definition of $\widehat{\theta}_h^{1/2}$ from \eqref{split_err_eq1} and  Theorem~\ref{lm1_in_dis_tr_11}(c) yields
    \begin{align}
\norm{\widehat{\theta}_h^{1/2}}_{a,h}^2\lesssim\begin{cases}
 h^{2}\norm{u}^2_{L^\infty(0,t_1;H^{3}(\mathcal{T}_h))}+h^{2(1+\beta)}\norm{u}^2_{L^\infty(0,t_1;H^{3+\beta}(\mathcal{T}_h))}\;&\text{ for }\; k = 0,\\[6pt]
h^{2(k+1)}\norm{u}^2_{L^\infty(0,t_1;H^{k+3}(\mathcal{T}_h))}\;&\text{ for }\; k \ge 1.
 \end{cases}\label{theta-def}
\end{align}
Utilize  \eqref{P3 phi_c2} to argue $\bar{\partial}_tu^{1/2}=(\Delta t)^{-1}\int_{0}^{t_{1}}u_t\dt$ and $\Pi_h^{k+2}(\bar{\partial}_tu^{1/2})=(\Delta t)^{-1}\int_{0}^{t_{1}}\Pi_h^{k+2}u_t\dt$. This  followed by $\|\int_{0}^{t_{1}}(u_t-\Pi_h^{k+2}u_t) \dt\|\le \int_{0}^{t_{1}}\|u_t-\Pi_h^{k+2}u_t\| \dt$ and the bounds from  Theorem \ref{lm1_in_dis_tr_11}(d) lead to
\begin{align}
    &\norm{\bar{\partial}_tu^{1/2}-\Pi_h^{k+2}(\bar{\partial}_tu^{1/2})}^{2}\lesssim h^{2(k+1)} {(\Delta t)^{-2}}\big(\int_{0}^{t_1}\norm{u_t}^2_{H^{k+1}(\mathcal{T}_h)}\dt\big)^2 \lesssim  h^{2(k+1)}\norm{u_t}^2_{L^\infty(0,t_1;H^{k+1}(\mathcal{T}_h))}.\label{new_ll2}
    \end{align}
Apply \eqref{P3 phi_c1 } and  triangle inequality to obtain $\norm{\nabla^{2}(u^{1/2}-{\cal E}_{K}({u}^{1/2}))}_K^2 \le \frac{1}{2}\big(\norm{\nabla^{2}(u^{0}-{\cal E}_{K}({u}^{0}))}_K^2+\norm{\nabla^{2}(u^{1}-{\cal E}_{K}({u}^{1}))}_K^2\big)$. This and  the bounds from  Theorem \ref{lm1_in_dis_tr_11}(d)  reveal
\begin{align}
    &\sum_{K\in {\cal T}_{h}}\norm{\nabla^{2}(u^{1/2}-{\cal E}_{K}({u}^{1/2}))}_K^2\lesssim  h^{2(k+1)}\norm{u}^2_{L^\infty(0,t_1;H^{k+3}(\mathcal{T}_h))}.\label{new_1last}
    \end{align}    
A combination \eqref{firs_last}-\eqref{new_1last}  leads to
\begin{align*}
&\norm{\bar{\partial}_tu^{1/2}-\bar{\partial}_tu_h^{1/2}}^{2}+\sum_{K\in {\cal T}_{h}}\norm{\nabla^{2}(u^{1/2}-{\cal R}_{K}(\widehat{u}_h ^{1/2}))}^{2}_{K}
\lesssim  \begin{cases}
 h^2+h^{2(1+\beta)}+(\Delta t)^{4}\;&\text{ for }\; k = 0,\\
h^{2(k+1)}+(\Delta t)^{4}\;&\text{ for }\; k \ge 1.
 \end{cases} 
\end{align*}
Elementary manipulations conclude the proof.
%along with the bounds of $\norm{\widetilde{R}^0}^2$ from Lemma~\textcolor{blue}{truncation lemma} and $\norm{\widehat{\theta}_h^{1/2}}_{a,h}^2$ from \textcolor{blue}{approximation properties} conclude the proof.
\qed
%\end{proof}

%\begin{proof}[\textbf{Proof}]
\medskip
\noindent
\uline{\textbf{Error equation.}} The definitions  $\widehat{e}^{n}_{h}=\widehat{I}^{k}_{h}(u^{n})-\widehat{u}^{n}_{h}$,  \eqref{inter}, and \eqref{pi-pro} show that $\text{for all }\widehat{v}_{h} \in \widehat{V}_h^0$,
\begin{align*}
    (\bar{\partial}_{t}^{2}{e}_{h}^{n},{v}_{h})+a_{h}(\widehat{e}^{n,{1}/{4}},\widehat{v}_{h})&=(\bar{\partial}_{t}^{2}{\Pi}^{k+2}_{h}({u}^{n})-\bar{\partial}_{t}^{2}{u}_h^{n},v_{h})+a_{h}(\widehat{I}^{k}_{h}(u^{n,{1}/{4}})-\widehat{u}^{n,1/4}_{h},\widehat{v}_{h})\\
    &=(\bar{\partial}_{t}^{2}{u}^{n},v_{h})-(\bar{\partial}_{t}^{2}{u}_h^{n},v_{h})+a_{h}(\widehat{I}^{k}_{h}(u^{n,{1}/{4}}),\widehat{v}_{h})-a_{h}(\widehat{u}^{n,1/4}_{h},\widehat{v}_{h})\\
    &=(\bar{\partial}_{t}^{2}u^n,v_{h})-(f^{n,1/4},{v}_{h})+a_{h}(\widehat{I}^{k}_{h}(u^{n,{1}/{4}}),\widehat{v}_{h})
\end{align*}
with \eqref{hwave_algorithm}  in the last identity. Utilize  $\widehat{I}^{k}_{h}(u^{n,{1}/{4}})=\widehat{\theta}_{h}^{n,{1}/{4}}+\widehat{E}_{h}(u^{n,{1}/{4}})$ from  \eqref{split_err_eq1} and \eqref{proj_op_1} to obtain $$a_{h}(\widehat{I}^{k}_{h}(u^{n,{1}/{4}}),\widehat{v}_{h})=a_{h}(\widehat{\theta}_{h}^{n,{1}/{4}},\widehat{v}_{h})+(\Delta^2 u^{n,1/4},{v}_{h}).$$
Moreover, since $u$ satisfies the PDE \eqref{P1 strong_form}, it follows that
\begin{equation*}
    (f^{n,1/4},{v}_{h})=(u_{tt}^{n,{1}/{4}},{v}_{h})+(\Delta^2 u^{n,1/4},{v}_{h}).
\end{equation*}
\begin{comment}
The HHO scheme \eqref{hwave_algorithm} and \eqref{hwave_err_eq_1}   leads to
\begin{equation}
(u_{tt}^{n,{1}/{4}}-\bar{\partial}_{t}^{2}{u}_{h}^{n},v_{h})-a_{h}(\widehat{u}_h^{n{1}/{4}},\widehat{v}_{h})+(\Delta^{2} u^{n,1/4},v_{h})=0.\nonumber
\end{equation}
The above equation along with the definitions \eqref{proj_op_1} and \eqref{split_err_eq1} implies that
\begin{align}
a_{h}(\widehat{\theta}^{n,{1}/{4}}_{h},\widehat{v}_{h})&=(u_{tt}^{n,{1}/{4}}-\bar{\partial}_{t}^{2}{u}_{h}^{n},v_{h})+a_{h}(\widehat{e}^{n,{1}/{4}},\widehat{v}_{h})\nonumber\\
&=(\bar{\partial}_{t}^{2}{I}^{k}_{h}({u}^{n})-\bar{\partial}_{t}^{2}{I}^{k}_{h}({u}^{n})+u_{tt}^{n,{1}/{4}}-\bar{\partial}_{t}^{2}{u}_{h}^{n},v_{h})+a_{h}(\widehat{e}^{n,{1}/{4}},\widehat{v}_{h}).\label{hwave_err_lm1_eq_1}
\end{align}
The definition of $\Pi_{h}^{k+2}$ and \eqref{P3 phi_c2} yields
\begin{align}
   (\bar{\partial}_{t}^{2}{I}^{k}_{h}({u}^{n}),v_{h})= (\bar{\partial}_{t}^{2}{\Pi}^{k+2}_{h}({u}^{n}),v_{h})
   &={\Delta t}^{-2}({{\Pi}^{k+2}_{h}({u}^{n+1})-2{\Pi}^{k+2}_{h}({u}^{n})+{\Pi}^{k+2}_{h}({u}^{n-1})},v_{h})\nonumber\\
   &={\Delta t}^{-2}({{u}^{n+1}-2{u}^{n}+{u}^{n-1}},v_{h})= (\bar{\partial}_{t}^{2}{u}^{n},v_{h}).\label{hwave_err_lm1_eq_2}
\end{align} 
\end{comment}
A combination of the last three displayed identities reveals that for all $\widehat{v}_{h} \in \widehat{V}_h^0$
\begin{equation}
(\bar{\partial}_{t}^{2}{e}_{h}^{n},{v}_{h})+a_{h}(\widehat{e}^{n,{1}/{4}}_h,\widehat{v}_{h})=(\bar{\partial}_{t}^{2}{u}^{n}-u_{tt}^{n,{1}/{4}},v_{h})+a_{h}(\widehat{\theta}^{n,{1}/{4}}_{h},\widehat{v}_{h})=(R^n,v_{h})+a_{h}(\widehat{\theta}^{n,{1}/{4}}_{h},\widehat{v}_{h})\label{erroreqn} 
\end{equation}
with ${R}^j:=\bar{\partial}_{t}^{2}{u}^{j}-u_{tt}^{j,{1}/{4}}$ for  any $j=1,2,\cdots,N-1.$

\medskip
\noindent
{
The next theorem establishes the energy error estimates for time levels greater than $t_1$ and its proof relies on the energy method combined with the discrete Gronwall lemma. We derive an error equation by substituting the error split into the discrete scheme  and utilize the fact that the exact solution satisfies the PDE at each time step. An appropriate  test function yields a key energy identity with initial error terms, temporal truncation error, and interpolation error. After consolidating all estimates, the discrete Gronwall lemma is applied to the resulting cumulative sum of energies at previous time steps.}

\medskip \noindent 
\uline{\textbf{Proof of the Theorem \ref{P1 implicit_Th2}.}}
The proof proceeds in four steps. A key identity which decomposes the right-hand side into initial error terms plus two terms $T_1$ 
and $T_2$ is derived in 
\textit{Step~1}. The initial error terms are bounded in Lemma~\ref{P1 lemma_on_norm_initial} while  $T_1$ 
and $T_2$ are estimated separately in \textit{Step~2} 
and \textit{Step~3}, respectively. The final bound is obtained in 
\textit{Step~4} by consolidating the preceding estimates.

\medskip
\noindent
{{\text{{\it Step 1. (key identity)}}}}  The choice of the test function $\widehat{v}_{h}=\widehat{e}_{h}^{n+1}-\widehat{e}_{h}^{n-1}$ in \eqref{erroreqn} leads to
\begin{eqnarray*}
    (\bar{\partial}_{t}^{2}{e}_{h}^{n},{e}_{h}^{n+1}-{e}_{h}^{n-1})+a_{h}(\widehat{e}_{h}^{n,{1}/{4}},\widehat{e}_{h}^{n+1}-\widehat{e}_{h}^{n-1})=({R}^n,{e}_{h}^{n+1}-{e}_{h}^{n-1})+a_{h}(\widehat{\theta}^{n,{1}/{4}}_{h},\widehat{e}_{h}^{n+1}-\widehat{e}_{h}^{n-1}).
\end{eqnarray*}
 Now, proceed similar as in Theorem~\ref{P2 full_stabilty} from \eqref{stab_ist}-\eqref{stab4} to obtain
\begin{equation}
\norm{\bar{\partial}_{t}e_{h}^{m+{1}/{2}}}^{2}+\norm{\widehat{e}_{h}^{m+{1}/{2}}}^{2}_{a,h}=\norm{\bar{\partial}_{t}e_{h}^{{1}/{2}}}^{2}+\norm{\widehat{e}_{h}^{{1}/{2}}}^{2}_{a,h}+T_{1}+T_{2},\label{P2 all}
\end{equation}
where $T_1:=\Delta t\displaystyle\sum_{n=1}^{m}(R^n,\bar{\partial}_{t}{e}_{h}^{n+{1}/{2}}+\bar{\partial}_{t}{e}_{h}^{n-{1}/{2}})$ and $T_2:=2\displaystyle\sum_{n=1}^{m}a_{h}(\widehat{\theta}^{n,{1}/{4}}_{h},\widehat{e}_{h}^{n+1/2}-\widehat{e}_{h}^{n-{1}/{2}})$.

\medskip
\noindent
{{\text{{\it Step 2. (control of $T_{1}$)}}}} %A Cauchy--Schwarz inequality $( R^n,\partial_{t}{e}_{h}^{n+{1}/{2}}+\partial_{t}{e}_{h}^{n-{1}/{2}})\le \norm{R^n}\norm{\bar{\partial}_{t}{e}_{h}^{n+1/2}+{\bar{\partial}_{t}e}_{h}^{n-1/2}}$ followed by Young's inequality $ab \le \frac{\epsilon}{2} a^2 +\frac{1}{2\epsilon} b^2$ with $a=\norm{R^n}$, $b=\norm{\bar{\partial}_{t}{e}_{h}^{n+1/2}+{\bar{\partial}_{t}e}_{h}^{n-1/2}}$ and $\epsilon =2T$ lead to
{Arguments analogous to \eqref{f1} with $f^{n,1/4}$ replaced by $R^n$ and $\widehat{u}_h^n$ replaced by $\widehat{e}_h^n$
%followed by an application of Lemma~\ref{Sec_lma_1}(c) 
yield
\begin{align*}
    T_1 %&\le \Delta t  \sum_{n=1}^m\norm{R^n}^2+\frac{\Delta t}{4T}\sum_{n=1}^m \norm{\bar{\partial}_{t}{e}_{h}^{n+1/2}+{\bar{\partial}_{t}e}_{h}^{n-1/2}}^2\nonumber\\
   % &\le \Delta t  \sum_{n=1}^m\norm{R^n}^2+\frac{\Delta t}{2T}\norm{\bar{\partial}_{t}{e}_{h}^{1/2}}^2+\frac{\Delta t}{2T}\norm{\bar{\partial}_{t}{e}_{h}^{m/2}}^2+\frac{\Delta t}{T}\sum_{n=1}^{m-1} \norm{\bar{\partial}_{t}{e}_{h}^{n+1/2}}^2\nonumber\\
    &\le \Delta t T \sum_{n=1}^m\norm{R^n}^2+\frac{1}{2}\norm{\bar{\partial}_{t}{e}_{h}^{{1}/{2}}}^{2}+\frac{1}{2}\norm{\bar{\partial}_{t}{e}_{h}^{m+{1}/{2}}}^{2}+\frac{\Delta t}{T} \sum_{n=1}^{m-1}\norm{\bar{\partial}_{t}{e}_{h}^{n+{1}/{2}}}^{2}.
    %&\le \frac{41}{1260}(\Delta t)^4\norm{u_{tttt}}^2_{L^2(L^2(\Omega))}+\frac{1}{2}\norm{\bar{\partial}_{t}{e}_{h}^{{1}/{2}}}^{2}+\frac{1}{2}\norm{\bar{\partial}_{t}{e}_{h}^{m+{1}/{2}}}^{2}+\frac{\Delta t}{T} \sum_{n=1}^{m-1}\norm{\bar{\partial}_{t}{e}_{h}^{n+{1}/{2}}}^{2}.
    %\label{T1}.
\end{align*}}
%Utilize the bounds for $R^n$ from \textcolor{blue}{trunation error lemma} and 
%In the last step, the observation $\Delta t/2 T\le 1/2$ (for second and third terms) is used.
%\begin{equation}
%    T_1 \le \textcolor{red}{C} \Delta t^{4}\norm{u_{tttt}}_{L^{2}(L^{2}(\Omega))}^{2} +\frac{1}{2}\norm{\bar{\partial}_{t}{e}_{h}^{{1}/{2}}}^{2}+\frac{1}{2}\norm{\bar{\partial}_{t}{e}_{h}^{m+{1}/{2}}}^{2}+\frac{\Delta t}{T} \sum_{n=1}^{m-1}\norm{\bar{\partial}_{t}{e}_{h}^{n+{1}/{2}}}^{2}\label{T1}.
%\end{equation}
{{\text{{\it Step 3. (control of $T_{2}$)}}}} An application of \eqref{sum_by_parts} and continuity of $a_h(\bullet,\bullet)$ from Lemma~\ref{bdd_lmma_a_T} yield
\begin{align}
\frac{1}{2}T_2&=\displaystyle\sum_{n=1}^{m}a_{h}(\widehat{\theta}^{n,{1}/{4}}_{h},\widehat{e}_{h}^{n+1/2}-\widehat{e}_{h}^{n-{1}/{2}})=a_{h}(\widehat{\theta}^{m,{1}/{4}}_{h},\widehat{e}_{h}^{m+{1}/{2}})-a_{h}(\widehat{\theta}^{1,{1}/{4}}_{h},\widehat{e}_{h}^{{1}/{2}})+\sum_{n=1}^{m-1}a_{h}(\widehat{\theta}^{n+1,{1}/{4}}_{h} -\widehat{\theta}^{n,{1}/{4}}_{h},\widehat{e}_{h}^{n+{1}/{2}})\nonumber\\
    &\le  \alpha \Big[\norm{\widehat{\theta}_h^{m,{1}/{4}}}_{a,h}\norm{\widehat{e}_h^{m+{1}/{2}}}_{a,h}+\norm{\widehat{\theta}_h^{1,{1}/{4}}}_{a,h}\norm{\widehat{e}_h^{{1}/{2}}}_{a,h}+\sum_{n=1}^{m-1}\norm{\widehat{\theta}^{n+1,{1}/{4}}_{h} -\widehat{\theta}^{n,{1}/{4}}_{h}}_{a,h}\norm{\widehat{e}_{h}^{n+{1}/{2}}}_{a,h}  \Big].\label{New_er}
\end{align}
Apply Young's inequality with  $(a,b,\epsilon)=( \alpha\norm{\widehat{\theta}_h^{m,{1}/{4}}}_{a,h},\norm{\widehat{e}_h^{m+{1}/{2}}}_{a,h},1/2)$, $( \alpha\norm{\widehat{\theta}_h^{1,{1}/{4}}}_{a,h},\norm{\widehat{e}_h^{{1}/{2}}}_{a,h},1/2)$, and 
$(\alpha \norm{\widehat{\theta}^{n+1,{1}/{4}}_{h} -\widehat{\theta}^{n,{1}/{4}}_{h}}_{a,h},\norm{\widehat{e}_{h}^{n+{1}/{2}}}_{a,h} ,\frac{\Delta t}{T} )$   to bound the first, second, and third terms, respectively, on the right-hand of last inequality as
\begin{align}
  & \alpha\norm{\widehat{\theta}_h^{m,{1}/{4}}}_{a,h}\norm{\widehat{e}_h^{m+{1}/{2}}}_{a,h}
  \le {\alpha^2}\norm{\widehat{\theta}_h^{m,{1}/{4}}}_{a,h}^2+\frac{1}{4}\norm{\widehat{e}_h^{m+{1}/{2}}}_{a,h}^2,\\
 &  \alpha\norm{\widehat{\theta}_h^{1,{1}/{4}}}_{a,h}\norm{\widehat{e}_h^{{1}/{2}}}_{a,h}\le{\alpha^2}\norm{\widehat{\theta}_h^{1,{1}/{4}}}_{a,h}^2+\frac{1}{4}\norm{\widehat{e}_h^{{1}/{2}}}_{a,h}^2,  \\
 &   \alpha \norm{\widehat{\theta}^{n+1,{1}/{4}}_{h} -\widehat{\theta}^{n,{1}/{4}}_{h}}_{a,h}\norm{\widehat{e}_{h}^{n+{1}/{2}}}_{a,h}  \le \frac{\alpha^2T}{2\Delta t} \norm{\widehat{\theta}^{n+1,{1}/{4}}_{h} -\widehat{\theta}^{n,{1}/{4}}_{h}}_{a,h}^2+\frac{\Delta t}{2T}\norm{\widehat{e}_h^{n+{1}/{2}}}_{a,h}^2.\label{thn}
\end{align}
Note that from \eqref{P3 phi_c1 },   
$\widehat{\theta}^{n+1,{1}/{4}}_{h} -\widehat{\theta}^{n,{1}/{4}}_{h}=\frac{1}{4}\big(\int_{t_{n}}^{t_{n+2}}{\widehat{\theta}_{ht}(t)}\dt+\int_{t_{n-1}}^{t_{n+1}}\widehat{\theta}_{ht}(t)\dt\big).$ This, an application of triangle inequality $\|\widehat{\theta}^{n+1,{1}/{4}}_{h} -\widehat{\theta}^{n,{1}/{4}}_{h}\|_{a,h}\le \frac{1}{4}\big(\int_{t_{n}}^{t_{n+2}}\|{\widehat{\theta}_{ht}(t)}\|_{a,h}\dt+\int_{t_{n-1}}^{t_{n+1}}\|\widehat{\theta}_{ht}(t)\|_{a,h}\dt\big)$, and  the  Cauchy--Schwarz inequality $\int_{t_{j}}^{t_{j+2}}\norm{\widehat{\theta}_{ht}(t)}_{a,h}\dt \le \sqrt{ 2\Delta t}\big(\int_{t_{j}}^{t_{j+2}}\norm{\widehat{\theta}_{ht}(t)}_{a,h}^2\dt \big)^{1/2}$ (for $j=n-1,n$) with some  elementary manipulations reveal
\begin{align}
    \norm{\widehat{\theta}^{n+1,{1}/{4}}_{h} -\widehat{\theta}^{n,{1}/{4}}_{h}}_{a,h}^2 \le \frac{\Delta t}{4}\big(\int_{t_{n}}^{t_{n+2}}\norm{\widehat{\theta}_{ht}(t)}_{a,h}^2\dt+\int_{t_{n-1}}^{t_{n+1}}\norm{\widehat{\theta}_{ht}(t)}_{a,h}^2\dt \big).\label{new_22}
\end{align}
A combination of \eqref{New_er}-\eqref{new_22} leads to 
\begin{align*}
&T_{2} \le 2{\alpha^2}\big[\norm{\widehat{\theta}_h^{m,{1}/{4}}}_{a,h}^2+\norm{\widehat{\theta}_h^{1,{1}/{4}}}_{a,h}^2\big]+\frac{1}{2}\norm{\widehat{e}_h^{m+{1}/{2}}}_{a,h}^2+\frac{1}{2}\norm{\widehat{e}_h^{{1}/{2}}}_{a,h}^2\nonumber\\&\qquad \qquad +\frac{\alpha^2T}{4}\sum_{n=1}^{m-1}\Big[\int_{t_{n}}^{t_{n+2}}\norm{\widehat{\theta}_{ht}(t)}_{a,h}^2\dt+\int_{t_{n-1}}^{t_{n+1}}\norm{\widehat{\theta}_{ht}(t)}_{a,h}^2\dt \Big]+\frac{\Delta t}{T}\sum_{n=1}^{m-1}\norm{\widehat{e}_h^{n+{1}/{2}}}_{a,h}^2.
\end{align*}
{{\text{{\it Step 4. (consolidation)}}}}
The bounds from Steps 2-3 in \eqref{P2 all}, estimates  from \eqref{ref_1_eh_bd_1},  and a rearrangement  of the terms reveal
\begin{align*}
  &\frac{1}{2}\big[\norm{\bar{\partial}_{t}e_{h}^{m+{1}/{2}}}^{2}+\norm{\widehat{e}_{h}^{m+{1}/{2}}}^{2}_{a,h}\big]\le \frac{{\Delta t}^2}{2}\norm{\widetilde{R}^0}^2+2\alpha^2\norm{\widehat{\theta}_h^{1/2}}_{a,h}^2+\Delta t T  \sum_{n=1}^m\norm{R^n}^2 
  %+\Delta t^{4}\norm{u_{tttt}}_{L^{2}(L^{2}(\Omega))}^{2}
  +2{\alpha^2}\big[\norm{\widehat{\theta}_h^{m,{1}/{4}}}_{a,h}^2 + \norm{\widehat{\theta}_h^{1,{1}/{4}}}_{a,h}^2\big]\\
  &\qquad +\frac{\alpha^2T}{4}\sum_{n=1}^{m-1}\Big[\int_{t_{n}}^{t_{n+2}}\norm{\widehat{\theta}_{ht}(t)}_{a,h}^2\dt+\int_{t_{n-1}}^{t_{n+1}}\norm{\widehat{\theta}_{ht}(t)}_{a,h}^2\dt \Big] +\frac{\Delta t}{T}\sum_{n=1}^{m-1}\norm{\widehat{e}_h^{n+{1}/{2}}}_{a,h}^2+\frac{\Delta t}{T}\sum_{n=1}^{m-1}\norm{\bar{\partial}_{t}{e}_{h}^{n+{1}/{2}}}^{2}.
\end{align*}
An application of  Lemma~\ref{Ch2-d-gronwall} with Remark~\ref{rem2.6} and a reordering of the terms lead  to%and the bound $\Delta t  \sum_{n=1}^m\norm{R^n}^2 \lesssim (\Delta t)^{4} \norm{u_{tttt}}^2_{L^2(L^2(\Omega))} $ from  Lemma~\ref{Sec_lma_1}(b) lead to
\begin{align}
\norm{\bar{\partial}_{t}e_{h}^{m+{1}/{2}}}^{2}+\norm{\widehat{e}_{h}^{m+{1}/{2}}}^{2}_{a,h}&\lesssim {\Delta t}^2\norm{\widetilde{R}^0}^2+\Delta t  \sum_{n=1}^m\norm{R^n}^2 +
  %+\Delta t^{4}\norm{u_{tttt}}_{L^{2}(L^{2}(\Omega))}^{2}
  \norm{\widehat{\theta}_h^{1/2}}_{a,h}^2
  %+\Delta t^{4}\norm{u_{tttt}}_{L^{2}(L^{2}(\Omega))}^{2}
  +\norm{\widehat{\theta}_h^{1,{1}/{4}}}_{a,h}^2+\norm{\widehat{\theta}_h^{m,{1}/{4}}}_{a,h}^2\nonumber\\
  &\qquad + \sum_{n=1}^{m-1}\Big[\int_{t_{n}}^{t_{n+2}}\norm{\widehat{\theta}_{ht}(t)}_{a,h}^2\dt+\int_{t_{n-1}}^{t_{n+1}}\norm{\widehat{\theta}_{ht}(t)}_{a,h}^2\dt \Big].\label{thm8_e_bd}
\end{align}
Analogous arguments with $1/2$ replaced by $m+1/2$ in  \eqref{ref_2_bd_ah2} and \eqref{ref_2_bd_ah}  yield
\begin{align*}
&\norm{\bar{\partial}_tu^{m+1/2}-\bar{\partial}_tu_h^{m+1/2}}^{2} \le 2\big(\norm{\bar{\partial}_t e_h^{m+1/2}}^{2}+\norm{\bar{\partial}_tu^{m+1/2}-\Pi_h^{k+2}(\bar{\partial}_tu^{m+1/2})}^{2}\big),\\
    &\sum_{K\in {\cal T}_{h}}\norm{\nabla^{2}(u^{m+1/2}-{\cal R}_{K}(\widehat{u}_h ^{m+1/2}))}^{2}_{K} \le 2\big(\norm{\widehat{e}_h^{m+1/2}}_{a,h}^2+\sum_{K\in {\cal T}_{h}}\norm{\nabla^{2}(u^{m+1/2}-{\cal E}_{K}({u} ^{m+1/2}))}^{2}_{K}\big).
\end{align*}
Sum up the last two displayed inequalities and utilize \eqref{thm8_e_bd} to show 
%, and  $\norm{\bar{\partial}_t(u^{m+1/2}-\Pi_h^{k+2}(u^{m+1/2}))}^{2} \lesssim \norm{(u_t-\Pi_h^{k+2}(u_t))}^{2} _{L^\infty(t_{m},t_{m+1};V_h)} $ from Taylor series expansion 
\begin{align}
&\norm{\bar{\partial}_t(u^{m+1/2}-u_h^{m+1/2})}^{2}+\sum_{K\in {\cal T}_{h}}\norm{\nabla^{2}(u^{m+1/2}-{\cal R}_{K}(\widehat{u}_h ^{m+1/2}))}^{2}_{K}\lesssim {\Delta t}^2\norm{\widetilde{R}^0}^2+\Delta t  \sum_{n=1}^m\norm{R^n}^2\nonumber\\
%&\lesssim \norm{\bar{\partial}_t(u^{m+1/2}-\Pi_h^{k+2}(u^{m+1/2}))}^{2}+\sum_{K\in {\cal T}_{h}}\norm{\nabla^{2}(u^{m+1/2}-{\cal E}_{K}({u} ^{m+1/2}))}^{2}_{K}+\norm{\bar{\partial}_t e_h^{m+1/2}}^{2}+ \norm{\widehat{e}_h ^{m+1/2}}_{a,h}^{2}\nonumber\\
&\quad+\norm{\widehat{\theta}_h^{1/2}}_{a,h}^2+\norm{\widehat{\theta}_h^{1,{1}/{4}}}_{a,h}^2+\norm{\widehat{\theta}_h^{m,{1}/{4}}}_{a,h}^2+ \sum_{n=1}^{m-1}\Big[\int_{t_{n}}^{t_{n+2}}\norm{\widehat{\theta}_{ht}(t)}_{a,h}^2\dt+\int_{t_{n-1}}^{t_{n+1}}\norm{\widehat{\theta}_{ht}(t)}_{a,h}^2\dt \Big]\nonumber\\
& \quad+ \norm{\bar{\partial}_tu^{m+1/2}-\Pi_h^{k+2}(\bar{\partial}_tu^{m+1/2})}^{2}+\sum_{K\in {\cal T}_{h}}\norm{\nabla^{2}(u^{m+1/2}-{\cal E}_{K}({u} ^{m+1/2}))}^{2}_{K}
.\label{last}
\end{align}
The remaining parts of the proof bounds the terms on the right-hand of \eqref{last}. Recall the definition  of $\widetilde{R}^0$ (resp. $R^n$) from \eqref{initial_eqn} (resp. \eqref{erroreqn}) and then  apply Lemma~\ref{Sec_lma_1}(a) (resp. Lemma~\ref{Sec_lma_1}(c))  to obtain
\begin{align*}
    {\Delta t}^2\norm{\widetilde{R}^0}^2+\Delta t  \sum_{n=1}^m\norm{R^n}^2 \lesssim (\Delta t)^{4} \big(\norm{u_{ttt}}^2_{L^\infty(L^2(\Omega))}+ \norm{u_{tttt}}^2_{L^2(L^2(\Omega))}\big) .
\end{align*}
Note that  $\norm{\widehat{\theta}_h^{j,{1}/{4}}}_{a,h}^2\le \frac{1}{2}\big(\norm{\widehat{\theta}_h^{j+{1}/{2}}}_{a,h}^2+\norm{\widehat{\theta}_h^{j-{1}/{2}}}_{a,h}^2\big)$ follows from \eqref{P3 phi_c1 } and triangle inequality for  $j=1,m$. This and the analogous arguments to that in \eqref{theta-def}  reveal
\begin{align*}
    \norm{\widehat{\theta}_h^{1/2}}_{a,h}^2+\norm{\widehat{\theta}_h^{1,{1}/{4}}}_{a,h}^2+\norm{\widehat{\theta}_h^{m,{1}/{4}}}_{a,h}^2 \lesssim\begin{cases}
 h^{2}\norm{u}^2_{L^\infty(H^{3}(\mathcal{T}_h))}+h^{2(1+\beta)}\norm{u}^2_{L^\infty(H^{3+\beta}(\mathcal{T}_h))}\;&\text{ for }\; k = 0,\\[6pt]
h^{2(k+1)}\norm{u}^2_{L^\infty(H^{k+3}(\mathcal{T}_h))}\;&\text{ for }\; k \ge 1.
 \end{cases}
\end{align*}
An application of Theorem~\ref{lm1_in_dis_tr_11}(c) and the bound
$\sum_{j=1}^{m-1}\int_{t_{j}}^{t_{j+2}}\|u_t\|^2_{X}\,dt \lesssim \|u_t\|^2_{L^2(X)}$ 
for $X =H^{3}(\mathcal{T}_h),$ $ H^{3+\beta}(\mathcal{T}_h), $ $ H^{k+3}(\mathcal{T}_h)$ 
and $j=n-1,n$, yield
 \begin{align*}
   \sum_{n=1}^{m-1}\Big[\int_{t_{n}}^{t_{n+2}}\norm{\widehat{\theta}_{ht}(t)}_{a,h}^2\dt+\int_{t_{n-1}}^{t_{n+1}}\norm{\widehat{\theta}_{ht}(t)}_{a,h}^2\dt \Big]\lesssim \begin{cases}
 h^{2}\norm{u_t}^2_{L^2(H^{3}(\mathcal{T}_h))}+h^{2(1+\beta)}\norm{u_t}^2_{L^2(H^{3+\beta}(\mathcal{T}_h))}\;&\text{ for }\; k = 0,\\[6pt]
h^{2(k+1)}\norm{u_t}^2_{L^2(H^{k+3}(\mathcal{T}_h))}\;&\text{ for }\; k \ge 1.
 \end{cases}
\end{align*}
Analogous arguments to those in~\eqref{new_ll2} and~\eqref{new_1last}, with $1/2$ replaced by $m+1/2$, reveal
\begin{align*}
    &\norm{\bar{\partial}_tu^{m+1/2}-\Pi_h^{k+2}(\bar{\partial}_tu^{m+1/2})}^{2} \lesssim h^{2(k+1)} \norm{u_t}^2_{L^\infty({t_m},t_{m+1};H^{k+1}(\mathcal{T}_h))},\\[6pt]
  &  \sum_{K\in {\cal T}_{h}}\norm{\nabla^{2}(u^{m+1/2}-{\cal E}_{K}({u} ^{m+1/2}))}^{2}_{K}\lesssim h^{2(k+1)} \norm{u}^2_{L^\infty({t_m},t_{m+1};H^{k+3}(\mathcal{T}_h))}.
\end{align*}
%\footnote{is it 2(k+3) or 2(k+1) in the first inequality, can you write the conclusion statement without order notation please}
%Apply \eqref{ref_1_eh_bd_1} to bound first two terms on right-hand side of \eqref{last}, Theorem \ref{lm1_in_dis_tr_11}(c)  to fourth and fifth term, and Lemma~\ref{bd_elli_proj} for last term. This and Theorem \ref{lm1_in_dis_tr_11}$(d)$
The last five displayed inequalities in \eqref{last} lead to
 \begin{align*} \norm{\bar{\partial}_t(u^{m+1/2}-u_h^{m+1/2})}^2&+\sum_{K\in {\cal T}_{h}}\norm{\nabla^{2}(u^{m+1/2}-{\cal R}_{K}(\widehat{u}_h ^{m+1/2}))}^2_{K}
  \lesssim  \begin{cases}
 h^2+h^{2(1+\beta)}+(\Delta t)^{4}\;&\text{ for }\; k = 0,\\
h^{2(k+1)}+(\Delta t)^{4}\;&\text{ for }\; k \ge 1.
 \end{cases}  
\end{align*}
This concludes the proof.\qed

\medskip
\noindent
{The next theorem establishes optimal order $L^2$-error estimates for the Newmark scheme with HHO-A scheme. Though the proof is motivated from $L^2$-error approximation of Leapfrog scheme for wave equation in \cite{grote2009optimal} using discontinuous Galerkin scheme, there are significant differences due to different space and time approximations, with {\it no requirement of the CFL condition in this proof}. The main idea is to modify \eqref{erroreqn} by using \eqref{split_err_eq1} to simplify the error equation in terms of  ${\rho}_{h}^{n}$ rather than ${e}_{h}^{n}$ in energy error estimate in Theorem~\ref{P1 implicit_Th2}. Unlike Theorem~\ref{P1 implicit_Th2}, this modification leaves the residual terms on the right-hand side of this new error equation in $L^2$ inner product. The choice of a suitable test function in this error equation leads to a telescoping energy identity which involves initial error and  the  residual truncation and projection terms.
The assumption $\widehat{E}_{h}(u^{0})=\widehat{I}_{h}(u^{0})$  makes ${\rho}_{h}^{0}=0$ which becomes crucial to obtain optimal $L^2$ estimates for initial error term. The truncation and projection terms
are then bounded using Lemma~\ref{Ch2-d-gronwall} and Theorem~\ref{m1_H^2_bound}, respectively. The discrete Gronwall lemma then yields a uniform-in-time estimate.}

\medskip
\noindent
\uline{\textbf{Proof of Theorem \ref{err_L2_new}.}}
The proof is divided into six steps. Step~1  reformulates  the error equation and then a 
suitable choice of the test function leads to a key energy identity. Step~2 estimates the residual and projection error terms 
on the right-hand side. 
Step~3 combines the results of Steps~1--2 and yields a uniform in time bound in terms of time truncation and spatial error. Step~4 provides the 
initial error bounds. Step~5 assembles the full error  
through a  triangle inequality 
and Poincar\'{e} estimate. Step~6 
bounds each remaining term using the truncation error lemmas and 
$L^2$ projection estimates.

\medskip
\noindent
{{\text{{\it Step 1. (key identity)}}}}
Utilize $\bar{\partial}_{t}^{2}{e}_{h}^{n}=\bar{\partial}_{t}^{2}{\theta}_{h}^{n}+\bar{\partial}_{t}^{2}{\rho}^{n}_{h}$ and $\widehat{e}^{n,{1}/{4}}=\widehat{\theta}_{h}^{n,1/4}+\widehat{\rho}^{n,/14}_{h}$ from \eqref{split_err_eq1} in \eqref{erroreqn} to obtain the error equation
\begin{equation}
    (\bar{\partial}_{t}^{2}{\rho}_{h}^{n},{v}_{h})+a_{h}(\widehat{\rho}^{n,{1}/{4}},\widehat{v}_{h})=-(\bar{\partial}_{t}^{2}{\theta}_{h}^{n},{v}_{h})+(R^n,v_{h})\text{ with }{R}^n=\bar{\partial}_{t}^{2}{u}^{n}-u_{tt}^{n,{1}/{4}}.\label{erroreqn2} 
\end{equation}
Choose $\widehat{v}_h =\widehat{\rho}^{n+1}-\widehat{\rho}^{n-1}$ in \eqref{erroreqn2} and proceed similar as in Theorem~\ref{P2 full_stabilty} from \eqref{stab}-\eqref{stab4} to show
\begin{align}
    \|\bar{\partial}_{t}{\rho}_{h}^{m+1/2}\|^2+\|\widehat{\rho}_{h}^{m+1/2}\|^2_{a,h}=  \|\bar{\partial}_{t}{\rho}_{h}^{1/2}\|^2+\|\widehat{\rho}_{h}^{1/2}\|^2_{a,h}+\Delta t \sum_{n=1}^m(R^n-\bar{\partial}_{t}^{2}{\theta}_{h}^{n},\bar{\partial}_{t}{\rho}_{h}^{n+1/2}+\bar{\partial}_{t}{\rho}_{h}^{n-1/2}).\label{one-new}
\end{align}
{\text{\it Step 2. (bound on residuals)}}
Arguments analogous  to \eqref{f1} with $f^{n,1/4}$ replaced by $R^n-\bar{\partial}_{t}^{2}{\theta}_{h}^{n}$  and $u_h^n$ by ${\rho}_{h}^n$ lead to
\begin{align} 
&\Delta t\sum_{n=1}^m(R^n-\bar{\partial}_{t}^{2}{\theta}_{h}^{n},\bar{\partial}_{t}{\rho}_{h}^{n+1/2}+\bar{\partial}_{t}{\rho}_{h}^{n-1/2})
\nonumber\\
&\le \Delta t T \sum_{n=1}^m\norm{R^n-\bar{\partial}_{t}^{2}{\theta}_{h}^{n}}^2+\frac{1}{2}\norm{\bar{\partial}_{t}{\rho}_{h}^{{1}/{2}}}^{2}+\frac{1}{2}\norm{\bar{\partial}_{t}{\rho}_{h}^{m+{1}/{2}}}^{2}+\frac{\Delta t}{T} \sum_{n=1}^{m-1}\norm{\bar{\partial}_{t}{\rho}_{h}^{n+{1}/{2}}}^{2}\nonumber\\
&  \le2\Delta t T \sum_{n=1}^m\big(\norm{R^n}^2+\norm{\bar{\partial}_{t}^{2}{\theta}_{h}^{n}}^2\big)+\frac{1}{2}\norm{\bar{\partial}_{t}{\rho}_{h}^{{1}/{2}}}^{2}+\frac{1}{2}\norm{\bar{\partial}_{t}{\rho}_{h}^{m+{1}/{2}}}^{2}+\frac{\Delta t}{T} \sum_{n=1}^{m-1}\norm{\bar{\partial}_{t}{\rho}_{h}^{n+{1}/{2}}}^{2}.\label{two-new}
\end{align}
{{\text{{\it Step 3. (intermediate energy estimate)}}}}
A combination of 
\eqref{one-new}-\eqref{two-new} yields
\begin{align*}
    \frac{1}{2}\|\bar{\partial}_{t}{\rho}_{h}^{m+1/2}\|^2+\|\widehat{\rho}_{h}^{m+1/2}\|^2_{a,h}\le  \frac{3}{2}\|\bar{\partial}_{t}{\rho}_{h}^{1/2}\|^2+\|\widehat{\rho}_{h}^{1/2}\|^2_{a,h}+2\Delta t T \sum_{n=1}^m\big(\norm{R^n}^2+\norm{\bar{\partial}_{t}^{2}{\theta}_{h}^{n}}^2\big)+\frac{\Delta t}{T} \sum_{n=1}^{m-1}\norm{\bar{\partial}_{t}{\rho}_{h}^{n+{1}/{2}}}^{2}.
\end{align*}
An application of  Lemma~\ref{Ch2-d-gronwall} with Remark~\ref{rem2.6} shows
    \begin{align}
    \|\bar{\partial}_{t}{\rho}_{h}^{m+1/2}\|^2+\|\widehat{\rho}_{h}^{m+1/2}\|^2_{a,h}\lesssim  \|\bar{\partial}_{t}{\rho}_{h}^{1/2}\|^2+\|\widehat{\rho}_{h}^{1/2}\|^2_{a,h}+\Delta t  \sum_{n=1}^m\big(\norm{R^n}^2+\norm{\bar{\partial}_{t}^{2}{\theta}_{h}^{n}}^2\big).\label{three-new1}
\end{align}
{{\text{{\it Step 4. (initial error bounds)}}}} The identities $ \bar{\partial}_t e_h^{1/2}=\bar{\partial}_t{\theta}_{h}^{1/2}+\bar{\partial}_t{\rho}^{1/2}_{h}$ and $\widehat{e}^{1/2}=\widehat{\theta}_{h}^{1/2}+\widehat{\rho}^{1/2}_{h}$ from  \eqref{split_err_eq1} applied to  \eqref{initial_eqn} yield
\begin{align}
   {2}{(\Delta t)^{-1}}(\bar{\partial}_t \rho_h^{1/2},v_h )+ a_h( \widehat{\rho}_h ^{1/2},\widehat{v}_h)&=
(\widetilde{R}^0,v_h)-  {2}{(\Delta t)^{-1}}(\bar{\partial}_t \theta_h^{1/2},v_h )\text{ with }\widetilde{R}^0={2}{(\Delta t)^{-1}}(\bar{\partial}_t u^{1/2}-{u_1}) -u_{tt}^{1/2}.\nonumber
\end{align}
The assumption $\widehat{E}_{h}(u^{0})=\widehat{I}_{h}(u^{0})$ and the initial discretiztion $\widehat{u}_{h}^0=\widehat{I}_{h}(u^{0})$ lead to $\widehat{\rho}_h^{0}= \widehat{E}_{h}(u^{0})- \widehat{u}^{0}_{h}=0$. This, the choice of  test function $\widehat{v}_h = \widehat{\rho}_h^{1/2} $, and the identity $\widehat{\rho}_h^{1/2}= \tfrac{1}{2}\Delta t\,\bar{\partial}_t \rho_h^{1/2}$ (which follows from $\rho_h^{0}=0$) yield
\begin{align}
   \|\bar{\partial}_t \rho_h^{1/2}\|^2+\|\widehat{\rho}_h ^{1/2}\|_{a,h}^2&=
\frac{1}{2}\Delta t(\widetilde{R}^0,\bar{\partial}_t \rho_h^{1/2}) - (\bar{\partial}_t \theta_h^{1/2}, \bar{\partial}_t \rho_h^{1/2})\nonumber\\
&\le \frac{1}{2}\Delta t\|\widetilde{R}^0\|\|\bar{\partial}_t \rho_h^{1/2}\|+\|\bar{\partial}_t \theta_h^{1/2}\|\|\bar{\partial}_t \rho_h^{1/2}\|\label{new}
\end{align}
with a Cauchy--Schwarz inequality in the last step. An application of Young's inequality with $a=\frac{1}{2}\Delta t\|\widetilde{R}^0\|$, $b=\|\bar{\partial}_t \rho_h^{1/2}\|$, and $\epsilon=1/2$  (resp. $a=\|\bar{\partial}_t \theta_h^{1/2}\|$, $b=\|\bar{\partial}_t \rho_h^{1/2}\|$, and $\epsilon=1/2$) show 
$$\frac{1}{2}\Delta t\|\widetilde{R}^0\|\|\bar{\partial}_t \rho_h^{1/2}\| \le \frac{1}{4} (\Delta t)^2\|\widetilde{R}^0\|^2+\frac{1}{4}\|\bar{\partial}_t \rho_h^{1/2}\|^2\;\Big(\text{resp. }\|\bar{\partial}_t \theta_h^{1/2}\|\|\bar{\partial}_t \rho_h^{1/2}\| \le \|\bar{\partial}_t \theta_h^{1/2}\|^2+\frac{1}{4}\|\bar{\partial}_t \rho_h^{1/2}\|^2\Big).$$
This and \eqref{new} reveal
\begin{align}
  \frac{1}{2} \|\bar{\partial}_t \rho_h^{1/2}\|^2+\|\widehat{\rho}_h ^{1/2}\|_{a,h}^2 \le \frac{1}{4}(\Delta t)^2\|\widetilde{R}^0\|^2+\|\bar{\partial}_t \theta_h^{1/2}\|^2.\label{four-new}
   \end{align}
\text{\it Step 5. (assembling the errors)}
A combination of \eqref{three-new1} and \eqref{four-new} with a rearrangement of terms  reveals
 \begin{align}
    \|\bar{\partial}_{t}{\rho}_{h}^{m+1/2}\|^2+\|\widehat{\rho}_{h}^{m+1/2}\|^2_{a,h}\lesssim  (\Delta t)^2\|\widetilde{R}^0\|^2+\Delta t  \sum_{n=1}^m\norm{R^n}^2+\|\bar{\partial}_t \theta_h^{1/2}\|^2+\Delta t  \sum_{n=1}^m\norm{\bar{\partial}_{t}^{2}{\theta}_{h}^{n}}^2.\label{three-new}
\end{align}
From \eqref{temp_dis}, we have  $u^{m+1}-u_h^{m+1}=\big(u^{m+1/2}-u^{m+1/2}_h\big)+\frac{1}{2}\Delta t \big(\bar{\partial}_{t}u^{m+1/2}-\bar{\partial}_{t}u_h^{m+1/2}\big)$. This and  \eqref{split_err_eq1}, and some elementary algebra yield
\begin{align*}
u^{m+1}-u_h^{m+1}&=\big(u^{m+1/2}-\Pi^{k+2}_{h}(u^{m+1/2})+\theta_{h}^{m+1/2}+{\rho}_{h}^{m+1/2}\big)\\
&\quad+\frac{1}{2}\Delta t\big(\bar{\partial}_{t}u^{m+1/2}-\Pi^{k+2}_{h}(\bar{\partial}_{t}u^{m+1/2})+\bar{\partial}_{t}\theta_h^{m+1/2}+\bar{\partial}_{t}{\rho}_{h}^{m+1/2}\big).
\end{align*}
The triangle inequality  (applied thrice), $\Delta t \le T$, and $\|{\rho}_h^{m+1/2}\|\le C_P \|\widehat{\rho}_h^{m+1/2}\|_{a,h}$ from \eqref{poinacre} show
\begin{align*}
 \|u^{m+1}-u_h^{m+1}\|^2 
 &\lesssim  \|u^{m+1/2}-\Pi^{k+2}_{h}(u^{m+1/2})\|^2+\norm{\bar{\partial}_tu^{m+1/2}-\Pi_h^{k+2}(\bar{\partial}_tu^{m+1/2})}^{2}\\
 &\quad+\|\theta_h^{m+1/2}\|^2+\|\bar{\partial}_{t}\theta_h^{m+1/2}\|^2+\|\widehat{\rho}_h^{m+1/2}\|_{a,h}^2+\|\bar{\partial}_{t}\rho_h^{m+1/2}\|^2.
\end{align*}
This and estimates from \eqref{three-new} with rearrangement of terms leads to
\begin{align}
     \|u^{m+1}-u_h^{m+1}\|^2 &\lesssim (\Delta t)^2\|\widetilde{R}^0\|^2+\Delta t  \sum_{n=1}^m\norm{R^n}^2+\norm{\bar{\partial}_tu^{m+1/2}-\Pi_h^{k+2}(\bar{\partial}_tu^{m+1/2})}^{2}+ \|u^{m+1/2}-\Pi^{k+2}_{h}(u^{m+1/2})\|^2\nonumber\\
     &\qquad +\|\theta_h^{m+1/2}\|^2+\|\bar{\partial}_t \theta_h^{1/2}\|^2+\|\bar{\partial}_{t}\theta_h^{m+1/2}\|^2+\Delta t  \sum_{n=1}^m\norm{\bar{\partial}_{t}^{2}{\theta}_{h}^{n}}^2.\label{new_all}
\end{align}
{{\text{{\it Step 6. (consolidation)}}}}
Recall the  following bounds for first two terms on the right-hand side of last displayed inequality from Step 4 of Theorem~\ref{P1 implicit_Th2}:
\begin{align}
   &(\Delta t)^2\|\widetilde{R}^0\|^2+\Delta t  \sum_{n=1}^m\norm{R^n}^2 \lesssim (\Delta t)^4 \big(\|u_{ttt}\|^2_{L^\infty(L^2(\Omega))}+\|u_{tttt}\|^2_{L^2(L^2(\Omega))}\big) .\label{new_trunc}
\end{align}
  An application of   Theorem \ref{lm1_in_dis_tr_11}(d) bounds the third term on the right-hand side of \eqref{new_all} viz.
  \begin{align}
  & \norm{\bar{\partial}_tu^{m+1/2}-\Pi_h^{k+2}(\bar{\partial}_tu^{m+1/2})}^{2}\lesssim  h^{2(k+3)}\norm{u_t}^2_{L^\infty(H^{k+3}(\mathcal{T}_h))}. 
\end{align}
For the fourth term, 
we utilize the estimates from Lemma~\ref{lm1_in_dis_tr_11}(d) to obtain
\begin{align}
    &\|u^{m+1/2}-\Pi^{k+2}_{h}(u^{m+1/2})\|^2\lesssim  h^{2(k+3)}\norm{u}^2_{L^\infty(0,t_1;H^{k+3}(\mathcal{T}_h))}.\label{2ju}
\end{align}
 An application of  triangle inequality with~\eqref{P3 phi_c1 } shows $\|{\theta}_h^{m+1/2}\|^2 \le\frac{1}{2}\big(\|{\theta}_h^{m}\|^2+\|{\theta}_h^{m+1}\|^2\big)$. This in combination with \eqref{split_err_eq1} 
 and bound from  Theorem~\ref{m1_H^2_bound} leads to
\begin{align}
 \|\theta_h^{m+1/2}\|^2   \lesssim \begin{cases}
 h^{4}\big(\norm{u}^2_{L^\infty(H^{3}(\mathcal{T}_h))}+h^{2\beta}\norm{u}^2_{L^\infty(H^{3+\beta}(\mathcal{T}_h))}+\|\Delta^2u\|^2_{L^\infty(L^2(\Omega))}\big)\;&\text{ for }\; k = 0,\\[6pt]
h^{2(k+3)}\big(\norm{u}^2_{L^\infty(H^{k+3}(\mathcal{T}_h))}+\|\Delta^2u\|^2_{L^\infty(H^{k-1}(\mathcal{T}_h))}\big)\;&\text{ for }\; k \ge 1.
 \end{cases}
\end{align}
Note that  $\|\bar{\partial}_{t}\theta_{h}^{n+1/2} \|={(\Delta t)^{-1}}\|\int_{t_n}^{t_{n+1}} \theta_{ht}\dt\|\le {(\Delta t)^{-1/2}}\big(\int_{t_n}^{t_{n+1}} \|\theta_{ht}\|^2\dt\big)$ by Cauchy--Schwarz inequality.  This and  bounds from  Theorem~\ref{m1_H^2_bound}  leads to
\begin{align}
\|\bar{\partial}_t \theta_h^{1/2}\|^2+\|\bar{\partial}_{t}\theta_h^{m+1/2}\|^2  \lesssim \begin{cases}
h^{4}\big(\norm{u_t}^2_{L^\infty(H^{3}(\mathcal{T}_h))}+h^{2\beta}\norm{u_t}^2_{L^\infty(H^{3+\beta}(\mathcal{T}_h))}+\|\Delta^2u_t\|^2_{L^\infty(L^2(\Omega))}\big)&\text{ for } k = 0,\\[6pt]
h^{2(k+3)}\big(\norm{u_t}^2_{L^\infty(H^{k+3}(\mathcal{T}_h))}+\|\Delta^2u_t\|^2_{L^\infty(H^{k-1}(\mathcal{T}_h))}\big)&\text{ for } k \ge 1.
 \end{cases}
\end{align}
Apply  Taylor's series and a Cauchy--Schwarz inequality to show $\norm{\bar{\partial}_t^2 \theta^n_h}^2 \le \frac{2}{3}(\Delta t)^{-1}\int_{t_{n-1}}^{t_{n+1}} \norm{\theta_{htt}(t)}^2 \dt.$ This and  Theorem~\ref{m1_H^2_bound}  %(resp. $\norm{ \bar{\partial}_t^2 \chi^n}^2  \le \frac{2}{3}(\Delta t)^{-1}\int_{t_{n-1}}^{t_{n+1}} \norm{\chi_{tt}(t)}^2 \dt$)
reveal that 
\begin{align}\label{new_last}
    \Delta t  \sum_{n=1}^m\norm{\bar{\partial}_{t}^{2}{\theta}_{h}^{n}}^2 \lesssim \begin{cases}
 h^{4}\big(\norm{u_{tt}}^2_{L^2(H^{3}(\mathcal{T}_h))}+h^{2\beta}\norm{u_{tt}}^2_{L^2(H^{3+\beta}(\mathcal{T}_h))}+\|\Delta^2u_{tt}\|^2_{L^2(L^2(\Omega))}\big)\;&\text{ for }\; k = 0,\\[6pt]
h^{2(k+3)}\big(\norm{u_{tt}}^2_{L^2(H^{k+3}(\mathcal{T}_h))}+\|\Delta^2u_{tt}\|^2_{L^2(H^{k-1}(\mathcal{T}_h))}\big)\;&\text{ for }\; k \ge 1.
 \end{cases}
\end{align}
A combination of \eqref{new_all}-\eqref{new_last} leads to
$ \norm{u^{m+1}- u_h^{m+1} }^2 \lesssim 
\begin{cases}
h^ 4+ h^{2(2+\beta)}+(\Delta t)^4\quad\;\;\; &\text{ for } k=0,\\
h^{2(k+3)} +(\Delta t)^4  \quad &\text{ for }  k\ge1.
\end{cases}
$\\
This concludes the proof.\qed
%Application of discrete Gronwall's inequality, estimate for \textcolor{blue}{$\theta_h$ from approximation property lemma and ritz projection estimates} completes the proof.\qed
%\end{proof}
\section{Crank-Nicolson scheme}\label{Sec_cr_nic}
We first establish the stability of the Crank--Nicolson scheme in Subsection~\ref{sub_l2_cr}. Subsection~\ref{Subs_err_cr} 
derives optimal energy error estimates for the variable $u$ in  HHO-A and HHO-B schemes, and optimal $L^2$ error estimates 
for both $u$ and $p$ under the HHO-A scheme.
\subsection{Stability}\label{sub_l2_cr}

{The first theorem proves the stability of the  scheme ~\eqref{HHO_Crank_Nico_sc_1},  
which guarantees the existence, uniqueness, and dependency of the discrete solution on the given data to Crank-Nicolson scheme. The proof relies on the energy method combined with the discrete Gronwall lemma. The key idea is to choose a suitable  test functions in both equations of the coupled system so that, after subtraction, the resulting expression produces an energy balance relation. A summation over all time levels yields a global energy identity containing forcing contributions and initial energies. The forcing terms are then rewritten using summation-by-parts and estimated, leading to an inequality involving accumulated energies at previous time levels. Finally, an application of the discrete Gronwall lemma provides a uniform-in-time stability bound.}

\medskip \noindent
\uline{\textbf{Proof of Theorem \ref{Stab-1}.}}
Choose $q_h=\bar{\partial}_{t}p^{n+1/2}_{h}$ in \eqref{HHO_Crank_Nico_sc_1-a} and 
$\widehat{v}_h=\bar{\partial}_{t}\widehat{u}_{h}^{n+1/2}$ in \eqref{HHO_Crank_Nico_sc_1-b}, and subtract the resulting equations to obtain
\begin{align}
(p_h^{n+1/2},\bar{\partial}_{t}p^{n+1/2}_{h})+a_{h}(\widehat{u}_{h}^{n+1/2},\bar{\partial}_{t} \widehat {u}_{h}^{n+1/2})=(f^{n+1/2},\bar{\partial}_{t}u_{h}^{n+1/2}).\label{full}
\end{align}  
The definitions in \eqref{temp_dis}, linearity of $a_h(\bullet,\bullet)$, \eqref{norm}, and some elementary manipulations yield
$$
2 \Delta t(p_h^{n+1/2},\bar{\partial}_{t}p^{n+1/2}_{h})
=\|p_h^{n+1}\|^2-\|p_h^{n}\|^2
\quad \text{and} \quad
2 \Delta t\, a_{h}(\widehat{u}_{h}^{n+1/2},\bar{\partial}_{t} \widehat{u}_{h}^{n+1/2})
=\|\widehat{u}_{h}^{n+1}\|^2_{a,h}-\|\widehat{u}_{h}^{n}\|^2_{a,h}.
$$
This and a summation for $n=0,1,\cdots, m$ in \eqref{full} with $\Delta t\bar{\partial}_{t}u_{h}^{n+1/2}=u_{h}^{n+1}-u_{h}^{n}$ lead to the  energy balance equation as
\vspace{-0.3cm}
\begin{align}
\norm{p_h^{m+1}}^2+\norm{\widehat{u}_{h}^{m+1}}^2_{a,h}=\norm{p_h^{0}}^2+\norm{\widehat{u}_{h}^{0}}^2_{a,h}+2 \sum_{n=0}^{m}(f^{n+1/2},u_{h}^{n+1}-u_{h}^{n}) \text{ for any } 0 \le m \le N-1.\label{f-all}
\vspace{-0.3cm}
\end{align}
An application of \eqref{sum_by_parts2} and the identity  $2(f^{n+1/2}-f^{n-1/2})=f^{n+1}-f^{n-1}$ from \eqref{temp_dis} show
\begin{align}
   2 \sum_{n=0}^{m}(f^{n+1/2},u_{h}^{n+1}-u_{h}^{n}) %&=(f^{1/2},u_{h}^{1}-u_{h}^{0})+\sum_{n=1}^{m}(f^{n+1/2},u_{h}^{n+1}-u_{h}^{n}) \nonumber\\
   % & =(f^{m+1/2},u_h^{m+1})-(f^{1/2},u_{h}^{0})-\sum_{n=1}^{m}(f^{n+1/2}-f^{n-1/2},u_{h}^{n}) \nonumber\\
    &=2(f^{m+1/2},u_h^{m+1})-2(f^{1/2},u_{h}^{0})-\sum_{n=1}^{m}(f^{n+1}-f^{n-1},u_{h}^{n}) .\label{f-series}
\end{align}
Utilize a Cauchy--Schwarz inequality, \eqref{poinacre}, and Young's inequality  to obtain
\begin{align}
    (f^{m+1/2},u_h^{m+1}) &\le \norm{f^{m+1/2}}\norm{u_h^{m+1}} \le C_P\norm{f^{m+1/2}}\norm{\widehat{u}_h^{m+1}}_{a,h} \le {C_P^2}\norm{f^{m+1/2}}^2+\frac{1}{4}\norm{\widehat{u}_h^{m+1}}_{a,h}^2 ,\\
     -(f^{1/2},u_h^{0}) &\le \norm{f^{1/2}}\norm{u_h^{0}} \le C_P\norm{f^{1/2}}\norm{\widehat{u}_h^{0}}_{a,h} \le {C_P^2}\norm{f^{1/2}}^2+\frac{1}{4}\norm{\widehat{u}_h^{0}}_{a,h}^2, \\
    - (f^{n+1}-f^{n-1},u_{h}^{n})& \le C_P\norm{f^{n+1}-f^{n-1}}\norm{\widehat{u}_h^{n}}_{a,h} \le \frac{TC_P^2}{2\Delta t}\norm{f^{n+1}-f^{n-1}}^2+\frac{\Delta t}{2T}\norm{\widehat{u}_h^{n}}_{a,h}^2.\label{f-l2}
\end{align}
A combination of \eqref{f-series}-\eqref{f-l2} in \eqref{f-all} leads to
\begin{align}
\norm{p_h^{m+1}}^2+\frac{1}{2}\norm{\widehat{u}_{h}^{m+1}}^2_{a,h}&\le \norm{p_h^{0}}^2+\frac{3}{2}\norm{\widehat{u}_{h}^{0}}^2_{a,h}
+\frac{\Delta t }{2T}\sum_{n=1}^{m} \norm{\widehat{u}_h^{n}}_{a,h}^2\nonumber\\
&\quad+2{C_P^2}\Big(\norm{f^{1/2}}^2+ \norm{f^{m+1/2}}^2+\frac{T}{4\Delta t} \sum_{n=1}^{m}\norm{f^{n+1}-f^{n-1}}^2 \Big).\label{two-two}
\end{align}
Utilize  $\|f^{n+1}-f^{n-1}\|=\|\int_{t_{n-1}}^{t_{n+1}}{f_t}\dt\|\le \int_{t_{n-1}}^{t_{n+1}}\|{f_t}\|\dt$ and a Cauchy--Schwarz inequality $\int_{t_{n-1}}^{t_{n+1}}\|{f_t}\|\dt\le\sqrt{\Delta t}\big(\int_{t_{n-1}}^{t_{n+1}}\|{f_t}\|^2\dt\big)^{1/2}$ to show  
$
\sum_{n=1}^{m} \|f^{n+1}-f^{n-1}\|^2 
\le 2\Delta t\norm{f_t}^2_{L^2(L^2(\Omega))}
$. This and the bounds $\|f^{k+1/2}\|^2 \leq \|f\|^2_{L^\infty(L^2(\Omega))}$ for any $k =0,1,\ldots,m$ applied in \eqref{two-two} reveal
\begin{align*}
\norm{p_h^{m+1}}^2+\frac{1}{2}\norm{\widehat{u}_{h}^{m+1}}^2_{a,h}
&\le \norm{p_h^{0}}^2+\frac{3}{2}\norm{\widehat{u}_{h}^{0}}^2_{a,h}
+{C_P^2}\Big(4\norm{f}^2_{L^\infty(L^2(\Omega))}+T\norm{f_t}^2_{L^2(L^2(\Omega))}\Big)+\frac{\Delta t }{2T}\sum_{n=1}^{m} \norm{\widehat{u}_h^{n}}_{a,h}^2.
\end{align*}
An application of  Lemma~\ref{Ch2-d-gronwall} with Remark~\ref{rem2.6}  leads to
\begin{align*}
\norm{p_h^{m+1}}^2+\norm{\widehat{u}_{h}^{m+1}}^2_{a,h}
&\lesssim \norm{p_h^{0}}^2+\norm{\widehat{u}_{h}^{0}}^2_{a,h}
+\norm{f}^2_{L^\infty(L^2(\Omega))}+\norm{f_t}^2_{L^2(L^2(\Omega))}.
\end{align*}
This concludes the proof.
%\begin{align*}
%\norm{p_h^{m+1}}^2+\norm{\widehat{u}_{h}^{m+1}}^2_{a,h}&\lesssim  \norm{p_h^{0}}^2+\norm{\widehat{u}_{h}^{0}}^2_{a,h}
 %+\norm{f}^2_{L^\infty(L^2(\Omega))}+T \Delta t \sum_{n=1}^{m}\norm{\delta_t f^n}^2 .
%\end{align*}
%The proof is complete by Taylor's series and a Cauchy--Schwarz inequality applied to the last term in the above inequality (the proof is similar to that of Lemma~\ref{app-lem}) to show.
%$$\Delta t \sum_{n=1}^{m}\norm{\delta_t f^n}^2 \lesssim \norm{f_t}^2_{L^2(L^2(\Omega))}.$$
%Taylor's series expansion $\Delta t \sum_{n=1}^{m}\norm{\delta_t f^n}^2 \lesssim \norm{f_t}^2_{L^2(L^2(\Omega))}$  completes the proof. \qed
\subsection{Error estimates}\label{Subs_err_cr}
First, let us introduce the split  
\begin{equation}
p^{n}-{p}_h^n=\big(p^n-\Pi_{h}^{k+2}(p^{n})\big)+\big(\Pi_{h}^{k+2}(p^{n})-{p}_{h}^{n}\big):={\chi}_{h}^{n}+{\gamma}_{h}^{n} \label{split-p}.
\end{equation}
Also recall from~\eqref{split_err_eq1} that
$$\widehat{e}^{n}_{h}:=\widehat{I}^{k}_{h}(u^{n})-\widehat{u}^{n}_{h}=\widehat{\theta}^{n}_{h}+\widehat{\rho}^{n}_{h},\;\text{where}\; \widehat{\theta}^{n}_{h}=\widehat{I}^{k}_{h}(u^{n})-\widehat{E}_{h}(u^{n}) \: \text{and }
\widehat{\rho}^{n}_{h}=\widehat{E}_{h}(u^{n})-\widehat{u}_{h}^{n}.
$$
%%$$\widehat{\theta}_h^n=\widehat{I}^{k}_{h}(u^{n})-\widehat{E}_{h}(u^{n})\quad\text{and}\quad\widehat{\rho}_h^n=\widehat{E}_{h}(u^{n})-\widehat{u}_{h}^{n}.$$ 
\uline{\textbf{Error equations.}}
Utilize $\widehat{\rho}_h^{n+1/2}=\widehat{E}_{h}(u^{n+1/2})-\widehat{u}_{h}^{n+1/2}$ and \eqref{HHO_Crank_Nico_sc_1-a} to obtain
\begin{align*}
(\bar{\partial}_{t}{\rho}_{h}^{n+1/2},q_h)&=  (E_{h}(\bar{\partial}_{t}u^{n+1/2})-\bar{\partial}_{t}u_h^{n+1/2},q_h)=(E_{h}(\bar{\partial}_{t}u^{n+1/2})-p_h^{n+1/2},q_h)\quad\text{ for all }{q}_h\in \mathcal{P}^{k+2}(\mathcal{T}_h).
\end{align*}
An application of \eqref{pi-pro}  shows $(\Pi_{h}^{k+2}(\bar{\partial}_{t}u^{n+1/2}),q_h)-(\bar{\partial}_{t}u^{n+1/2},q_h)=0\text{ for all }{q}_h\in \mathcal{P}^{k+2}(\mathcal{T}_h)$. Utilize this, recall $\widehat{\theta}_h^{n+1/2}=\widehat{I}^{k}_{h}(u^{n+1/2})-\widehat{E}_{h}(u^{n+1/2})$,  and introduce {${\xi^{j+1}}:=\bar{\partial}_{t}u^{j+1/2}-u_{t}^{j+1/2}$} for any $j=0,1,2,\cdots, N,$ to show
\begin{align*}
   (\bar{\partial}_{t}{\rho}_{h}^{n+1/2},q_h)&=(E_{h}(\bar{\partial}_{t}u^{n+1/2})-\Pi_{h}^{k+2}(\bar{\partial}_{t}u^{n+1/2}),q_h)+(\bar{\partial}_{t} u^{n+1/2}-p_h^{n+1/2},q_h)\\
&=-(\bar{\partial}_{t} \theta_h^{n+1/2},q_h)+({\xi^{n+1}},q_h)+(u_t^{n+1/2}-p_{h}^{n+1/2},q_h)\quad\text{ for all }  {q}_h\in \mathcal{P}^{k+2}(\mathcal{T}_h).
\end{align*}
%where ${\xi^{n+1}}:=\bar{\partial}_{t}u^{n+1/2}-u_{t}^{n+1/2}$ 
Note that $u_t(t)=p(t)$ for all $t \in [0,T]$, hence $(u_t^{n+1/2},q_h)=(p^{n+1/2},q_h).$ Apply \eqref{pi-pro} with the choice of  polynomial degree $s=k+2$ to show $ (p^{n+1/2},q_h)=(\Pi^{k+2}_h(p^{n+1/2}),q_h).$ Thus, the definition of  $\gamma_h^n$ from \eqref{split-p} leads to {\it first error equation} of the system \eqref{HHO_Crank_Nico_sc_1}, viz.
% $$(u_t^{n+1/2}-p_{h}^{n+1/2},q_h)=(p^{n+1/2}-p_{h}^{n+1/2},q_h)=(\gamma_h^{n+1/2},q_h).$$
 %This together with the definitions of $\theta_h^{n}$ and $\gamma_h^{n}$ in the above displayed equation yields
\begin{align}
(\bar{\partial}_{t}{\rho}_{h}^{n+1/2},q_h)
&=-(\bar{\partial}_{t} \theta_h^{n+1/2},q_h) +({\xi^{n+1}},q_h)
+(\gamma_h^{n+1/2},q_h)\quad\text{ for all }  {q}_h\in \mathcal{P}^{k+2}(\mathcal{T}_h).
\label{zeta}
\end{align}
Utilize definitions ${\gamma}_{h}^{n}=\Pi_{h}^{k+2}(p^{n})-{p}_{h}^{n}$ (resp. ${\rho}_{h}^{n}=\widehat{E}_{h}(u^{n})-\widehat{u}_{h}^{n}$) from \eqref{split-p} (resp. \eqref{split_err_eq1}),  the  projection $a_h(\widehat{E}_{h}({u}^{n+1/2}),\widehat{v}_{h})=(\Delta^2 u^{n+1/2},v_h)$ from \eqref{proj_op_1}, and that $(\widehat{u}^{n+1}_h,p^{n+1})$ satisfy \eqref{HHO_Crank_Nico_sc_1-b} to arrive at
\begin{align}
    (\bar{\partial}_{t}{\gamma}_{h}^{n+1/2},{v}_{h})+a_{h}(\widehat{\rho}_{h}^{n+1/2},\widehat{v}_{h})&=(\Pi_h^{k+1}(\bar{\partial}_{t}p^{n+1/2})-\bar{\partial}_{t}p_h^{n+1/2},v_{h})+a_h(\widehat{E}_{h}({u}^{n+1/2}),\widehat{v}_{h})-a_h(\widehat{u}_h^{n+1/2},\widehat{v}_{h})\nonumber\\
    &=(\Pi_h^{k+1}(\bar{\partial}_{t}p^{n+1/2}),v_{h})-(f^{n+1/2},v_{h})+(\Delta^2 u^{n+1/2},v_h)\nonumber \quad\text{for all}\quad \widehat {v}_h \in \widehat{V}_{h}^{0}.
\end{align}
Note that $(u,p)$ satisfies \eqref{P1 strong_form_1}, hence $(-f^{n+1/2}+\Delta^2 u^{n+1/2},v_h)=-(p_t^{n+1/2},v_h)$. Also, utilize $(\Pi_{h}^{k+2}(p^{n}),v_{h})=(p^{n},v_{h})$ from \eqref{pi-pro} (with the choice of  polynomial degree $s=k+2$) and introduce
{$\eta^{j+1}:=\bar{\partial}_{t}p^{j+1/2}-p_{t}^{j+1/2}$}  for any $j=0,1,2,\cdots, N,$ to obtain the {\it second error equation} of the system \eqref{HHO_Crank_Nico_sc_1} as
\begin{eqnarray}
   (\bar{\partial}_{t}{\gamma}_{h}^{n+1/2},{v}_{h})+a_{h}(\widehat{\rho}_{h}^{n+1/2},\widehat{v}_{h}) =(\bar{\partial}_{t}p^{n+1/2}-p_{t}^{n+1/2},v_{h})=(\eta^{n+1},v_{h}) \quad\text{for all}\quad \widehat {v}_h \in \widehat{V}_{h}^{0}.\label{sec_lm41_new}
\end{eqnarray}
{The next theorem aims to prove the energy estimates for \(u\) and the \(L^2\) estimates for the  variable \(p\)  for Crank-Nicolson scheme. The proof relies on derivation of a coupled system of error equations by using error decompositions for both variables, the discrete scheme, and the fact that the PDE system in \eqref{P1 strong_form_1} is satisfied at each time $t_n$. The key idea is to choose suitable test functions in these error equations so that the coupled terms cancel leading to a discrete energy identity involving only the approximation errors and truncation terms. A summation-by-parts is then utilized to rewrite the right-hand side in a form that separates temporal discretization errors from spatial approximation errors. These contributions are estimated independently using truncation error bounds and projection estimates. Finally, the discrete Gronwall lemma yields stability bounds for the auxiliary error components, from which the energy estimate for $u$ and the $L^2$-estimate for $p$ follow through the error decomposition and approximation properties.}

\medskip
\noindent
\uline{\textbf{Proof of Theorem \ref{err_thm}.}} The proof proceeds in seven steps. A key identity is derived in  \textit{Step~1} by choosing appropriate test functions in \eqref{zeta} (resp.  \eqref{sec_lm41_new}), and the right-hand side of the key identity is modified 
in \textit{Step~2} to decompose it into temporal and spatial discretization errors. Bounds for 
the truncation error terms and the spatial approximation terms are 
obtained in \textit{Step~3} and \textit{Step~4}, respectively. These 
are then combined in \textit{Step~5} to derive the energy estimates 
for $u$. Finally, the suboptimal $L^2$-norm estimate for $p$ is established in 
\textit{Step~6} for both HHO-A and HHO-B schemes. The $L^2$-norm estimate for $p$  in  HHO-A variant is derived in \textit{Step~7}.

\medskip
\noindent
{\it Step 1. (key identity)}
Choose $q_h=\bar{\partial}_{t}\gamma_h^{n+1/2}$ (resp. $\widehat{v}_h=\bar{\partial}_{t}\widehat{\rho}_{h}^{n+1/2}$) in \eqref{zeta} (resp.  \eqref{sec_lm41_new}) as test functions  to deduce
\begin{align*}
&(\bar{\partial}_{t}{\rho}_{h}^{n+1/2},\bar{\partial}_{t}\gamma_h^{n+1/2})
=-(\bar{\partial}_{t} \theta_h^{n+1/2},\bar{\partial}_{t}\gamma_h^{n+1/2}) +(\xi^{n+1},\bar{\partial}_{t}\gamma_h^{n+1/2})
+(\gamma_h^{n+1/2},\bar{\partial}_{t}\gamma_h^{n+1/2}),\\
&\big( \text{resp. }(\bar{\partial}_{t}{\gamma}_{h}^{n+1/2},\bar{\partial}_{t}{\rho}_{h}^{n+1/2})+a_{h}(\widehat{\rho}_{h}^{n+{1}/{2}},\bar{\partial}_{t}{\widehat{\rho}}_{h}^{n+1/2})=(\eta^{n+1},\bar{\partial}_{t}{\rho}_{h}^{n+1/2})\big).
%\label{4_1_eq_1}
\end{align*}
Subtract the last two displayed identities and rearrange the terms (with cancellation of term $(\bar{\partial}_{t}{\rho}_{h}^{n+1/2},\bar{\partial}_{t}\gamma_h^{n+1/2})$) to obtain
\begin{align*}
    (\gamma_h^{n+1/2},\bar{\partial}_{t}\gamma_h^{n+1/2})+a_{h}(\widehat{\rho}_{h}^{n+{1}/{2}},\bar{\partial}_{t}{\widehat{\rho}}_{h}^{n+1/2})&=(\bar{\partial}_{t} \theta_h^{n+1/2},\bar{\partial}_{t}\gamma_h^{n+1/2}) -(\xi^{n+1},\bar{\partial}_{t}\gamma_h^{n+1/2})+(\eta^{n+1},\bar{\partial}_{t}{\rho}_{h}^{n+1/2}) .
\end{align*}
Utilize $2\Delta t(\gamma_h^{n+1/2},\bar{\partial}_{t}\gamma_h^{n+1/2})= \norm{{\gamma}_{h}^{n+1}}^2- \norm{{\gamma}_{h}^{n}}^2$ and $2\Delta ta_{h}(\widehat{\rho}_{h}^{n+{1}/{2}},\bar{\partial}_{t}{\widehat{\rho}}_{h}^{n+1/2})=\norm{\widehat{\rho}_{h}^{n+1}}_{a,h}^2- \norm{\widehat{\rho}_{h}^{n}}_{a,h}^2$ from definitions \eqref{temp_dis} and \eqref{norm} to show
\begin{align*}
    \norm{{\gamma}_{h}^{n+1}}^2- \norm{{\gamma}_{h}^{n}}^2+  \norm{\widehat{\rho}_{h}^{n+1}}_{a,h}^2- \norm{\widehat{\rho}_{h}^{n}}_{a,h}^2&=2\Delta t \big((\bar{\partial}_{t} \theta_h^{n+1/2},\bar{\partial}_{t}\gamma_h^{n+1/2}) -(\xi^{n+1},\bar{\partial}_{t}\gamma_h^{n+1/2})+(\eta^{n+1},\bar{\partial}_{t}{\rho}_{h}^{n+1/2}) \big).
\end{align*}
A summation from $n=0$ to $n=m$ for any $0\le m \le N-1$, and $\gamma_h^{0}=0,\widehat{\rho}_{h}^{0}=0$ (by initialization  $\widehat{u}_{h}^{0}=\widehat{E}_{h}(u_{0})$ and  ${p}^{0}_{h}=\Pi_{h}^{k+2}(p^0)$) lead to
\begin{align}
    \norm{{\gamma}_{h}^{m+1}}^2  +\norm{\widehat{\rho}_{h}^{m+1}}_{a,h}^2&=  2\sum_{n=0}^m\big((\bar{\partial}_{t} \theta_h^{n+1/2},\gamma_h^{n+1}-\gamma_h^{n}) -(\xi^{n+1},\gamma_h^{n+1}-\gamma_h^{n}) +(\eta^{n+1},{\rho}_{h}^{n+1}-{\rho}_{h}^{n})\big). \label{cr}
\end{align}
{\it Step 2. (modification to RHS of \eqref{cr})}
Utilize  \eqref{sum_by_parts2},  the  observation that $\gamma_h^{0}=0$ and $\bar{\partial}_{t} \theta_h^{n+1/2}-\bar{\partial}_{t} \theta_h^{n-1/2}=\Delta t \bar{\partial}^2_{t} \theta_h^{n} $ from \eqref{P3 phi_c2} to infer
\begin{align}
   \sum_{n=0}^m(\bar{\partial}_{t} \theta_h^{n+1/2},\gamma_h^{n+1}-\gamma_h^{n}) &=(\bar{\partial}_{t} \theta_h^{m+1/2},\gamma_h^{m+1}) -(\bar{\partial}_{t} \theta_h^{1/2},\gamma_h^{0})  -\sum_{n=1}^m (\bar{\partial}_{t} \theta_h^{n+1/2}-\bar{\partial}_{t} \theta_h^{n-1/2},\gamma_h^{n})\nonumber\\
   &=  (\bar{\partial}_{t} \theta_h^{m+1/2},\gamma_h^{m+1})  -\Delta t \sum_{n=1}^m (\bar{\partial}^2_{t} \theta_h^{n},\gamma_h^{n}).\label{cr-p1}
\end{align}
Apply a Cauchy--Schwarz  and  Young's inequalities to obtain 
\begin{align}
     (&\bar{\partial}_{t} \theta_h^{m+1/2},\gamma_h^{m+1})\le 2\norm{\bar{\partial}_{t} \theta_h^{m+1/2}}^2+\frac{1}{8}\norm{\gamma_h^{m+1}}^2\; \text{and}\;-(\bar{\partial}^2_{t} 
      \theta_h^{n},\gamma_h^{n}) \le 2{T}\norm{\bar{\partial}^2_{t} \theta_h^{n}}^2+\frac{1}{8T}\norm{\gamma_h^{n}}^2.\label{cc3}
      \end{align}
      A combination of \eqref{cr-p1} and \eqref{cc3} leads to
      \begin{align}
          2 \sum_{n=0}^m(\bar{\partial}_{t} \theta_h^{n+1/2},\gamma_h^{n+1}-\gamma_h^{n}) \le 4 \norm{\bar{\partial}_{t} \theta_h^{m+1/2}}^2+\frac{1}{4}\norm{\gamma_h^{m+1}}^2+  4T\Delta t\sum_{n=1}^m\norm{\bar{\partial}^2_{t} \theta_h^{n}}^2+\frac{\Delta t}{4T}  \sum_{n=1}^m\norm{\gamma_h^{n}}^2.\label{l1}
      \end{align}
      Analogous arguments for the second term on the right-hand side of \eqref{cr} show that
      \begin{align}
          -2\sum_{n=0}^m(\xi^{n+1},\gamma_h^{n+1}-\gamma_h^{n})\le  4\norm{\xi^{m+1}}^2+\frac{1}{4}\norm{\gamma_h^{m+1}}^2+  4T\Delta t\sum_{n=1}^m\norm{\bar{\partial}_{t} \xi^{n+1/2}}^2+\frac{\Delta t}{4T}  \sum_{n=1}^m\norm{\gamma_h^{n}}^2.
\end{align}
Argue similar as in \eqref{cr-p1} and utilize $\eta^{n+1}-\eta^{n}=\Delta t\bar{\partial}_t \eta^{n+1/2}$ from \eqref{P3 phi_c2} to infer
\begin{align*}
   \sum_{n=0}^m(\eta^{n+1},{\rho}_{h}^{n+1}-{\rho}_{h}^{n+1}) =(\eta^{n+1},\rho_h^{m+1})  -\Delta t \sum_{n=1}^m (\bar{\partial}_t \eta^{n+1/2},\rho_h^{n}).
\end{align*}
A Cauchy--Schwarz inequality, \eqref{poinacre}, and a Young's inequality reveal
\begin{align*}
    &(\eta^{n+1},\rho_h^{m+1})  \le \|\eta^{n+1}\|\|\rho_h^{m+1} \|\le C_{P}\|\eta^{n+1}\|\|\widehat{\rho}_h^{m+1} \|_{a,h} \le C_{P}^2\norm{\eta^{m+1}}^2+\frac{1}{4}\norm{\widehat{\rho}_h^{m+1}}^2_{a,h},\\
    & (\bar{\partial}_t \eta^{n+1/2},\rho_h^{n}) \le \|\bar{\partial}_t \eta^{n+1/2}\|\|\rho_h^{n}\| \le C_{P}\|\bar{\partial}_t \eta^{n+1/2}\|\|\widehat{\rho}_h^{n}\|_{a,h} \le TC_{P}^2\|\bar{\partial}_t \eta^{n+1/2}\|^2+\frac{1}{4T}\|\widehat{\rho}_h^{n}\|_{a,h}^2.
\end{align*}
A combination of last three identities 
lead to
\begin{align}
 2\sum_{n=0}^m(\eta^{n+1},{\rho}_{h}^{n+1}-{\rho}_{h}^{n+1})\le  2C_{P}^2\norm{\eta^{m+1}}^2+\frac{1}{2}\norm{\widehat{\rho}_h^{m+1}}^2_{a,h}+ 2\Delta t TC_{P}^2\sum_{n=1}^m\norm{\bar{\partial}_{t} \eta^{n+1/2}}^2+\frac{\Delta t}{2T}  \sum_{n=1}^m\norm{\widehat{\rho}_h^{n}}^2_{a,h}.\label{l2}
\end{align}
A combination of \eqref{l1}-\eqref{l2} in \eqref{cr} and rearrangement of terms results in
\begin{align*}
   \frac{1}{2}\big( \norm{{\gamma}_{h}^{m+1}}^2  +\norm{\widehat{\rho}_{h}^{m+1}}_{a,h}^2\big) &\le  4\big(\norm{ \xi^{m+1}}^2+{T}{\Delta t}\sum_{n=1}^m \norm{ \bar{\partial}_t \xi^{n+1/2}}^2 \big)+ 2{C_{P}^2}\big(\norm{\eta^{m+1}}^2+T{\Delta t}\sum_{n=1}^m \norm{\bar{\partial}_t \eta^{n+1/2}}^2\big)\\
    &\quad+4\big(\norm{\bar{\partial}_{t} \theta_h^{m+1/2}}^2+{T}{\Delta t}\sum_{n=1}^m \norm{\bar{\partial}^2_{t} \theta_h^{n}}^2\big)+ \frac{\Delta t}{ 2T} \sum_{n=1}^m \big(\norm{\gamma_h^{n}}^2+ \norm{\widehat{\rho}_h^{n}}^2_{a,h}\big).
\end{align*}
An application of  Lemma~\ref{Ch2-d-gronwall} with Remark~\ref{rem2.6} reveals
\begin{align}
    &\norm{{\gamma}_{h}^{m+1}}^2  +\norm{\widehat{\rho}_{h}^{m+1}}_{a,h}^2\nonumber\\
    &\lesssim     \norm{ \xi^{m+1}}^2 +{\Delta t}\sum_{n=1}^m \norm{ \bar{\partial}_t \xi^{n+1/2}}^2+ \norm{\eta^{m+1}}^2+{\Delta t}\sum_{n=1}^m \norm{\bar{\partial}_t \eta^{n+1/2}}^2+\norm{\bar{\partial}_{t} \theta_h^{m+1/2}}^2+{\Delta t}\sum_{n=1}^m \norm{\bar{\partial}^2_{t} \theta_h^{n}}^2.\label{et1}
     \end{align}
{Argue as in  \eqref{ref_2_bd_ah}, with $1/2$ replaced by $m+1$, to derive
  \begin{align*}
    \sum_{K\in {\cal T}_{h}}\norm{\nabla^{2}(u^{m+1}-{\cal R}_{K}(\widehat{u}_h ^{m+1}))}^{2}_{K} &\le 2\big(\norm{\widehat{e}_{h}^{m+1}}_{a,h}^2+\sum_{K\in {\cal T}_{h}}\norm{\nabla^{2}(u^{m+1}-{\cal E}_{K}({u} ^{m+1}))}^{2}_{K}\big)\\
    &\le 2\big(2\norm{\widehat{\rho}_{h}^{m+1}}_{a,h}^2+2\norm{\widehat{\theta}_{h}^{m+1}}_{a,h}^2+\sum_{K\in {\cal T}_{h}}\norm{\nabla^{2}(u^{m+1}-{\cal E}_{K}({u} ^{m+1}))}^{2}_{K}\big)
\end{align*}
with  \eqref{split_err_eq1} and a  triangle  inequality in the last step.
}

\medskip
\noindent
Utilize the bounds for the term $\norm{\widehat{\rho}_{h}^{m+1}}_{a,h}^2$ from 
\eqref{et1} and regroup the terms to show
 \begin{align}
    &\sum_{K\in {\cal T}_{h}}\norm{\nabla^{2}(u^{m+1}-{\cal R}_{K}(\widehat{u}_h ^{m+1}))}^{2}_{K} 
    \lesssim  \norm{ \xi^{m+1}}^2 + \norm{\eta^{m+1}}^2 +{\Delta t}\sum_{n=1}^m \norm{ \bar{\partial}_t \xi^{n+1/2}}^2+{\Delta t}\sum_{n=1}^m \norm{\bar{\partial}_t \eta^{n+1/2}}^2\nonumber\\
    &\qquad\qquad\quad+\norm{\widehat{\theta}_{h}^{m+1}}^2+\norm{\bar{\partial}_{t} \theta_h^{m+1/2}}^2+{\Delta t}\sum_{n=1}^m \norm{\bar{\partial}^2_{t} \theta_h^{n}}^2+ \sum_{K\in {\cal T}_{h}}\norm{\nabla^{2}(u^{m+1}-{\cal E}_{K}({u} ^{m+1}))}^{2}_{K}.\label{p1}
    \end{align}
     {\it Step 3. (bound for truncation terms)}
     Recall that $\xi^{m+1}:=\bar{\partial}_{t}u^{m+1/2}-u_{t}^{m+1/2}$ (resp. $\eta^{m+1}=\bar{\partial}_{t}p^{m+1/2}-p_{t}^{m+1/2})$ and utilize the bounds from Lemma~\ref{Sec_lma_1}(b) to obtain 
\begin{equation}
   \norm{ \xi^{m+1}}^2 \lesssim (\Delta t)^4\norm{u_{ttt}}^2_{L^2 (L^2(\Omega))}\; \big(\text{resp. } \norm{ \eta^{m+1}}^2 \lesssim (\Delta t)^4\norm{p_{ttt}}^2_{L^2 (L^2(\Omega))}\big).\label{new15}
\end{equation}
Utilize  \eqref{temp_dis}  to derive  the identity
\begin{align*}
    \bar{\partial}_t \xi^{n+1/2} =\frac{1}{\Delta t} \big( \xi^{n+1}-\xi^{n}\big)=\frac{1}{\Delta t} \big(\bar{\partial}_tu^{n+1/2}-u_t^{n+1/2}-\bar{\partial}_tu^{n-1/2}+u_t^{n-1/2}\big)= \bar{\partial}^2_t u^n-\delta_t u^n_t.
\end{align*}
An application  of \eqref{temp_dis} and  Taylor series reveal that 
\begin{align*}
    \bar{\partial}^2_t u^n &= \frac{u^{n+1}-2u^n+u^{n-1}}{(\Delta t)^2}=u^n_{tt}+\frac{1}{24 (\Delta t)^2}\Big[\int_{t_n}^{t_{n+1}} (t_{n+1}-t)^3u_{tttt}(t)\dt+\int_{t_n}^{t_{n-1}} (t_{n}-t)^3u_{tttt}(t)\dt \Big],\\
    \delta_t u^n_t&= \frac{u^{n+1}-u^{n-1}}{2\Delta t}= u_{tt}^n+\frac{1}{12 \Delta t}\Big[\int_{t_n}^{t_{n+1}} (t_{n+1}-t)^2u_{tttt}(t)\dt+\int_{t_n}^{t_{n-1}} (t_{n}-t)^2u_{tttt}(t)\dt \Big].
\end{align*}
A combination of  last three displayed expressions followed by  $(t_{n}-t) \le 2 \Delta t$ and $(t_{n+1}-t) \le2 \Delta t $ for $t \in [t_{n-1}, t_{n+1}]$, shows
\begin{align*}
    \norm{ \bar{\partial}_t \xi^{n+1/2}}=\|\bar{\partial}^2_t u^n-\delta_t u^n_t\| \lesssim  \Delta t\int_{t_{n-1}}^{t_{n+1}} \norm{u_{tttt}(t)}\dt \lesssim (\Delta t)^{3/2}\Big(\int_{t_{n-1}}^{t_{n+1}} \norm{u_{tttt}(t)}^2\dt\Big)^{1/2}
\end{align*}
with a Cauchy--Schwarz inequality $\int_{t_{n-1}}^{t_{n+1}} \norm{u_{tttt}(t)}\dt  \lesssim \sqrt{\Delta t}\big(\int_{t_{n-1}}^{t_{n+1}} \norm{u_{tttt}(t)}^2\dt\big)^{1/2}$in the last step. 

\medskip \noindent
This (resp. similar arguments replacing $u$ with $p$ in the last inequality) and some basic manipulations show
\begin{equation}
    \Delta t\sum_{n=1}^m \norm{ \bar{\partial}_t \xi^{n+1/2}}^2 \lesssim (\Delta t)^4\norm{u_{tttt}}^2_{L^2 (L^2(\Omega))}\; \big( \text{resp. } \Delta t\sum_{n=1}^m \norm{ \bar{\partial}_t \eta^{n+1/2}}^2 \lesssim (\Delta t)^4\norm{p_{tttt}}^2_{L^2 (L^2(\Omega))}\big). \label{u-ttttt}
\end{equation}
%A combination of \eqref{et1}-\eqref{u-ttttt} reveals
%\begin{align}
  %  \norm{{\gamma}_{h}^{m+1}}^2  +\norm{\widehat{\rho}_{h}^{m+1}}_{a,h}^2 &\lesssim    \norm{\bar{\partial}_{t} \theta_h^{m+1/2}}^2+{\Delta t}\sum_{n=1}^m \norm{\bar{\partial}^2_{t} \theta_h^{n}}^2 +(\Delta t)^4.\label{full-cn}
  %   \end{align}
     {\it Step 4. (bound for spatial approximation terms)}
First note that $\|\bar{\partial}_{t} \widehat{\theta}_h^{m+1/2}\|_{a,h} \le  \Delta t ^{-1}\int_{t_m}^{t_{m+1}}\|\widehat{\theta}_{ht}(t)\|_{a,h} \dt. %\big(\text{resp. } \bar{\partial}_{t} \chi_h^{m+1/2} = \Delta t ^{-1}\int_{t_n}^{t_{n+1}}\chi_{ht}(t) \dt\big)
$ An application of Taylor's series and a Cauchy--Schwarz inequality  show $\norm{\bar{\partial}_t^2 \widehat{\theta}^n_h}^2_{a,h} \le \frac{2}{3}(\Delta t)^{-1}\int_{t_{n-1}}^{t_{n+1}} \norm{\widehat{\theta}_{htt}(t)}^2_{a,h} \dt.$ 
An application of \eqref{poinacre} and the 
last two  displayed inequalities  followed by bounds from Theorem~\ref{lm1_in_dis_tr_11}(c)  %(resp. $\norm{ \bar{\partial}_t^2 \chi^n}^2  \le \frac{2}{3}(\Delta t)^{-1}\int_{t_{n-1}}^{t_{n+1}} \norm{\chi_{tt}(t)}^2 \dt$)
reveal
\begin{align}
&\norm{\widehat{\theta}_{h}^{m+1}}^2+\norm{\bar{\partial}_{t} \theta_h^{m+1/2}}^2+{\Delta t}\sum_{n=1}^m \norm{\bar{\partial}^2_{t} \theta_h^{n}}^2  \le C_P^2\big( \norm{\widehat{\theta}_{h}^{m+1}}_{a,h}^2+\norm{\bar{\partial}_{t} \widehat{\theta}_h^{m+1/2}}^2_{a,h}+{\Delta t}\sum_{n=1}^m \norm{\bar{\partial}^2_{t} \widehat{\theta}_h^{n}}^2_{a,h}\big)\nonumber  \\
&\lesssim \begin{cases}
 h^{2}\big(\norm{u}^2_{L^\infty(H^{3}(\mathcal{T}_h))}+\norm{u_t}^2_{L^\infty(H^{3}(\mathcal{T}_h))}+\norm{u_{tt}}^2_{L^2(H^{3}(\mathcal{T}_h))}\big)\\\qquad+h^{2(1+\beta)}\big(\norm{u}^2_{L^\infty(H^{3+\beta}(\mathcal{T}_h))}+\norm{u_t}^2_{L^\infty(H^{3+\beta}(\mathcal{T}_h))}+\norm{u_{tt}}^2_{L^2(H^{3+\beta}(\mathcal{T}_h))}\big)\;&\text{ for } k = 0,\\[6pt]
h^{2(k+1)}\big(\norm{u}^2_{L^\infty(H^{k+3}(\mathcal{T}_h))}+\norm{u_t}^2_{L^\infty(H^{k+3}(\mathcal{T}_h))}+\norm{u_{tt}}^2_{L^2(H^{k+3}(\mathcal{T}_h))}\big)\;&\text{ for } k \ge 1.
 \end{cases}\label{p3}
\end{align}
%\begin{align}  
%\Delta t \sum_{n=1}^m \norm{\bar{\partial}_t^2 \theta^n_h}^2\le \frac{2}{3} \sum_{n=1}^m\int_{t_{n-1}}^{t_{n}} \norm{\theta_{htt}(t)}^2 \dt +\frac{2}{3} \sum_{n=1}^m\int_{t_{n}}^{t_{n+1}} \norm{\theta_{htt}(t)}^2 \dt \le \frac{4}{3} \norm{\theta_{htt}(t)}^2_{L^2 (L^2(\Omega))}.\label{p3-ist}
%\big(
%\text{resp. }  \Delta t \sum_{n=1}^m \norm{\bar{\partial}_t^2 \chi^n}^2\le \frac{2}{3} \norm{\chi_{tt}(t)}^2_{L^2 (L^2(\Omega))} \big).
%\end{align}
%The similar arguments also leads to the bounds 
%\begin{equation}
 %\norm{ \theta_{ht}^{m+1/2}}^2 \lesssim \norm{\theta_{ht}}^2_{L^\infty (L^2(\Omega))}, \text{ and } \Delta t\sum_{n=1}^m \norm{ \delta_t \theta^n_{ht}}^2\lesssim \norm{\theta_{tt}(t)}^2_{L^2 (L^2(\Omega))} .
%\end{equation}
  Argue similar as in \eqref{new_1last} with $1/2$ replaced by $m+1$ to verify
\begin{align}
    \sum_{K\in {\cal T}_{h}}\norm{\nabla^{2}(u^{m+1}-{\cal E}_{K}({u} ^{m+1}))}^{2}_{K}\lesssim h^{2(k+1)} \norm{u}^2_{L^\infty({t_m},t_{m+1};H^{k+3}(\mathcal{T}_h))}\lesssim h^{2(k+1)} \norm{u}^2_{L^\infty(H^{k+3}(\mathcal{T}_h))}.\label{p2}
\end{align}
{\it Step 5. (energy estimates for $u$)}
A combination of \eqref{p1}-\eqref{p2} leads to
 \begin{align}
    \sum_{K\in {\cal T}_{h}}\norm{\nabla^{2}(u^{m+1}-{\cal R}_{K}(\widehat{u}_h ^{m+1}))}^{2}_{K}\lesssim
 \begin{cases}
 h^{2}+h^{2(1+\beta)}+(\Delta t)^{4}\;&\text{ for }\; k = 0,\\
h^{2(k+1)}+(\Delta t)^{4}\;&\text{ for }\; k \ge 1.
 \end{cases}
    \end{align}
%A combination of  \eqref{full-cn}-\eqref{u-ttttt} and  the approximation properties of $\theta_h(t)$ from  Theorem~\eqref{lm1_in_dis_tr_11}(c) completes the proof.
{\it Step 6. ($L^2$ estimate for $p$)}
 First, we apply \eqref{split-p} and a  triangle inequality  to verify
$
 \norm{p^{m+1/2}-p_h^{m+1/2}}^{2} \le 2\big( \norm{{\chi}_{h}^{m+1}}^{2}+\norm{{\gamma}_{h}^{m+1}}^{2}\big)$. 
This and bound for $\norm{{\gamma}_{h}^{m+1}}^{2}$ from \eqref{et1} with $\norm{\bar{\partial}_{t} \theta_h^{m+1/2}}^2+{\Delta t}\sum_{n=1}^m \norm{\bar{\partial}^2_{t} \theta_h^{n}}^2 \lesssim \norm{\bar{\partial}_{t} \widehat{\theta}_h^{m+1/2}}_{a,h}^2+{\Delta t}\sum_{n=1}^m \norm{\bar{\partial}^2_{t} \widehat{\theta}_h^{n}}^2_{a,h}$ from \eqref{poinacre} yields 
\begin{align}
    \norm{p^{m+1/2}-p_h^{m+1/2}}^{2} &\lesssim    \norm{{\chi}_{h}^{m+1}}^{2}+  \norm{ \xi^{m+1}}^2 +{\Delta t}\sum_{n=1}^m \norm{ \bar{\partial}_t \xi^{n+1/2}}^2+ \norm{\eta^{m+1}}^2\nonumber\\
    &\quad +{\Delta t}\sum_{n=1}^m \norm{\bar{\partial}_t \eta^{n+1/2}}^2+\norm{\bar{\partial}_{t} \widehat{\theta}_h^{m+1/2}}_{a,h}^2+{\Delta t}\sum_{n=1}^m \norm{\bar{\partial}^2_{t} \widehat{\theta}_h^{n}}^2_{a,h}.\label{pener}
\end{align}
 From \eqref{split-p} and Theorem~\ref{lm1_in_dis_tr_11}(d), we have
 \begin{equation}
     \norm{{\chi}_{h}^{m+1}}^{2}=\|p^{m+1}-\Pi_{h}^{k+2}(p^{m+1})\|^2\lesssim h^{2(k+3)}\|p\|^2_{L^\infty(H^{k+3}(\mathcal{T}_h))}.\label{new13}
 \end{equation}
This and the bounds form \eqref{new15}-\eqref{p3} lead to
\begin{align*}
    \norm{p^{m+1/2}-p_h^{m+1/2}}^{2} & \lesssim
 \begin{cases}
 h^{2}+h^{2(1+\beta)}+(\Delta t)^{4}\;&\text{ for }\; k = 0,\\
h^{2(k+1)}+(\Delta t)^{4}\;&\text{ for }\; k \ge 1.
 \end{cases}
\end{align*}
{\it Step 7. ($L^2$-estimates for $p$ in HHO-A scheme)}
 Argue similar to Step 4 to show $$\|\bar{\partial}_{t} \theta_h^{m+1/2}\|\le  \Delta t ^{-1}\int_{t_m}^{t_{m+1}}\|\theta_{ht}(t)\| \dt \text{ and } \norm{\bar{\partial}_t^2 \theta^n_h}^2 \le \frac{2}{3}(\Delta t)^{-1}\int_{t_{n-1}}^{t_{n+1}} \norm{\theta_{htt}(t)}^2 \dt.$$
 These two inequalities and  Theorem~\ref{m1_H^2_bound} yield
%  %(resp. $\norm{ \bar{\partial}_t^2 \chi^n}^2  \le \frac{2}{3}(\Delta t)^{-1}\int_{t_{n-1}}^{t_{n+1}} \norm{\chi_{tt}(t)}^2 \dt$)
%reveals that Apply similar arguments as in \eqref{p4} followed by theorem~\ref{m1_H^2_bound} as 
\begin{align*}
  \norm{\bar{\partial}_{t} \theta_h^{m+1/2}}^2+{\Delta t}\sum_{n=1}^m \norm{\bar{\partial}^2_{t} \theta_h^{n}}^2 \lesssim \begin{cases}
 h^{4}\big(\norm{u_t}^2_{L^\infty(H^{3}(\mathcal{T}_h))}+h^{2{\beta}}\norm{u_t}^2_{L^\infty(H^{3+\beta}(\mathcal{T}_h))}+\norm{\Delta^2 u_t}^2_{L^\infty(L^2(\Omega))}\big)\\
 + h^{4}\big(\norm{u_{tt}}^2_{L^2(H^{3}(\mathcal{T}_h))}+h^{2{\beta}}\norm{u_{tt}}^2_{L^2(H^{3+\beta}(\mathcal{T}_h))}+\norm{\Delta^2 u_{tt}}^2_{L^2(L^2(\Omega))}\big)\;&\text{for } k = 0,\\
h^{2(k+3)}\big(\norm{u_t}^2_{L^\infty(H^{k+3}(\mathcal{T}_h))}+\norm{\Delta^2 u_t}^2_{L^\infty(H^{k-1}(\mathcal{T}_h))}\big)\\[8pt]
+h^{2(k+3)}\big(\norm{u_{tt}}^2_{L^2(H^{k+3}(\mathcal{T}_h))}+\norm{\Delta ^2 u_{tt}}^2_{L^2(H^{k-1}(\mathcal{T}_h))}\big)\;&\text{for } k \ge 1.
 \end{cases}
\end{align*}
This, the bounds from \eqref{new15}, \eqref{u-ttttt}, and  \eqref{new13} in \eqref{pener} lead to
$$\displaystyle \norm{ p^{m+1}- p_h^{m+1} }^2 \lesssim \begin{cases}
    h^4+ h^{2(2+\beta)}+(\Delta t)^4 \quad\;\;\; &\text{ for } k=0,\\
    h^{2(k+3)} +(\Delta t)^4\quad &\text{ for }  k\ge1.
\end{cases}$$
This concludes the proof of the theorem.
\qed

\medskip
\noindent
{The next theorem establishes optimal order $L^2$-error estimates for the primal variable $u$ and is motivated from analysis of \cite{deka}. However, the reference {\it does not  discuss the convergence} analysis for new variable $p$ in the mentioned reference. The proof is inspired by energy arguments commonly used for second-order time discretizations, but requires several additional ingredients due to the HHO spatial approximation and the structure of the coupled residual terms. The main idea is to first reformulate the error equation into a form involving only the auxiliary error component ${\rho}_h^n$, so that all residual contributions appear solely through $L^2$ inner products. An auxiliary sequence $\widehat{\sigma}_h^n$ defined through a discrete time accumulation of ${\rho}_h^{n+1/2}$, is then introduced to convert the resulting relation into a telescoping energy identity. This construction plays a crucial role in avoiding direct estimates of accumulated error terms and allows the error contributions to be separated into spatial approximation and temporal truncation components. These contributions are then bounded independently using projection estimates, truncation error bounds, and the discrete Gronwall lemma, leading to a uniform-in-time optimal order estimate.}

\medskip
\noindent
\uline{\textbf{Proof of Theorem \ref{err_lma}}.} 
The proof is organized into six steps. The error equation \eqref{zeta} is first 
modified in \textit{Step~1} into a suitable form, from which a key 
inequality is derived in \textit{Step~2}. The right-hand side 
of  key 
inequality is then decomposed into spatial approximation and temporal truncation terms in \textit{Step~3}, 
followed by the space approximation bounds in \textit{Step~4} and 
the truncation error bounds in \textit{Step~5}. The proof is 
completed in \textit{Step~6} by consolidating these estimates.

\medskip
\noindent
{{\text{{\it Step 1. 
 (a combined error equation)}}}} First observe that ${\gamma}_{h}^{1/2}=\frac{\Delta t}{2}\bar{\partial}_{t}{\gamma}_{h}^{1/2}$ holds from initialization $\gamma_h^0=0$. Also for $1 \leq n\leq N-1$, the definitions in \eqref{temp_dis} reveal that ${\gamma}_{h}^{n+1/2} ={ \frac{\Delta t}{2} \sum_{k=0}^{n}}\bar{\partial}_{t}{\gamma}_{h}^{k+1/2}+{\frac{\Delta t}{2} \sum_{k=0}^{n-1}}\bar{\partial}_{t}{\gamma}_{h}^{k+1/2}$.
 A combination of these two identities with the property that
 ${v}_h\in \mathcal{P}^{k+2}(\mathcal{T}_h)$
 %for any $\widehat{v}_{h} = (v_{h}, v_{\mathcal{F}_{h}}, \zeta_{\mathcal{F}_{h}})\in \widehat{V}_{h}^{0}$
 allows us to  rewrite \eqref{zeta} as: for all $\widehat{v}_{h} = (v_{h}, v_{\mathcal{F}_{h}}, \zeta_{\mathcal{F}_{h}})\in \widehat{V}_{h}^{0},$
\begin{align}\label{Rho_eq_2}
(\bar{\partial}_{t} {\rho}_{h}^{n+1/2},v_h)=
\begin{cases}
(-\bar{\partial}_{t}{\theta}_{h}^{1/2}+\xi^{1}+\frac{\Delta t}{2}\bar{\partial}_{t}{\gamma}_{h}^{1/2},v_h)\;\; \qquad \qquad \qquad \qquad \qquad  \quad \;\;\text{for } n=0,\\
(-\bar{\partial}_{t}{\theta}_{h}^{n+1/2}+\xi^{n+1}+{\displaystyle \frac{\Delta t}{2} \sum_{k=0}^{n}}\bar{\partial}_{t}{\gamma}_{h}^{k+1/2}+{\displaystyle \frac{\Delta t}{2} \sum_{k=0}^{n-1}}\bar{\partial}_{t}{\gamma}_{h}^{k+1/2},v_h)\;\; \text{for } 1 \leq n\leq N-1.
\end{cases}
\end{align}
Recall   \eqref{sec_lm41_new} for $n=0$ (resp. for $k=0,1,2,\cdots, n$) to obtain 
\begin{align}
&(\bar{\partial}_{t}{\gamma}_{h}^{1/2},v_h)=
-a_h(\widehat{\rho}_{h}^{1/{2}},\widehat{v}_{h})+(\eta^1,v_h). \\
& (\bar{\partial}_{t}{\gamma}_{h}^{k+1/2},{v}_{h})+a_{h}(\widehat{\rho}_{h}^{k+1/2},\widehat{v}_{h})=(\eta^{k+1},v_{h}) \nonumber 
\end{align}
%The equation, follows from \eqref{sec_lm41_new}, holds for  any $ k=0,1,2,\cdots,n$.
Take a summation over $k=0$ to $n$ and $k=0$ to $n-1$ in the last equation and then sum up the resultant equations to obtain: 
\begin{align}
& (\sum_{k=0}^{n}\bar{\partial}_{t}{\gamma}_{h}^{k+1/2}+\sum_{k=0}^{n-1}\bar{\partial}_{t}{\gamma}_{h}^{k+1/2},v_h)=-a_h(\sum_{k=0}^{n}\widehat{\rho}_{h}^{k+\frac{1}{2}}+\sum_{k=0}^{n-1}\widehat{\rho}_{h}^{k+\frac{1}{2}},\widehat{v}_{h})+( \eta^{n+1}+2{\sum_{k=0}^{n-1}}\eta^{k+1},v_h).
\end{align}
Introduce a sequence $\{\widehat{\sigma}_{h}^{n}\}_{n=0}^{N-1}$ as:    $\widehat{\sigma}_{h}^{0}=0$ and $\widehat{\sigma}_{h}^{n}=\Delta t \sum_{k=0}^{n-1}\widehat{\rho}_{h}^{k+\frac{1}{2}}$ for $1\leq n\leq N-1$ . This and the definitions in \eqref{temp_dis} lead to
\begin{align}\label{Sec_sigm_3}
\widehat{\sigma}_{h}^{n+1/2}=\frac{1}{2}\big(\widehat{\sigma}_{h}^{n+1}+\widehat{\sigma}_{h}^{n}\big)=
\begin{cases}
  \frac{\Delta t}{2} \widehat{\rho}_{h}^{1/{2}}&\mbox{for}\; n=0,\\
  {\displaystyle \frac{\Delta t}{2}\sum_{k=0}^{n}}\widehat{\rho}_{h}^{k+\frac{1}{2}}+{\displaystyle \frac{\Delta t}{2}\sum_{k=0}^{n-1}}\widehat{\rho}_{h}^{k+\frac{1}{2}}&\mbox{for}\; 1\leq n\leq N-1.
\end{cases}
\end{align}
We also define a sequence $\{\epsilon_{h}^{n}\}_{n=0}^{N-1}$ by
\begin{align}
\epsilon_{h}^{n}=
\begin{cases}-\bar{\partial}_{t}\theta_{h}^{1/2}+\xi^{1}+\frac{\Delta t}{2} \eta^{1}&\mbox{for}\; n=0,\\
-\bar{\partial}_{t}\theta_{h}^{n+1/2}+\xi^{n+1}+\frac{\Delta t}{2} \eta^{n+1}+{\displaystyle \Delta t\sum_{k=0}^{n-1}}\eta^{k+1}&\mbox{for } 1\leq n\leq N-1.
\end{cases} \label{epsilon}
\end{align}
A combination of \eqref{Rho_eq_2}-\eqref{epsilon} and elementary manipulations reveal
\begin{equation}
(\bar{\partial}_{t}\rho_{h}^{n+1/2},v_{h})+a_{h}(\widehat{\sigma}_{h}^{n+1/2},\widehat{v}_{h})=(\epsilon_{h}^{n},v_{h})\;\; \text{ for all } 0\leq n\leq N-1 \text{ and } \widehat {v}_h \in \widehat{V}_{h}^{0}.\label{eq_8_ne_1}
\end{equation}
{{\text{{\it Step 2. 
 (key inequality)}}}}  Multiply \eqref{eq_8_ne_1} by  $2\Delta t$ then select $\widehat{v}_{h}=\widehat{\rho}_h^{n+1/2}$ as a test function and  utilize the identity $\widehat{\rho}_h^{n+1/2}=\bar{\partial}_{t} \widehat{\sigma}_{h}^{n+1/2}$ from  definition of $\widehat{\sigma}_h^n$,  to obtain
  \begin{equation*}
2\Delta t(\bar{\partial}_{t}\rho_{h}^{n+1/2},\rho_h^{n+1/2})+2\Delta ta_{h}(\widehat{\sigma}_{h}^{n+1/2},\bar{\partial}_{t} \widehat{\sigma}_{h}^{n+1/2})=2\Delta t(\epsilon_{h}^{n},\rho_h^{n+1/2})\;\; \text{ for all } 0\leq n\leq N-1 .
\end{equation*}
Invoke the identities $2\Delta t(\rho_h^{n+1/2},\bar{\partial}_{t}\rho_h^{n+1/2})= \|\rho_{h}^{n+1}\|^2- \|\rho_{h}^{n}\|^2$ and $2\Delta t\, a_{h}(\widehat{\sigma}_{h}^{n+{1}/{2}},\bar{\partial}_{t}{\widehat{\sigma}}_{h}^{n+1/2})=\|\widehat{\sigma}_{h}^{n+1}\|_{a,h}^2- \|\widehat{\sigma}_{h}^{n}\|_{a,h}^2$ from \eqref{temp_dis} and \eqref{norm} to show
 $$\norm{\rho_{h}^{n+1}}^{2}-\norm{\rho_{h}^{n}}^{2}+\norm{\widehat{\sigma}_{h}^{n+1}}^{2}_{a,{h}}-\norm{\widehat{\sigma}_{h}^{n}}^{2}_{a,{h}}=2 \Delta t(\epsilon_{h}^{n},\rho_{h}^{n+1/2}).$$
 A summation for $n=0,1,\cdots,m$, the observations $\widehat{\sigma}_{h}^{0}=0$ from  definition, and $\gamma_h^0=0$ from initialization yield
\begin{equation}
\norm{\rho_{h}^{m+1}}^{2}+\norm{\widehat{\sigma}_{h}^{m+1}}^{2}_{a,{h}}=2 \Delta t\sum_{n=0}^{m}(\epsilon_{h}^{n},\rho_{h}^{n+1/2})\quad\text{ for any $0\leq m \leq N-1$}.\nonumber
\end{equation}
Ignore the non-negative term $\norm{\widehat{\sigma}_{h}^{m+1}}^{2}_{a,{h}}$ on the left-hand side then apply  a Cauchy--Schwarz inequality $(\epsilon_{h}^{n},\rho_{h}^{n+1/2}) $ $\le \|\epsilon_{h}^{n}\|\|\rho_{h}^{n+1/2}\|$ followed by  Young's inequality $\|\epsilon_{h}^{n}\|\|\rho_{h}^{n+1/2}\| \le 2T\|\epsilon_{h}^{n}\|^2+\frac{1}{2T} \|\rho_{h}^{n+1/2}\|^2$ on the right-hand side to establish
%Next, using the facts that $\widehat{\sigma}_{h}^{0}=0$  along with Cauchy--Schwarz and Young's inequality in the above equation leads to
\begin{align}
\norm{\rho_{h}^{m+1}}^{2}&\leq4T\Delta t\sum_{n=0}^{m}\norm{\epsilon_{h}^{n}}^{2}+ \frac{\Delta t}{T}\sum_{n=0}^{m}\norm{\rho_{h}^{n+1/2}}^{2}\nonumber.
\end{align}
Utilize \eqref{P3 phi_c1 } and a triangle inequality to derive
$$ \frac{\Delta t}{T}\sum_{n=0}^{m}\norm{\rho_{h}^{n+1/2}}^{2} \le  \frac{\Delta t}{2T}\sum_{n=0}^{m}\big(\norm{\rho_{h}^{n+1}}^{2}+\norm{\rho_{h}^{n}}^{2}\big) \le  \frac{\Delta t}{2T}\norm{\rho_{h}^{m+1}}^{2}+ \frac{\Delta t}{T}\sum_{n=0}^{m}\norm{\rho_{h}^{n}}^{2} \le  \frac{1}{2}\norm{\rho_{h}^{m+1}}^{2}+\frac{\Delta t}{T}\sum_{n=0}^{m}\norm{\rho_{h}^{n}}^{2}$$
with $\Delta t \le T$ in the last step.
A combination of last two displayed inequalities leads to
\begin{align*}
\frac{1}{2}\norm{\rho_{h}^{m+1}}^{2}&\leq4T\Delta t\sum_{n=0}^{m}\norm{\epsilon_{h}^{n}}^{2}+\frac{\Delta t}{T}\sum_{n=0}^{m}\norm{\rho_{h}^{n}}^{2}.
\end{align*}
An application of Lemma~\ref{Ch2-d-gronwall} with Remark~\ref{rem2.6} reveals
\begin{equation}
\norm{\rho_{h}^{m+1}}^{2}\lesssim \Delta t\sum_{n=0}^{m}\norm{\epsilon_{h}^{n}}^{2}.\label{Sec_4_rho_1}
\end{equation}
%with the constant absorbed in "$\lesssim$" depending of $T$ linearly and Euler number $e.$\\
{{\text{{\it Step 3. 
 (control for RHS of \eqref{Sec_4_rho_1})}}}} Utilize \eqref{epsilon},  triangle inequality, and some basic manipulations to obtain 
\begin{align}
    \Delta t\sum_{n=0}^{m}\norm{\epsilon_{h}^{n}}^{2} &\lesssim 
\Delta t \sum_{n=0}^{m}\norm{\bar{\partial}_{t}\theta_{h}^{n+1/2}}^2+\Delta t \sum_{n=0}^{m}\norm{\xi^{n+1}}^2+{(\Delta t)^3}\sum_{n=0}^{m}\sum_{k=0}^{n}\norm{\eta^{k+1}}^2 \text{ for any $0\leq m \leq N-1$}.\label{ist-in-last}
\end{align}
%The application of 
%\begin{align*}
%\epsilon_{h}^{n}&=-\bar{\partial}_{t}\theta_{h}^{n+1/2}+\xi^{n+1}+\frac{\Delta t}{2}\eta^{n+1}+\Delta t\sum_{k=0}^{n-1}\eta^{k+\frac{1}{2}}=
%-\bar{\partial}_{t}\theta_{h}^{n}+\xi^{n+1}+\frac{\Delta t}{2}\eta^{n+1}+\frac{\Delta t}{2}\sum_{k=0}^{n-1}\eta^{k+{1}}+\frac{\Delta t}{2}\sum_{k=0}^{n-1}\eta^{k}\nonumber\\
%&=-\bar{\partial}_{t}\theta_{h}^{n+1/2}+\xi^{n+1}+\frac{\Delta t}{2}\sum_{k=1}^{n}\eta^{k}+\frac{\Delta t}{2}\sum_{k=0}^{n}\eta^{k}.
%\end{align*}
%For $0 \le n \le M-1$, an application of Cauchy--Schwarz and triangle inequalities leads to
%\begin{align*}
%\norm{\epsilon_{h}^{n}}^{2}&\le \norm{\bar{\partial}_{t}\theta_{h}^{n+1/2}}^{2}+\norm{\xi^{n+1}}^{2}+\frac{(\Delta t)^{2}}{4}\norm{\sum_{k=1}^{n}\eta^{k}}^{2}+\frac{(\Delta t)^{2}}{4}\norm{\sum_{k=0}^{n}\eta^{k}}^{2}\nonumber\\
%&\leq \norm{\bar{\partial}_{t}\theta_{h}^{n+1/2}}^{2}+\norm{\xi^{n+1}}^{2}+\frac{(\Delta t)^{2}}{2}M\sum_{k=0}^{M-1}\norm{\eta^{k}}^{2}\leq \norm{\bar{\partial}_{t}\theta_{h}^{n+1/2}}^{2}+\norm{\xi^{n+1}}^{2}+\frac{T}{2}\big(\Delta t \sum_{k=0}^{M-1}\norm{\eta^{k}}^{2}\big),
%\end{align*}
%where the fact $M \Delta t= T$ in the last step. Sum it for $n=0,\cdots,M-1$ and multiply both sides by $\Delta t$ to reveal
%\begin{align}
%\Delta t\sum_{n=0}^{M-1}\norm{\epsilon_{h}^{n}}^{2} \leq \Delta t\sum_{n=0}^{M-1}\norm{\bar{\partial}_{t}\theta_{h}^{n+1/2}}^{2}+\Delta t\sum_{n=0}^{M-1}\norm{\xi^{n+1}}^{2}+\Delta t\sum_{n=0}^{M-1}\frac{T}{2}\big(\Delta t\norm{\sum_{k=0}^{M-1}\eta^{k}}^{2}\big).\label{Sec_ip_9_y}
%\end{align}
{{\text{{\it Step 4. (spatial bounds)}}}}
 Note that  $\|\bar{\partial}_{t}\theta_{h}^{n+1/2} \|={(\Delta t)^{-1}}\|\int_{t_n}^{t_{n+1}} \theta_{ht}\dt\|\le {(\Delta t)^{-1/2}}\big(\int_{t_n}^{t_{n+1}} \|\theta_{ht}\|^2\dt\big)$ by a Cauchy--Schwarz inequality.  This and the  bounds from  Theorem~\ref{m1_H^2_bound} lead to
 \begin{align}
 {\Delta t}\sum_{n=0}^m \norm{\bar{\partial}_{t} \theta_h^{n+1/2}}^2 \lesssim \begin{cases}
  h^{4}\big(\norm{u_{t}}^2_{L^2(H^{3}(\mathcal{T}_h))}+h^{2{\beta}}\norm{u_{t}}^2_{L^2(H^{3+\beta}(\mathcal{T}_h))}+\norm{\Delta^2 u_{t}}^2_{L^2(L^2(\Omega))}\big)\;\text{ for }\; k = 0,\\[6pt]
h^{2(k+3)}\big(\norm{u_{t}}^2_{L^2(H^{k+3}(\mathcal{T}_h))}+\norm{\Delta ^2 u_{t}}^2_{L^2(H^{k-1}(\mathcal{T}_h))}\big)\;\text{ for }\; k \ge 1.
 \end{cases}\label{con}
\end{align}
{{\text{{\it Step 5. (truncation error bounds)}}}}
It follows from the definitions of $\xi^{n+1}$ and $\eta^{n+1}$, and their bounds from  Lemma \eqref{Sec_lma_1}(b) that
\begin{align}
&\Delta t\sum_{n=0}^{m}\norm{\xi^{n+1}}^{2}\lesssim (\Delta t)^{4} \norm{u_{ttt}}^{2}_{L^{2}(L^{2}(\Omega))},\\
&(\Delta t)^2 \sum_{n=0}^{m}\big(\Delta t\sum_{k=0}^{n}\norm{\eta^{k+1}}^{2}\big)\lesssim (m+1)(\Delta t)^{6} \norm{p_{ttt}}^{2}_{L^{2}(L^{2}(\Omega))} \lesssim (\Delta t)^{4} \norm{p_{ttt}}^{2}_{L^{2}(L^{2}(\Omega))}   \label{xi}
\end{align}
with $m\Delta t \le T$ and $ \Delta t\le T$ in the last inequality.

\medskip\noindent
%The last  equation together with the fact that $\Delta t\sum_{n=1}^{m} \norm{u_{tttt}}^{2}_{L^{2}(L^{2}(\Omega))} \le T \norm{u_{tttt}}^{2}_{L^{2}(L^{2}(\Omega))}  $ and $\Delta t \le T$ leads to
%\begin{equation}
%  \frac{(\Delta t)^3}{4}\norm{\eta^1}^2+ 
  %  \Delta t \sum_{n=1}^{m}\big(\frac{(\Delta t)^2}{4}\norm{\eta^{n+1}}^2+(\Delta t)^2 \sum_{k=0}^{n-1}\norm{\eta^{k+1}}^2\big)\lesssim (\Delta t)^{4} \norm{u_{tttt}}^{2}_{L^{2}(L^{2}(\Omega))}.\label{last-eqn}
%\end{equation}
{\it Step 6. (consolidation)}
A combination of \eqref{Sec_4_rho_1}-\eqref{xi}  shows
\begin{align*}
  \norm{\rho_{h}^{m+1}}^2   \lesssim 
\begin{cases}
    h^4+ h^{4(2+\beta)}+(\Delta t)^4 & \text{ for } k=0,\\
    h^{2(k+3)} +(\Delta t)^4 & \text{ for }  k\ge1.
\end{cases}
\end{align*}
Also note that Theorem~\ref{m1_H^2_bound} yields
 \begin{align*}
 \norm{\theta_h^{m+1}}^2 \lesssim \begin{cases}
  h^{4}\big(\norm{u}^2_{L^\infty(H^{3}(\mathcal{T}_h))}+h^{2{\beta}}\norm{u}^2_{L^\infty(H^{3+\beta}(\mathcal{T}_h))}+\norm{\Delta^2 u}^2_{L^\infty(L^2(\Omega))}\big)\;\text{ for }\; k = 0,\\[6pt]
h^{2(k+3)}\big(\norm{u}^2_{L^\infty(H^{k+3}(\mathcal{T}_h))}+\norm{\Delta ^2 u}^2_{L^\infty(H^{k-1}(\mathcal{T}_h))}\big)\;\text{ for }\; k \ge 1.
 \end{cases}
\end{align*}
An application of \eqref{split_err_eq1} with \eqref{pi-pro} followed by a  triangle inequality  shows $\norm{u^{m+1}- u_h^{m+1}}^2 \lesssim \|u^{m+1}-\Pi^{k+2}_{h}(u^{m+1})\|^2+\norm{\rho_h^{m+1}}^2+ \norm{\theta_h^{m+1}}^2$. This, the bound from \eqref{2ju} for the first and estimates from the last two displayed inequalities for the last two terms on the right-hand side, concludes the proof.
%\end{equation}
%\begin{equation}

%A combination of \eqref{Sec_ip_9_y}-\eqref{xi} in \eqref{Sec_4_rho_1} along with the observation $\norm{\zeta^{n+1/2}}^2_{L^2(V_h)} \le \max_{0\leq n\leq M}\norm{\widehat{\sigma}_{h}^{l}}^{2}_{a,{h}} $ and the application of triangle inequality
%\begin{align*}
% \max_{0\leq n\leq M}\norm{u^{m+1/2}- U^{m+1/2} }^2 &+ \norm{ u^{m+1/2}- U^{m+1/2} }^2_{L^2(V_h)} \\
 %&\le \max_{0\leq n\leq M}\norm{\zeta^{n+1}}^2+\max_{0\leq n\leq M}\norm{\rho^{n+1}}^2  + \norm{ \zeta^{n+1/2}}^2_{L^2(V_h)}+ \norm{ \rho^{n+1/2} }^2_{L^2(V_h)}.   
%\end{align*}
%followed by utilization of bounds from \eqref{app_prop}
 %completes the rest of the proof.
\section{Numerical studies}\label{num_sec}
This section discusses numerical results that validate the theoretical estimates derived in Lemma~\ref{P1 lemma_on_norm_initial}, Theorems \ref{P1 implicit_Th2}-\ref{err_L2_new},  and Theorems \ref{err_thm}-\ref{err_lma}. The convergence rates and efficiency of the HHO scheme in the space direction and Newmark scheme \eqref{hwave_int_cond}-\eqref{hwave_algorithm} and Crank-Nicolson scheme \eqref{HHO_Crank_Nico_sc_1} in the time direction are demonstrated. Two $h$-refined mesh families (Cartesian and polygonal Voronoi-like) meshes have been used for the HHO method in the space direction, see Figure \ref{Exm1_fig1}$(a)$-$(b)$.
The spatial discretization employs the HHO scheme with polynomial orders $k\in \{0,1,2,3\}$, using meshes composed of  16, 64, 256, 1024,  and  4096 elements.

\medskip\noindent
The errors in $L^{\infty}(H^2)$ seminorm and $L^{\infty}(L^{2})$ norm are denoted as
\begin{align*}
  \norm{\mathfrak{E}_h^{u}}_{L^{\infty}(H^{2})}& :=\max_{1\leq n\leq m}\big(\sum_{K\in {\cal T}_{h}}\norm{\nabla^{2}(u^{n}-{\cal R}_{K}(\widehat{u}^{n}_{K}))}_{K}^2\big)^{1/2},\\
 \norm{\mathfrak{E}_h^{u}}_{L^{\infty}(L^{2})}& :=\max_{1\leq n\leq m}\big(\sum_{K\in {\cal T}_{h}}\norm{u^{n}-{\cal R}_{K}(\widehat{u}^{n}_{K})}_{K}^2\big)^{1/2},\\ 
     \norm{\mathfrak{E}_h^{p}}_{L^{\infty}(L^{2})}& :=\max_{1\leq n\leq m}\big(\sum_{K\in {\cal T}_{h}}\norm{p^{n}-{p}^{n}_{K}}_{K}^2\big)^{1/2}.
\end{align*}

\medskip \noindent
\begin{example}[\bf{Convergence for smooth solution}]\rm{
Choose the domain $\Omega\times(0,T]:= (0,1)^{2}\times (0,0.1].$ The solution of the problem \eqref{P1 strong_form}-\eqref{P1 strong_icbc} is given by
%\begin{equation*}
$u(x,y,t)=\exp(-t)x^{2}y^{2}(1-x)^{2}(1-y)^{2}$. Table \ref{summary1} summarizes the details of the tables in which the results of this numerical experiment are presented.
\begin{figure}[h]
  	\begin{subfigure}{0.5\textwidth}	
  		\centering
  		\includegraphics[width=5.00cm, height=4.20cm]{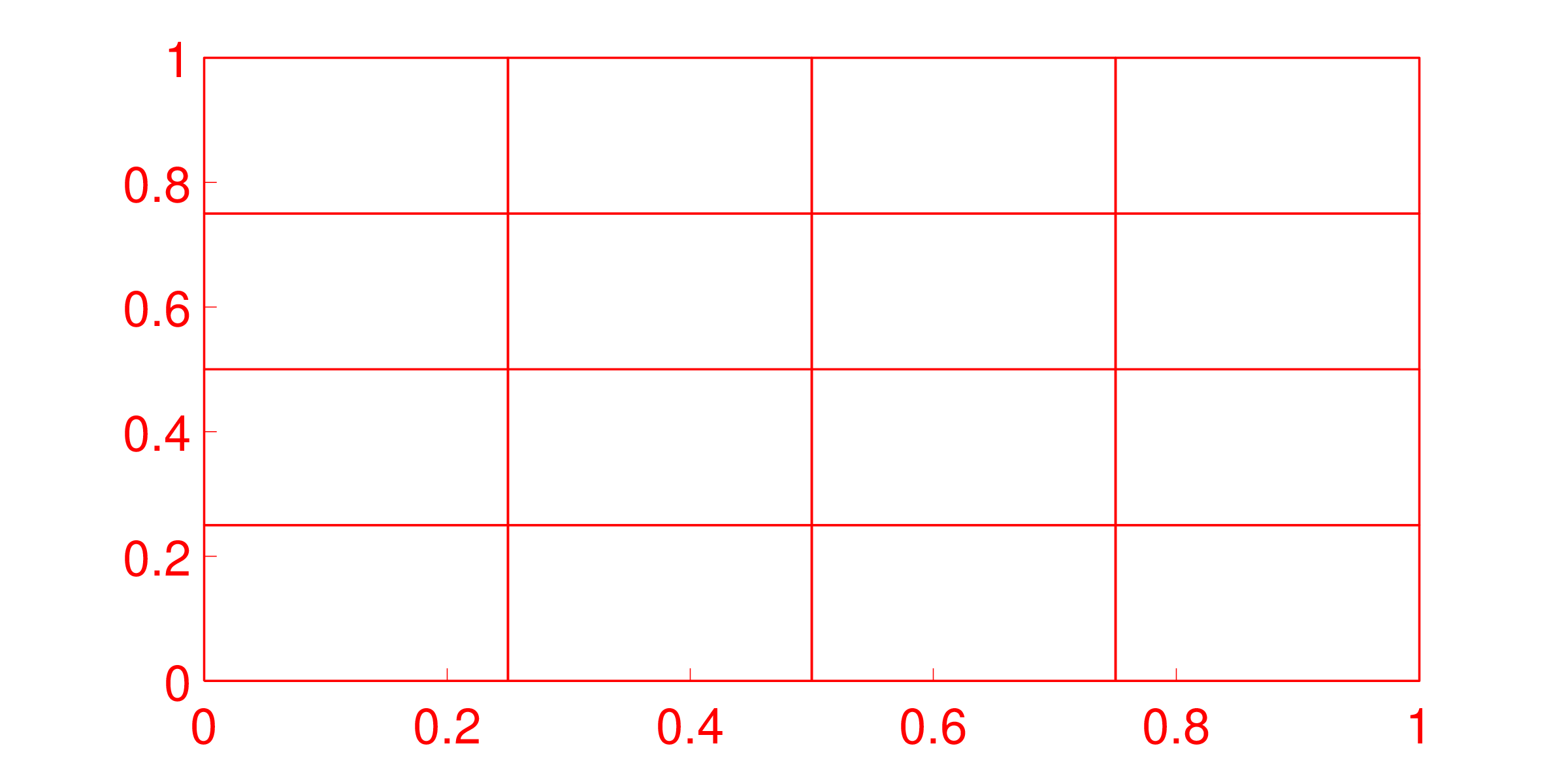}
  		\caption{}
  	\end{subfigure}
  	\begin{subfigure}{0.5\textwidth}	
  		\centering
  		\includegraphics[width=5.00cm, height=4.20cm]{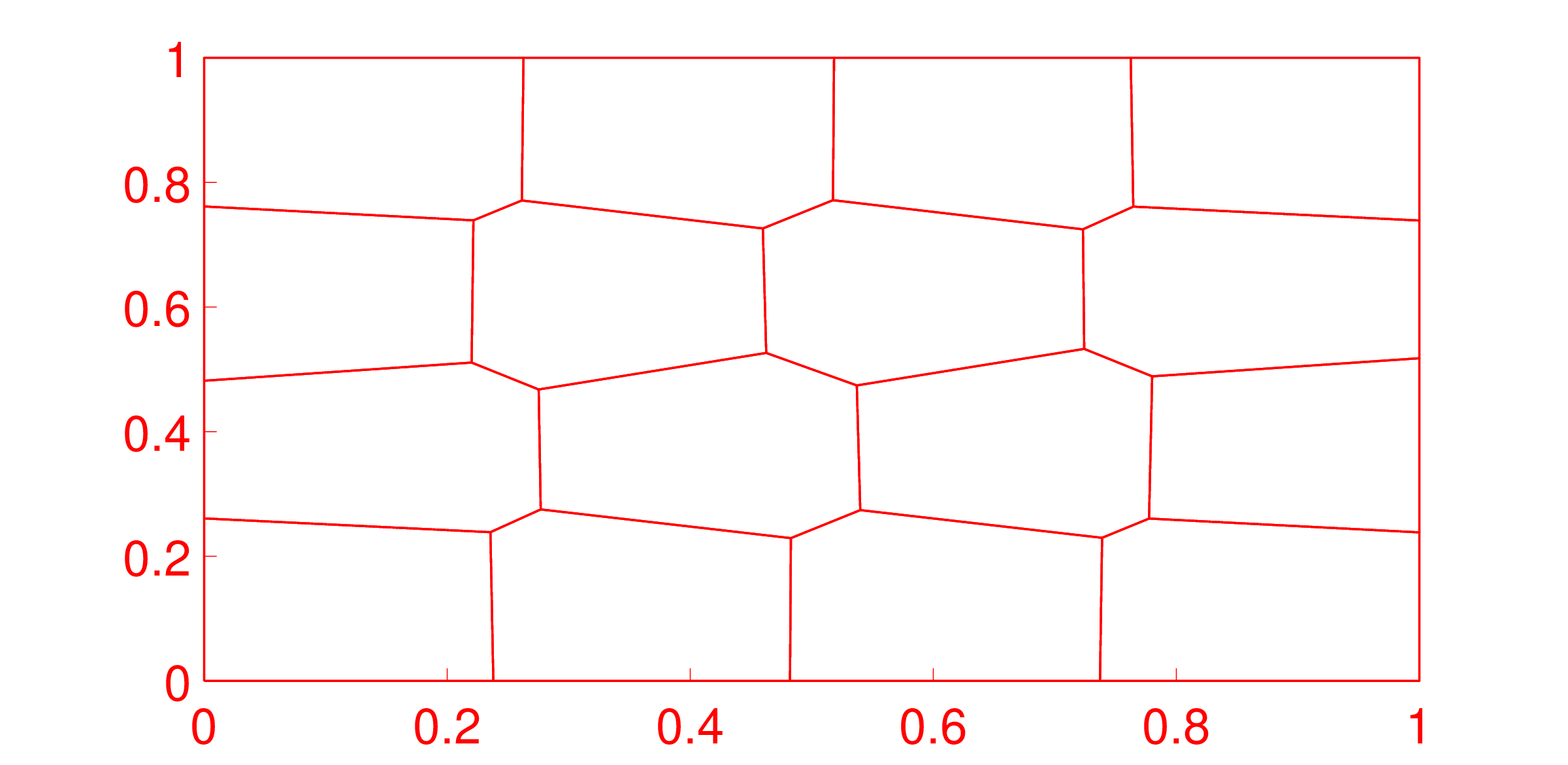}
  		\caption{ }
  	\end{subfigure}
  	\caption{The Cartesian $(a)$ and polygonal $(b)$ meshes with $16$ elements.}
  	\label{Exm1_fig1}
  \end{figure}
\begin{table}[h]
 \caption{Summary of tables for the numerical experiment.} \label{summary1}
\centering
\begin{tabular}{|l|l|}
\hline
 ~~~~\mbox{Newmark scheme}& ~~~~\mbox{Crank-Nicolson scheme}\\ \hline
  \mbox{ Cartesian} ~~~~~~\mbox{ Polygonal}& \mbox{Cartesian} ~~~~~~~~~~~~~ \mbox{ Polygonal}\\ 
 Table \ref{Table_A_1}~~~~~~~~Table \ref{Table_A_2}& Tables \ref{Table_A_3}-\ref{Table_A_5}~~~~~~Tables \ref{Table_A_4}-\ref{Table_A_5}\\
% Table \ref{Rect_exm_1_L_2}~~~~~~~ Table \ref{Poly_exm_1_L_2}& Table \ref{Crank_HHO_A_rect_1_L2}~~~~~~Table \ref{Crank_HHO_A_rect_1_L21}\\
\hline
\end{tabular}
\end{table}

% \begin{table}[h]
%  \caption{Summary of tables for HHO-B scheme.} \label{summary2}
% \centering
% \begin{tabular}{|l|l|l|}
% \hline
% \mbox{Error} & ~~~~\mbox{Newmark scheme}& ~~~~\mbox{Crank-Nicolson scheme}\\ \hline
%  & \mbox{ Cartesian} ~~~~~~\mbox{ Polygonal}& \mbox{Cartesian} ~~~~ \mbox{ Polygonal}\\ 
% $L^{\infty}(H^{2})$ & Table \ref{Rect_exm_1_H_2_HHO_B}~~~~~~~~Table \ref{Poly_exm_1_H_2_HHO_B}&Table \ref{Crank_HHO_B_rect_1_energy_norm}~~~~~~Table \ref{Crank_HHO_B_rect_1_energy_norm1}\\
% $L^{\infty}(L^{2})$& Table \ref{Rect_exm_1_L_2_HHO_B}~~~~~~~ Table \ref{Poly_exm_1_L_2_HHO_B}& Table \ref{Crank_HHO_B_rect_1_L_2_norm}~~~~~~Table \ref{Crank_HHO_B_poly_L_2_norm}\\
% \hline
% \end{tabular}
% \end{table}

\medskip \noindent
Each table illustrate the experimental order of convergence (EOC) and discrete errors in various norms. The convergence rates for the Cartesian and polygonal meshes are consistent with the theoretical estimates derived in  Lemma~\ref{P1 lemma_on_norm_initial}, Theorems \ref{P1 implicit_Th2}-\ref{err_L2_new}, and  Theorem~\ref{err_thm}-\ref{err_lma}. In Table \ref{Table_A_1} (resp. Table \ref{Table_A_2})
%(resp. Tables \ref{Poly_exm_1_H_2}-\ref{Poly_exm_1_L_2}) 
 errors in $L^{\infty}(H^{2})$ and $L^{\infty}(L^{2})$)-norm for the HHO-A/HHO-B Newmark scheme \eqref{hwave_int_cond}-\eqref{hwave_algorithm} are described for Cartesian (resp. polygonal) meshes. For HHO-A/HHO-B Crank-Nicolson scheme \eqref{HHO_Crank_Nico_sc_1}, results are presented in Tables \ref{Table_A_3}-\ref{Table_A_5} (resp. Table \ref{Table_A_4}-\ref{Table_A_5}) for Cartesian (resp. polygonal) meshes.
 For both schemes, we observe that optimal convergence rate $\mathcal{O}(h^{k+1})$, $k\in \{0,1,2,3\}$ is achieved for $L^{\infty}(H^{2})$-error estimate. The $L^{\infty}(L^2)$-error converges at the optimal rate $\mathcal{O}(h^{k+3}),$ except at $k=0,$ where the rate is only $\mathcal{O}(h^{2})$.
 \begin{table}[H]
\setlength{\tabcolsep}{4pt} % Reduce column separation
\renewcommand{\arraystretch}{0.6} % Reduce row separation
\scriptsize % Reduce font size
\caption{{\it Newmark scheme}. Convergence and error profiles  with Cartesian meshes for  $u$.}
	\centering
  \begin{tabular}{|l|l l|l l||l l|l l|}
     \hline
    \multirow{2}{*}{\mbox{\# of cells}} &
      \multicolumn{4}{c||}{HHO-A} &
      \multicolumn{4}{c|}{HHO-B} \\
       \cline{2-9}
    & $\norm{\mathfrak{E}_h^{u}}_{L^{\infty}(H^{2})}$ & EOC & $\norm{\mathfrak{E}_h^{u}}_{L^{\infty}(L^{2})}$ & EOC & $\norm{\mathfrak{E}_h^{u}}_{L^{\infty}(H^{2})}$ & EOC & $\norm{\mathfrak{E}_h^{u}}_{L^{\infty}(L^{2})}$ & EOC \\
    \hline
    \multicolumn{9}{|c|}{$k=0$} \\
    \hline
~~~~~16&6.80e-02&---&1.54e-02&---&2.35e-01&---&1.97e-02&---\\
 ~~~~~64&3.46e-02&0.97&3.92e-03&1.98&
 1.26e-01&0.90&5.15e-03&1.90\\
 ~~~~~256&1.74e-02&0.98&9.87e-04&1.99&6.50e-02&0.96&1.31e-03&1.97\\
 ~~~~~1024&8.78e-03&0.99&2.47e-04&2.00&3.27e-02&0.99&3.31e-04&1.99\\
 ~~~~~4096&4.40e-03&1.00&6.19e-05&2.00&1.64e-02&1.00&8.32e-05&2.00\\
 \hline
   \multicolumn{9}{|c|}{$k=1$} \\
      \hline
   ~~~~~16&5.68e-03&---&6.76e-05&---&2.12e-02&---&2.75e+00&---\\
   ~~~~~64&1.46e-03&1.96&4.37e-06&3.95&5.93e-03&1.85&1.89e-01&3.86\\
   ~~~~~256&3.69e-04&1.98&2.76e-07&3.99&1.54e-03&1.94&1.22e-02&3.95\\
   ~~~~~1024&9.26e-05&2.00&1.73e-08&4.00&3.90e-04&1.98&7.76e-04&3.98\\
   ~~~~~4096&2.31e-05&2.00&1.08e-09&4.00&9.79e-05&1.99&4.89e-05&3.99\\
 \hline
    \multicolumn{9}{|c|}{$k=2$} \\
      \hline
      ~~~~~16&1.29e-05&---&3.67e-06&---&1.76e-04&---&2.53e-03&---\\
      ~~~~~64&1.71e-06&2.91&1.16e-07&4.98&2.21e-05&2.99&8.09e-05&4.97\\
      ~~~~~256&2.20e-07&2.96&3.65e-09&4.99&2.77e-06&2.99&2.55e-06&4.98\\
      ~~~~~1024&2.78e-08&2.98&1.14e-10&5.00&3.47e-07&3.00&8.11e-08&4.98\\
      ~~~~~4096&3.57e-09&2.96&3.30e-12&5.11&4.35e-08&3.00&2.53e-09&5.00\\
      \hline
       \multicolumn{9}{|c|}{$k=3$} \\
      \hline
      ~~~~~16&4.79e-07&---&2.97e-01&---&5.60e-06&---&5.01e-03&---\\
      ~~~~~64&3.10e-08&3.95&7.65e-05&5.96&3.50e-07&4.00&8.20e-05&5.93\\
      ~~~~~256&1.96e-09&3.98&1.72e-08&6.05&2.18e-08&4.00&1.26e-06&6.01\\
      ~~~~~1024&1.23e-10&4.00&2.68e-10&6.00&1.36e-09&4.00&1.95e-08&6.02\\
      ~~~~~4096&7.70e-12&4.00&4.10e-12&6.00&8.43e-11&4.01&3.03e-10&6.00\\
      \hline
  \end{tabular}
  \label{Table_A_1}
\end{table}
\begin{table}[H]
\setlength{\tabcolsep}{4pt} % Reduce column separation
\renewcommand{\arraystretch}{0.6} % Reduce row separation
\scriptsize % Reduce font size
\caption{{\it Newmark scheme}. Convergence and error profiles  with {polygonal meshes} for  $u$.}
\centering
\begin{tabular}{|l|l l|l l||l l|l l|}
    \hline
    \multirow{2}{*}{\mbox{\# of cells}} &
      \multicolumn{4}{c||}{HHO-A} &
      \multicolumn{4}{c|}{HHO-B} \\
       \cline{2-9}
   & $\norm{\mathfrak{E}_h^{u}}_{L^{\infty}(H^{2})}$ & EOC & $\norm{\mathfrak{E}_h^{u}}_{L^{\infty}(L^{2})}$ & EOC & $\norm{\mathfrak{E}_h^{u}}_{L^{\infty}(H^{2})}$ & EOC & $\norm{\mathfrak{E}_h^{u}}_{L^{\infty}(L^{2})}$ & EOC \\
    \hline
    \multicolumn{9}{|c|}{$k=0$} \\
    \hline
~~~~~16 & 4.31e-02 & --- & 1.77e-01 & --- & 1.56e-01 & --- & 8.54e-02 & --- \\
~~~~~64 & 2.07e-02 & 1.05 & 4.61e-02 & 1.93 & 7.92e-02 & 0.98 & 2.24e-02 & 1.93 \\
~~~~~256 & 1.02e-02 & 1.01 & 1.74e-02 & 1.93 & 3.97e-02 & 0.99 & 5.72e-03 & 1.96 \\
~~~~~1024&5.10e-03 & 1.00 & 2.94e-03 & 1.99 & 1.99e-02 & 1.00 & 1.44e-03 & 1.98 \\
~~~~~4096 &2.55e-03&1.00&7.35e-04&2.00&9.98e-03&1.00&3.58e-04&2.00\\
    \hline
    \multicolumn{9}{|c|}{$k=1$} \\
    \hline
~~~~~16 & 5.71e-03 & --- & 4.27e-04 & --- & 7.68e-03 & --- & 2.55e-04 & --- \\
~~~~~64 & 1.35e-03 & 2.08 & 2.75e-05 & 3.96 & 1.95e-03 & 1.97 & 1.64e-05 & 3.96 \\
~~~~~256 & 3.31e-04 & 2.02 & 2.26e-08 & 4.99 & 4.92e-04 & 1.99 & 1.03e-06 & 3.99 \\
~~~~~1024 & 8.25e-05 & 2.00 & 1.09e-07 & 3.99 & 1.23e-04 & 2.00 & 6.48e-08 & 4.00 \\
~~~~~4096&2.05e-05&2.00&6.82e-09&4.00&3.07e-05&2.00&4.03e-09&4.00\\
    \hline
    \multicolumn{9}{|c|}{$k=2$} \\
    \hline
~~~~~16 & 3.16e-05 & --- & 2.26e-05 & --- & 8.93e-05 & --- & 2.36e-05 & --- \\
~~~~~64 & 2.59e-06 & 3.60 & 7.19e-07 & 4.97 & 1.11e-05 & 2.99 & 7.57e-07 & 4.97 \\
~~~~~256 & 2.64e-07 & 3.29 & 3.62e-08 & 5.73 & 1.40e-06 & 3.00 & 2.38e-08 & 4.99 \\
~~~~~1024 & 3.08e-08 & 3.09 & 7.12e-10 & 4.99 & 1.75e-07 & 3.00 & 7.43e-10 & 5.00 \\
~~~~~4096&3.86e-09&3.00&2.22e-11&5.00&2.18e-08&3.00&2.31e-11&5.00\\
    \hline
    \multicolumn{9}{|c|}{$k=3$} \\
    \hline
~~~~~16 & 3.46e-06 & --- & 1.26e-01 & --- & 2.66e-06 & --- & 9.60e-06 & --- \\
~~~~~64 & 1.36e-07 & 4.66 & 1.02e-04 & 5.13 & 1.67e-07 & 3.99 & 1.42e-07 & 6.07 \\
~~~~~256 & 6.68e-09 & 4.35 & 2.26e-08 & 5.73 & 1.04e-08 & 4.00 & 2.15e-09 & 6.04 \\
~~~~~1024 & 3.82e-10 & 4.12 & 3.52e-10&6.00& 6.54e-10 & 4.00 & 3.31e-11 & 6.02 \\
~~~~~4096&2.38e-11&4.00&5.47e-12 & 6.00&4.10e-11&4.00&5.10e-13&6.02\\
    \hline
\end{tabular}
\label{Table_A_2}
\end{table}
\begin{table}[H]
\setlength{\tabcolsep}{4pt} % Reduce column separation
\renewcommand{\arraystretch}{0.6} % Reduce row separation
\scriptsize % Reduce font size
\caption{{\it Crank-Nicolson scheme}. Convergence and error profiles with Cartesian meshes for  $u$.}
\centering
\begin{tabular}{|l|l l|l l||l l|l l|}
    \hline
    \multirow{2}{*}{\mbox{\# of cells}} &
      \multicolumn{4}{c||}{HHO-A} &
      \multicolumn{4}{c|}{HHO-B} \\
       \cline{2-9}
    & $\norm{\mathfrak{E}_h^{u}}_{L^{\infty}(H^{2})}$ & EOC & $\norm{\mathfrak{E}_h^{u}}_{L^{\infty}(L^{2})}$ & EOC & $\norm{\mathfrak{E}_h^{u}}_{L^{\infty}(H^{2})}$ & EOC & $\norm{\mathfrak{E}_h^{u}}_{L^{\infty}(L^{2})}$ & EOC \\
    \hline
    \multicolumn{9}{|c|}{$k=0$} \\
    \hline
~~~~~16 & 1.29e+00 & --- & 1.53e-01 & --- & 1.86e-01 & --- & 1.58e-02 & --- \\
~~~~~64 & 6.46e-01 & 1.00 & 3.85e-02 & 1.99 & 9.31e-02 & 1.00 & 4.00e-03 & 1.99 \\
~~~~~256 & 3.23e-01 & 1.00 & 9.64e-03 & 1.99 & 4.65e-02 & 1.00 & 1.00e-03 & 1.99 \\
~~~~~1024 & 1.61e-01 & 1.00 & 2.41e-03 & 2.00 & 2.32e-02 & 1.00 & 2.50e-04 & 2.00 \\
~~~~~4096&8.05e-02&1.00&6.01e-04&2.00&1.16e-02&1.00&6.25e-05&2.00\\
    \hline
    \multicolumn{9}{|c|}{$k=1$} \\
    \hline
~~~~~16 & 8.41e-02 & --- & 5.71e-04 & --- & 8.94e-03 & --- & 4.70e-04 & --- \\
~~~~~64 & 2.10e-02 & 1.99 & 3.58e-05 & 3.99 & 2.21e-03 & 2.01 & 3.04e-05 & 3.95 \\
~~~~~256 & 5.26e-03 & 1.99 & 2.24e-06 & 3.99 & 5.53e-04 & 2.00 & 1.91e-06 & 3.98 \\
~~~~~1024 & 1.31e-03 & 2.00 & 1.40e-07 & 4.00 & 1.38e-04 & 2.00 & 1.20e-07 & 4.00 \\
~~~~~4096&3.26e-04&2.00&8.72e-09&4.00&3.43e-05&2.00&7.48e-09&4.00\\
    \hline
    \multicolumn{9}{|c|}{$k=2$} \\
    \hline
~~~~~16 & 1.10e-02 & --- & 1.84e-05 & --- & 8.13e-04 & --- & 3.03e-05 & --- \\
~~~~~64 & 1.39e-03 & 2.99 & 5.80e-07 & 5.00 & 1.02e-04 & 2.99 & 9.82e-07 & 4.94 \\
~~~~~256 & 1.74e-04 & 2.99 & 1.81e-08 & 5.00 & 1.28e-05 & 2.99 & 3.09e-08 & 4.98 \\
~~~~~1024 & 2.17e-05 & 3.00 & 5.68e-10 & 5.00 & 1.60e-06 & 3.00 & 9.70e-10 & 5.00 \\
~~~~~4096&2.70e-06&3.00&1.78e-11&5.00&1.99e-07&3.00&3.04e-11&5.00\\
    \hline
    \multicolumn{9}{|c|}{$k=3$} \\
    \hline
~~~~~16 & 4.05e-04 & --- & 2.88e-04 & --- & 2.95e-05 & --- & 7.01e-05 & --- \\
~~~~~64 & 2.54e-05 & 3.99 & 4.58e-06 & 5.97 & 1.85e-06 & 3.99 & 1.10e-06 & 5.99 \\
~~~~~256 & 1.59e-06 & 4.00 & 6.93e-08 & 6.04 & 1.16e-07 & 3.99 & 1.70e-08 & 6.01 \\
~~~~~1024 & 9.94e-08 & 4.00 & 1.06e-09 & 6.02 & 7.25e-09 & 4.00 & 2.63e-10 & 6.00 \\
~~~~~4096 &6.18e-09&4.00&1.65e-11&6.00&4.52e-10&4.00&4.10e-12&6.00\\
    \hline
\end{tabular}
\label{Table_A_3}
\end{table}
\vspace{-0.5cm}
\begin{table}[H] 
\setlength{\tabcolsep}{4pt} % Reduce column separation
\renewcommand{\arraystretch}{0.9} % Reduce row separation
\scriptsize % Reduce font size
\caption{{\it Crank-Nicolson scheme}. Convergence and error profiles with polygonal meshes for $u$.} 
\centering \begin{tabular}{|l|l l|l l||l l|l l|} \hline
\multirow{2}{*}{\mbox{\# of cells}} & \multicolumn{4}{c||}{HHO-A} & \multicolumn{4}{c|}{HHO-B} \\ 
\cline{2-9}
& $\norm{\mathfrak{E}_h^{u}}_{L^{\infty}(H^{2})}$ & EOC & $\norm{\mathfrak{E}_h^{u}}_{L^{\infty}(L^{2})}$ & EOC & $\norm{\mathfrak{E}_h^{u}}_{L^{\infty}(H^{2})}$ & EOC & $\norm{\mathfrak{E}_h^{u}}_{L^{\infty}(L^{2})}$ & EOC \\
\hline 
\multicolumn{9}{|c|}{$k=0$} \\ 
\hline
~~~~~16 & 1.24e+00 & --- & 1.53e-01 & --- & 1.25e+00 & --- & 1.09e-01 & --- \\ 
~~~~~64 & 6.23e-01 & 1.00 & 3.85e-02 & 1.99 & 6.28e-01 & 0.99 & 2.74e-02 & 1.99 \\
~~~~~256 & 3.11e-01 & 1.00 & 9.65e-03 & 1.99 & 3.14e-01 & 0.99 & 6.86e-03 & 2.00 \\ 
~~~~~1024 & 1.55e-01 & 1.00 & 2.41e-03 & 2.00 & 1.57e-01 & 1.00 & 1.71e-03 & 2.00 \\
~~~~~4096&7.72e-02&1.00&5.99e-04&2.00&7.84e-02&1.00&4.25e-04&2.00\\
\hline
\multicolumn{9}{|c|}{$k=1$} \\
\hline 
~~~~~16 & 8.13e-02 & --- & 5.71e-04 & --- & 6.96e-02 & --- & 2.03e-04 & --- \\ 
~~~~~64 & 2.03e-02 & 1.99 & 3.58e-05 & 3.99 & 1.74e-02 & 1.99 & 1.27e-05 & 3.99 \\ 
~~~~~256 & 5.09e-03 & 1.99 & 2.24e-06 & 3.99 & 4.37e-03 & 1.99 & 7.98e-07 & 3.99 \\
~~~~~1024 & 1.27e-03 & 2.00 & 1.40e-07 & 4.00 & 1.09e-03 & 2.00 & 4.98e-08 & 4.00 \\ 
~~~~~4096&3.16e-04&2.00&8.73e-09&4.00&2.72e-04&2.00&3.10e-09&4.00\\
\hline
\multicolumn{9}{|c|}{$k=2$} \\ 
\hline
~~~~~16 & 1.10e-02 & --- & 1.84e-05 & --- & 5.31e-03 & --- & 6.57e-06 & --- \\
~~~~~64 & 1.39e-03 & 2.99 & 5.80e-07 & 4.99 & 6.66e-04 & 3.00 & 2.06e-07 & 4.99 \\
~~~~~256 & 1.74e-04 & 2.99 & 1.81e-08 & 4.99 & 8.33e-05 & 3.00 & 6.44e-09 & 4.99 \\ 
~~~~~1024 & 2.18e-05 & 2.99 & 5.68e-10 & 5.00 & 1.04e-05 & 3.00 & 2.03e-10 & 4.99 \\
~~~~~4096&2.72e-06&3.00&8.89e-12&5.00&1.30e-06&3.00&6.33e-12&5.00\\
\hline \multicolumn{9}{|c|}{$k=3$} \\
\hline
~~~~~16 & 4.05e-04 & --- & 1.05e-05 & --- & 1.94e-04 & --- & 6.49e-04 & --- \\
~~~~~64 & 2.54e-05 & 3.99 & 1.81e-07 & 5.86 & 1.21e-05 & 3.99 & 1.06e-05 & 5.94 \\
~~~~~256 & 1.59e-06 & 3.99 & 2.82e-09 & 6.00 & 7.61e-07 & 3.99 & 1.67e-07 & 5.98 \\ 
~~~~~1024 & 9.96e-08 & 4.00 & 4.34e-11 & 6.02 & 4.76e-08 & 4.00 & 2.63e-09 & 5.99 \\
~~~~~4096&6.20e-09&4.00&6.74e-13&6.00&2.96e-09&4.00&4.10e-11&6.00\\
\hline
\end{tabular}
\label{Table_A_4}
\end{table}
\vspace{-0.5cm}
\begin{table}[H]
\setlength{\tabcolsep}{4pt} % Reduce column separation
\renewcommand{\arraystretch}{0.6} % Reduce row separation
\scriptsize % Reduce font size
\caption{{\it Crank-Nicolson scheme}. Convergence and error profiles for Cartesian and polygonal meshes for velocity ($p$).}
\centering
\begin{tabular}{|l|l l|l l||l l|l l|}
\hline
\multirow{2}{*}{\mbox{\# of cells}} & \multicolumn{4}{c||}{Cartesian Meshes} & \multicolumn{4}{c|}{Polygonal Meshes} \\
\cline{2-9}
& \multicolumn{2}{c|}{HHO-A} & \multicolumn{2}{c||}{HHO-B} & \multicolumn{2}{c|}{HHO-A} & \multicolumn{2}{c|}{HHO-B} \\
\cline{2-9}
& $\norm{\mathfrak{E}_h^{p}}_{L^{\infty}(L^{2})}$ & EOC & $\norm{\mathfrak{E}_h^{p}}_{L^{\infty}(L^{2})}$ & EOC & $\norm{\mathfrak{E}_h^{p}}_{L^{\infty}(L^{2})}$ & EOC & $\norm{\mathfrak{E}_h^{p}}_{L^{\infty}(L^{2})}$ & EOC \\
\hline
\multicolumn{9}{|c|}{$k=0$} \\
\hline
~~~~~16 & 2.84e-01 & --- & 1.54e-01 & --- & 3.78e-01 & --- & 2.43e-01 & --- \\
~~~~~64 & 7.26e-02 & 1.97 & 3.91e-02 & 1.98 & 9.60e-02 & 1.97 & 6.16e-02 & 1.98 \\
~~~~~256 & 1.82e-02 & 1.99 & 9.83e-03 & 1.99 & 2.41e-02 & 1.99 & 1.54e-02 & 1.99 \\
~~~~~1024 & 4.57e-03 & 2.00 & 2.46e-03 & 2.00 & 6.03e-03 & 2.00 & 3.87e-03 & 2.00 \\
~~~~~4096&1.14e-03&2.00&6.12e-04&2.00&1.50e-03&2.00&9.65e-04&2.00\\
\hline
\multicolumn{9}{|c|}{$k=1$} \\
\hline
~~~~~16 & 3.03e-03 & --- & 1.90e-03 & --- & 4.27e-03 & --- & 2.42e-03 & --- \\
~~~~~64 & 1.92e-04 & 3.98 & 1.20e-04 & 3.98 & 2.70e-04 & 3.98 & 1.52e-04 & 3.98 \\
~~~~~256 & 1.21e-05 & 3.99 & 7.56e-06 & 3.99 & 1.69e-05 & 3.99 & 9.58e-06 & 3.99 \\
~~~~~1024 & 7.57e-07 & 4.00 & 4.73e-07 & 4.00 & 1.05e-06 & 4.00 & 5.99e-07 & 4.00 \\
~~~~~4096&4.72e-08&4.00&2.94e-08&4.00&6.53e-08&4.00&3.75e-08&4.00\\
\hline
\multicolumn{9}{|c|}{$k=2$} \\
\hline
~~~~~16 & 9.74e-05 & --- & 6.11e-05 & --- & 1.36e-04 & --- & 7.75e-05 & --- \\
~~~~~64 & 3.09e-06 & 4.97 & 1.93e-06 & 4.98 & 4.32e-06 & 4.98 & 2.44e-06 & 4.98 \\
~~~~~256 & 9.71e-08 & 4.99 & 6.07e-08 & 4.99 & 1.35e-07 & 4.99 & 7.66e-08 & 4.99 \\
~~~~~1024 & 3.04e-09 & 5.00 & 1.89e-09 & 5.00 & 4.24e-09 & 5.00 & 2.39e-09 & 5.00 \\
~~~~~4096&9.46e-11&5.00&5.89e-11&5.00&1.32e-10&5.00&7.42e-11&5.00\\
\hline
\multicolumn{9}{|c|}{$k=3$} \\
\hline
~~~~~16 & 3.96e-05 & --- & 2.51e-05 & --- & 1.61e-02 & --- & 2.77e-05 & --- \\
~~~~~64 & 6.22e-07 & 5.99 & 4.06e-07 & 5.95 & 2.65e-04 & 5.92 & 4.55e-07 & 5.93 \\
~~~~~256 & 9.66e-09 & 6.01 & 6.28e-09 & 6.01 & 4.18e-06 & 5.99 & 7.09e-09 & 6.00 \\
~~~~~1024 & 1.50e-10 & 6.00 & 9.80e-11 & 6.00 & 6.53e-08 & 6.00 & 1.10e-10 & 6.00\\
~~~~~4096&2.34e-12&6.00&1.53e-12&6.00&9.99e-10&6.00&1.72e-12&6.00\\
\hline
\end{tabular}
\label{Table_A_5}
\end{table}
 }
 \end{example}
\begin{example}[\bf{Vibration in heterogeneous media}]\label{exm2_sec}\rm{This example demonstrates the HHO scheme for a  modified wave problem following an approach similar to \cite{MR4444402}. Consider
\begin{equation*}
u_{tt}(x,t) + \Delta(c(x)\Delta u(x,t))= f(x,t), \quad (x,t) \in \Omega \times \left(0,T \right], 
\end{equation*}
with homogeneous clamped boundary conditions and initial conditions
\begin{equation*}
u=\frac{\partial u}{\partial n} =0 \text{ on }\partial \Omega \times \left(0,T \right];\quad u(x,0)=u_0(x) \;\; \mbox{and}\;\; u_t(x,0)=u_1(x) \text{ in }\Omega.
\end{equation*}
Here, the positive coefficient $c$ characterizes the rigidity of the material. The domain is $\Omega = (-1,1)^2$ with material heterogeneity:
\begin{equation*}
 \quad c(x)=\begin{cases}
     1 \quad \text{if} \; x_{2} <0.2\\
     9 \quad \text{if} \; x_{2}\geq 0.2,
 \end{cases}
 \quad T=\frac{3}{100},
 \quad f=0,
\end{equation*}
and the initial data
\begin{equation*}
u_{0}= \frac{1}{5}\exp(-|10x|^{2})(1-x_{1}^{2})^{2}(1-x_{2}^{2})^{2} \quad \text{and} \quad u_{1}=0.
\end{equation*}
\noindent%
\begin{figure}[H]
\begin{minipage}[t]{0.44\textwidth}%
{The system dynamics are initiated by a regularized Dirac impulse give by $c(x)$ defined above. Following the referenced work \cite{MR4444402,amir}, we define a control region (or a sensor) $\Omega_{c}:= (0.75-l_{c},0.75+l_{c})\times(-l_{c}, l_{c})$, with $l_{c}=\frac{1}{32}$ and evaluate the  sensor signal  as $u_{c}(t):= \int_{\Omega_{c}} u_{h}(x,t)dx.$ For the Newmark-HHO scheme with $k=0$, we analyze and compare the signal at different arrival times at the sensor for different spatial grids and time step sizes. Simulations are performed on fixed cubic meshes of $50\times 50$ and $100\times 100$ grids, using a time step $\Delta t= T/N$, where $N$ is the number of time sub-intervals. The numerical results are shown in Figure~\ref{Exm2_fig2} for $50\times 50$ where as the same plot is  obtained for $100\times 100$ grid and hence its plot is omitted.}
\end{minipage}
\hspace{0.08\textwidth}
\begin{minipage}[t]{0.44\textwidth}%
\vspace{0pt}%
\begin{tikzpicture}[scale=2.2]
  % Draw the main square domain
  \draw[thick] (-1,-1) rectangle (1,1);
  
  % Draw the division line at x2 = 0.2
  \draw[thick] (-1,0.2) -- (1,0.2);
  
  % Fill the lower region (c = 1) with light color
  \fill[color=white!90!blue] (-1,-1) rectangle (1,0.2);
  
  % Fill the upper region (c = 9) with slightly darker color
  \fill[color=white!70!red] (-1,0.2) rectangle (1,1);
  
  \fill[color=black!50, opacity=0.6] (0.719,-0.031) rectangle (0.781,0.031);
  \draw[black, thick] (0.719,-0.031) rectangle (0.781,0.031);

  % Add the division line label
  \node[right] at (1.05,0.36) {$x_2 = 0.2$};

  \draw[red,thick] (-1,0.2) -- (1.3,0.2);
  % Add region labels
  \node[align=center] at (0,-0.4) {$c(x_1, x_2) = 1$};
  \node[align=center] at (0,0.6) {$c(x_1, x_2) = 9$};
  
  % Add corner coordinates
  \node[below left] at (-1,-1) {$(-1,-1)$};
  \node[below right] at (1,-1) {$(1,-1)$};
  \node[above left] at (-1,1) {$(-1, 1)$};
  \node[above right] at (1,1) {$(1, 1)$};

 \draw[->, line width=1.2pt] (1, -0.01) -- (1.2, -0.01);
  \node[left=0.15cm] at (1.8,-0.01) {$x_2 = 0$};

  \draw[->, line width=1.2pt] (0,1) -- (0,1.2);
  \node[below=0.15cm] at (0,1.5) {$x_1 = 0$};

  % Arrow pointing to control region Ωc
  \draw[->, line width=1.2pt, green] (0.05, 0.03) -- (0.69, 0.008);
  \node[above right=0.1cm, green] at (-0.23,-0.15) {$\Omega_c$};
\end{tikzpicture}%
\caption{(Example~\ref{exm2_sec}). Piecewise constant function $c(x_{1},x_{2})$ in the domain $\Omega=(-1,1)^2.$}
\label{Fig_6.2_exm2}
\end{minipage}%%
\end{figure}
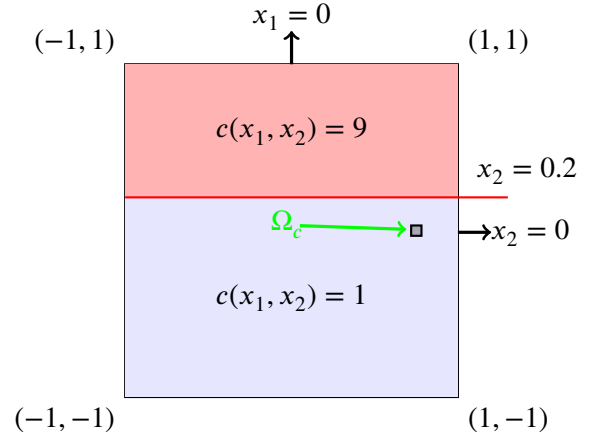
} 
\noindent
The numerical results obtained for this benchmark problem demonstrate that the proposed Newmark fully discrete HHO scheme performs robustly in a non-homogeneous domain setting. In particular, the sensor signal $u_c(t)$ exhibits negligible variation between simulations performed with larger and smaller time step sizes, providing strong numerical evidence that the scheme is unconditionally stable with respect to the choice of $\Delta t$. Furthermore, the close agreement between the solutions computed on the coarser $50\times 50$ and finer $100\times 100$ spatial grids confirms that the method is spatially efficient, yielding reliable approximations even at moderate mesh resolutions. Taken together, these observations highlight the practical efficiency of the proposed scheme for wave propagation problems in non-homogeneous media, where both temporal and spatial flexibility are of significant computational importance.
\begin{figure}[H]
\centering	
 \includegraphics[width=12.00cm, height=6.80cm]{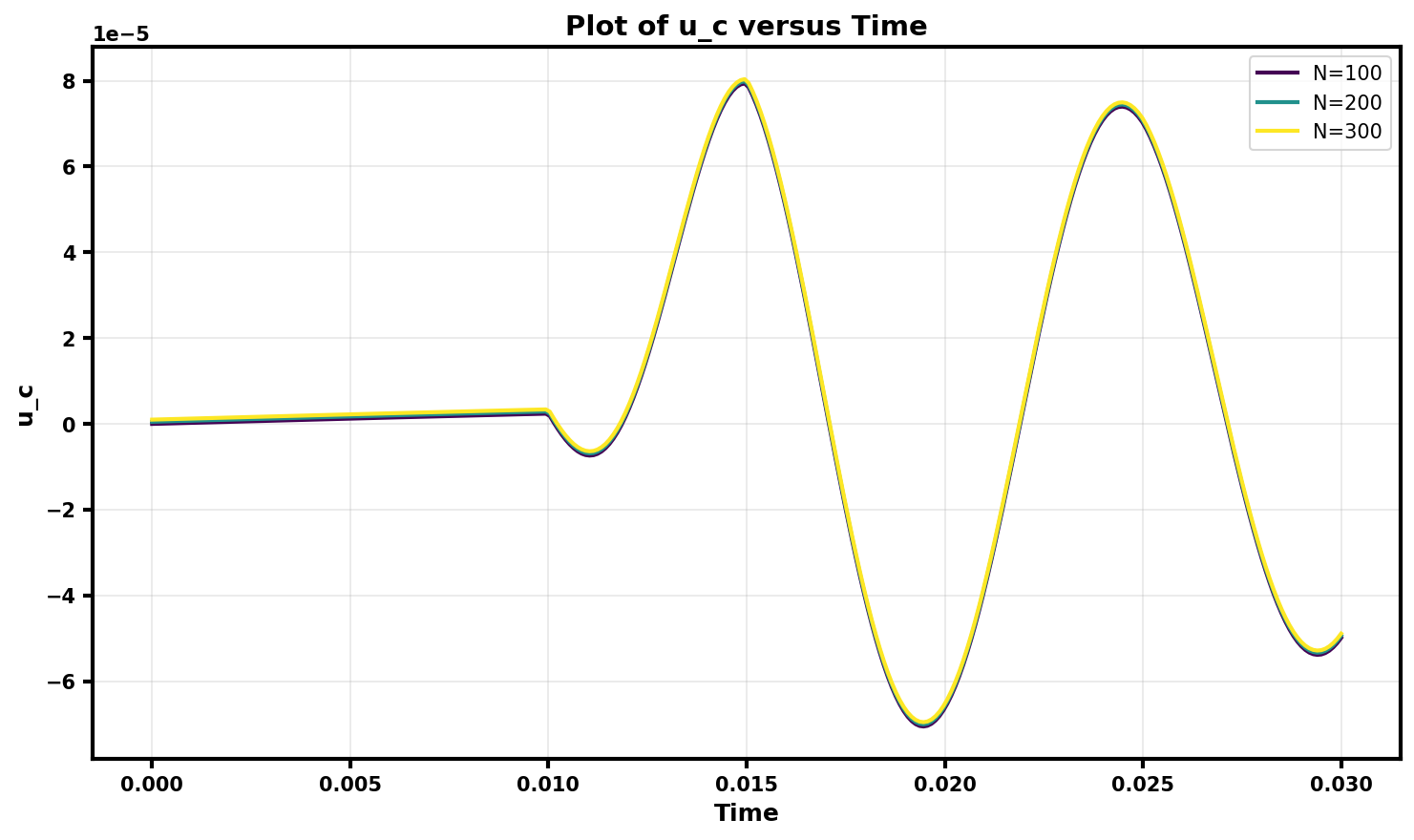}
 %  \hfill
%\begin{subfigure}{0.40\textwidth}	
% \includegraphics[width=8.00cm, height=7.80cm]{uc_plot_new_100.png}
 % 		\caption{}
%\end{subfigure}
\caption{(Example~\ref{exm2_sec}). $u_{c}(t)$ vs $t$ at $50\times 50$ grid for HHO scheme with $k=0$.}
\label{Exm2_fig2}
\end{figure}
\end{example}

\medskip \noindent
{\bf \textit{Data availability}:} Data available on the request from the authors.

\section*{Acknowledgments}
Raman Kumar acknowledges funding by the European Union (ERC Synergy, NEMESIS, project number 101115663). Views and opinions expressed are, however, those of the authors only and do not necessarily reflect those of the European Union or the European Research Council Executive Agency. Neither the European Union nor the granting authority can be held responsible for them. Neela Nataraj and Aamir Yousuf acknowledges the support of J.C. Bose grant ANRF/JBG/2025/000209/HAA.
\bibliographystyle{siam}
\bibliography{Hhowave}
\end{document}